\documentclass[11pt, reqno]{amsart}
\usepackage{amsmath}
\usepackage{amsthm}
\usepackage{amssymb}
\usepackage{amscd,mathrsfs}
\usepackage{amsbsy}
\usepackage{color}
\usepackage{epsfig}
\usepackage{graphicx}
\usepackage{psfrag}
\usepackage[normalem]{ulem}
\usepackage{bbm}
\usepackage{hyperref}
\usepackage{changebar}
\usepackage{comment}

\usepackage[marginpar=.8 in]{geometry}

\usepackage{tikz}
\usetikzlibrary{arrows}
\usetikzlibrary{decorations.pathreplacing}

\theoremstyle{plain}
\newtheorem{theorem}{Theorem}[section]
\newtheorem{remark}[theorem]{Remark}
\newtheorem{proposition}[theorem]{Proposition}
\newtheorem{lemma}[theorem]{Lemma}
\newtheorem{corollary}[theorem]{Corollary}

\newtheorem{definition}[theorem]{Definition}

\newtheorem*{lemma*}{Lemma}

\theoremstyle{definition}

\makeatletter
\newcommand{\labeltext}[2]{%
  \@bsphack
  \csname phantomsection\endcsname 
  \def\@currentlabel{#1}{\label{#2}}%
  \@esphack
}
\makeatother

\newcommand{\strike}[1]{\textcolor{red}{\sout{#1}}}

\newfont\bbf{msbm10 at 12pt}
\def\eps{\varepsilon}
\def\R{{\mathbb R}}

\def\N{{\mathbb N}}

\def\B{{\mathcal B}}

\def\P{{\mathcal P}}
\def\F{{\mathcal F}}
\def\G{{\mathcal G}}
\def\D{{\mathcal D}}

\def\L{{\mathcal L}}
\def\Q{{\mathcal Q}}
\def\cont{{\mathcal C}}

\def\cR{{\mathcal R}}
\def\es{{\emptyset}}
\def\sm{\setminus}

\def\lev{\textnormal{level}}

\def\tag#1{\hfill \qquad  #1}

\def\diam{\mbox{\rm diam} }

\def\crit{{\mathcal C\;\!\!r}}

\def\bd{\partial }
\def\le{\leqslant}
\def\ge{\geqslant}

\newcommand{\Crit}{\textnormal{Crit}}

\newcommand{\hI}{\mathring{I}}
\newcommand{\f}{\mathring{f}}
\newcommand{\hDelta}{\mathring{\Delta}}
\newcommand{\hPi}{\mathring{\Pi}}
\newcommand{\hPhi}{\mathring{\Phi}}

\newcommand{\Lp}{\mathcal{L}}

\newcommand{\vf}{\varphi}
\newcommand{\ve}{\varepsilon}

\newcommand{\tpsi}{\tilde{\psi}}
\newcommand{\hLp}{\mathring{\Lp}}

\newcommand{\hF}{\mathring{F}}
\newcommand{\hf}{\mathring{f}}
\newcommand{\hY}{\mathring{Y}}
\newcommand{\hX}{\mathring{X}}
\newcommand{\he}{\mathring{e}}
\newcommand{\hg}{\mathring{g}}
\newcommand{\bhg}{\mathring{G}}
\newcommand{\hcR}{\mathring{\cR}}
\newcommand{\hm}{\hat{m}}

\newcommand{\dlj}{\Delta_{\ell,j}}

\newcommand{\beq}{\begin{equation}}
\newcommand{\eeq}{\end{equation}}
\newcommand{\hatf}{\hat{f}}
\newcommand{\hatI}{\hat{I}}

\newcommand{\hatpi}{\hat{\pi}}
\def\lmin{\lambda_{\text{min}}}
\def\lmax{\lambda_{\text{max}}}
\def\order{d}

\def\M{\mathcal{M}}

\newcommand{\vertiii}[1]{{\left\vert\kern-0.25ex\left\vert\kern-0.25ex\left\vert #1 
    \right\vert\kern-0.25ex\right\vert\kern-0.25ex\right\vert}}
\newcommand{\invertiii}[1]{{\vert\kern-0.25ex\vert\kern-0.25ex\vert #1 
    \vert\kern-0.25ex\vert\kern-0.25ex\vert}}

\numberwithin{equation}{section}

\begin{document}

\title{Asymptotic escape rates and limiting distributions for multimodal maps}
\author[M.F. Demers]{Mark F. Demers}
\address{Mark F. Demers\\ Department of Mathematics \\
Fairfield University\\
Fairfield, CT 06824 \\
USA}\email{\href{mailto:mdemers@fairfield.edu}{mdemers@fairfield.edu}}
\urladdr{\url{http://faculty.fairfield.edu/mdemers}}

\author[M. Todd]{Mike Todd}
\address{Mike Todd\\ Mathematical Institute\\
University of St Andrews\\
North Haugh\\
St Andrews\\
KY16 9SS\\
Scotland
} \email{\href{mailto:m.todd@st-andrews.ac.uk}{m.todd@st-andrews.ac.uk}}
\urladdr{\url{http://www.mcs.st-and.ac.uk/~miket/}}

\date{\today}

\begin{abstract} 
We consider multimodal maps with holes and study the evolution of the open systems with respect to
equilibrium states for both geometric and H\"older potentials.  For small holes, we show that a large
class of initial distributions share the same escape rate and converge to a unique absolutely continuous conditionally 
invariant measure; we also prove a variational principle connecting
the escape rate to the pressure on the survivor set, with no conditions on the placement of the hole.  
Finally, introducing a weak condition on the centre of the hole, we prove scaling limits for the 
escape rate for holes centred at both periodic and nonperiodic points, as the diameter of the hole goes to zero.  
\end{abstract}

\thanks{Part of this work was completed at CIRM, Luminy in 2017 and 2018, at ICMS, Scotland in 2018, and during visits of MT to Fairfield University in 2017, 2018 and 2019.  MD is partially supported by NSF grant DMS 1800321.  }

\maketitle

\section{Introduction}
\label{sec:intro}

Dynamical systems with holes arise naturally in the study of systems whose domain is not
invariant under the dynamics.  They have been studied in connection with absorbing states in Markov chains
\cite{vere, FKMP},
metastable states in deterministic systems
\cite{DolWri, BahVai, GHW} and neighbourhoods of nonattracting invariant sets \cite{young dev}, as well as in components
of large systems of interacting components in non-equilibrium statistical mechanics \cite{DGKK}.

In the present paper, for a class of  multimodal maps with holes in the form of intervals, we study the escape 
rates and limiting behaviours of
the open systems
with respect to equilibrium states and conformal measures for broad classes of potentials.  The systems in question
have exponential rates of escape,\footnote{For systems with subexponential rates of escape, the results are
qualitatively different since there can be no conditionally invariant limiting distribution
\cite{DemFer}.  See \cite{DetGeo, FMS, APT, DetRah, DemTodMP, BDT} for examples of studies in the subexponential regime.}   in which the 
escape rate and limiting behaviour of the open system is expressed through the existence and
properties of a physical conditionally invariant measure, absolutely continuous with respect to
a given conformal measure. In this setting, given a map $f : I \circlearrowleft$ and identifying a set $H \subset I$ as a hole, one defines the open system by $\f = f|_{\hI^1}$, where 
$\hI^1 = (I\setminus H) \cap f^{-1}(I\setminus H)$.
A {\em conditionally invariant measure} $\mu$ is a Borel probability measure satisfying,
\[
\mu(A) = \frac{\hf_*\mu(A)}{\hf_*\mu(I)}
\quad \mbox{for all Borel $A \subset I$.}
\]
The evolution of measures in the open system is described by  the sequence
$\hf^n_*\mu_0/\hf^n_*\mu_0(I)$ for initial distributions $\mu_0$.  If the limit of such a sequence
exists and is independent of $\mu_0$ for a reasonable class of initial distributions, we call the
resulting measure a limiting (or physical) conditionally invariant measure.
For open systems with exponential rates of escape, the typical agenda of strong dynamical
properties includes a common
 rate of escape for natural
classes of densities, the convergence of such densities to a limiting
conditionally invariant measure under iteration of the dynamics,
and a variational principle connecting the escape rate to the pressure of the open system on the
survivor set, the (singular) set of points which never enters the hole.

Such results have been obtained primarily for uniformly
hyperbolic systems, beginning with expanding maps \cite{PiaYor, CollMar, LivMau},
Anosov diffeomorphisms \cite{CheMar, CheMarTrou}, finite \cite{FerPol}
and countable \cite{Du} state topological Markov chains, and dispersing billiards
\cite{DWY, Dembill}.  Their extension to 
nonuniformly hyperbolic systems has been primarily restricted to uni- and multi-modal
interval maps \cite{BDM, DemTod17, PolUrb} and intermittent maps \cite{DemTodMP}.

The purpose of the present paper is to prove strong hyperbolic properties for 
open systems associated with multimodal maps in greater generality and removing many of the technical 
assumptions made in previous works.  As such, the present paper represents
a significant simplification and extension of results available in the context of nonuniformly
hyperbolic open systems.  Previous works in the setting of unimodal maps with holes
have required strong conditions both on the map (Misiurewicz maps in \cite{Demlog}; 
a Benedicks-Carleson condition in \cite{BDM, DemTod17}; a topologically tame
condition in \cite{PolUrb}), and on the placement of the hole (slow approach to (see \cite{BDM, DemTod17}),
or complete avoidance of (see \cite{PolUrb}),  the
post critical set by the boundary or centre of the hole).

The principal innovation we introduce to the study of open systems in this paper is the use of 
Hofbauer extensions, a type
of Markov extension of the original system. 
Introduced in \cite{Hof86}, they have been used extensively in the study of interval maps.  However,
to date, they have not been implemented for systems with holes.  
In this paper we construct
Hofbauer extensions of our open system, with additional cuts added to our partition depending
on the boundary and centre of our hole.  Doing so enables us to consider the lift of the
hole as a union of 1-cylinders in the extension.  Leveraging recent estimates on
complexity from \cite{DobTod15}, we proceed to build an induced map and related Young tower 
over the Hofbauer extension in order to apply the framework developed in
\cite{DemTodMP} for Young towers with holes.

This two-step approach (rather than simply constructing a Young tower for the open system directly)
allows us to remove many of the technical assumptions needed in
previous works for interval maps with holes, as described above.  Indeed,
we establish the standard suite of strong hyperbolic properties for the open system assuming only that the
hole is a finite union of small intervals (Theorem~\ref{thm:accim}),
entirely eliminating the need for 
previous assumptions on its placement 
or on the orbits of its boundary points.  We also prove the scaling limits
for the escape rate as the 
hole shrinks to a point under much weaker assumptions than used previously 
(Theorems~\ref{thm:zerohole_Holder} and \ref{thm:zerohole_geom}).  
In addition, we greatly broaden the class of potentials
we are able to treat in this setting:  we treat all H\"older continuous potentials, as well as 
the geometric potentials $\phi = -t\log|Df|$ for an interval of $t$ containing $[0,1)$;  if the map satisfies
a Collet-Eckmann condition, we treat $t \ge 1$ as well.  This is in 
contrast to \cite{DemTod17} which restricted $t$ to a small interval around 1, and
\cite{PolUrb} which treated only H\"older potentials with bounded variation.

The paper is organised as follows.  In Section~\ref{sec:setup} we define the class of
maps and potentials we shall study, and recall important definitions regarding pressure and open systems.
In Section~\ref{sec:results} we state our main results, and in Section~\ref{sec:construct}
we carry out our main construction of the Hofbauer extensions and associated induced maps, proving 
that they enjoy tail bounds and mixing properties that are uniform in the size of the hole.  
In Section~\ref{sec:gap} we prove the 
key spectral properties for the induced open system, which are then leveraged in 
Section~\ref{sec:Young tower} for Young towers, and in Section~\ref{sec:zerohole} to establish
the small hole asymptotic.


\section{Setup}
\label{sec:setup}

\subsection{Dynamics}

For $I$ denoting the unit interval, let $\F$ denote the class of $C^3$ maps  $f:I\to I$ with 

\begin{itemize}
\item all critical points non-flat: there exists a finite set $\Crit \subset I$ such that  for each $c\in \Crit$ there is a $C^3$ diffeomorphism $\phi$ in a neighbourhood of $c$ with $\phi(c) =0$ such that $f(x) = \pm|\phi(x)|^\order+ f (c)$ for some $\order>1$, the \emph{order} of $c$;

\item negative Schwarzian derivative, i.e., $\frac{D^3f}{Df}-\frac32\left(\frac{D^2f}{Df}\right)^2\le 0$;

\item the \emph{locally eventually onto (leo)}/topologically exact condition: for any open set $U\subset I$ there exists $n\in \N$ such that $f^n(U)=I$, a form of topological transitivity;

\item for each $c$, 
$$|Df^n(f(c))|\to \infty.$$
\end{itemize}

Note that it is possible to weaken the conditions listed here, but this would lead to a significantly 
more complex exposition.

Sometimes we will require a stronger condition:
we say that $f$ satisfies the \emph{Collet-Eckmann} condition if there exist $C, \gamma>0$ such that for each $c\in \Crit$, and all $n\in \N$,
\begin{equation}|Df^n(f(c))|\ge Ce^{\gamma n} \, . \tag{CE}
\end{equation}

\subsection{Potentials, pressure and equilibrium states}

Given $f\in \F$ we let $\M$ denote the set of $f$-invariant probability measures.  Then for a potential $\phi:\to [-\infty, \infty]$, we define the \emph{pressure} by
\begin{equation}
P(\phi) := \sup \left\{ h_\mu(f) + \int \phi \, d\mu: \mu\in\M \text{ and } \mu(-\phi) < \infty\right\}. \label{eq:pressure}
\end{equation}

A measure $\mu\in \M$ is called an \emph{equilibrium state} for $\phi$ if $h_\mu(f) + \int \phi \, d\mu= P(\phi)$.

Given $\phi:I\to [-\infty, \infty]$, we say that a sigma-finite measure $m_\phi$ is \emph{$\phi$-conformal} if whenever $U$ is a Borel set and $f:U\to f(U)$ is a bijection then 
$$m_\phi(f(U))=\int_Ue^{-\phi}~dm_\phi.$$
(For example, Lebesgue measure is $-\log|Df|$-conformal.)
Notice that we can iterate this relation:  if $f^n:U\to f(U)$ is a bijection, then 
\begin{equation}
m_\phi(f^n(U))=\int_Ue^{-S_n\phi}~dm_\phi ,
\label{eq:conf_meas}
\end{equation}
where $S_n\phi = \sum_{i=0}^{n-1} \phi \circ f^i$. 
We will also be interested in functions $\psi:I\to [-\infty, \infty]$ \emph{cohomologous} to $\phi$; namely, there exists a function $h$ such that $\phi=\psi+h-h\circ f$.  These functions share equilibrium states, though they may produce different, but equivalent, conformal measures. 

We will consider equilibrium states for two types of potentials:  H\"older continuous
potentials and geometric potentials.

{\em (i) H\"older continuous potentials.}   In \cite{LiRiv14} it was shown that any H\"older potential $\phi$ is cohomologous to a H\"older potential $\tilde \phi$ with $\tilde \phi<P(\tilde\phi)$ on $I$ (note that there can be many such potentials).  
It is therefore no loss of generality to assume, as we will throughout, that for our H\"older potentials, $\phi<P(\phi)$.

{\em (ii) Geometric potentials.}  We set $\phi=-\log|Df|$ and consider the family 
$\{t\phi\}_{t\in \R}$.   We let $p_t:=P(t\phi)$ and denote $m_t=m_{t\phi-p_t}$ if this measure exists.
For a $p$-periodic point $x$, define its Lyapunov exponent by  $\lambda(x):=\frac1p\log|Df^p(x)|$.  
As in \cite[Appendix A]{PR}, for $f\in \F$ and $x\in I$, it is always the case that $\lambda(x)>0$. 
Then define
$$\lmin=\inf\{\lambda(x): x \text{ is periodic}\} \text{ and } \lmax:=\sup\{\lambda(x): x \text{ is periodic}\}$$
For $\mu\in\M$, let its Lyapunov exponent be defined by $\lambda(\mu):=\int \log|Df|~d\mu$.  
 By  \cite[Proposition 4.7]{PR}, if $f\in \F$ then 
 $$\inf\{\lambda(\mu):\mu\in \M\}=\lmin \text{ and } \sup\{\lambda(\mu):\mu\in \M\}=\lmax.$$
Noting from the definition of pressure that $p_t\ge -t\lmin$, we define
$$t^+:=\sup\{t\in \R: p_t>-t\lmin\} \text{ and } t^-:=\inf\{t\in \R: p_t>-t\lmax\}.$$
These are referred to as the \emph{freezing point} and the \emph{condensation point} of $f$, respectively.  
It is immediate that $t^- < 0$.
For $f \in \F$, there is always an absolutely continuous invariant probability measure, which implies that $p_1=0$ and $t^+\ge 1$.  As in \cite{PR}, (CE) implies $t^+>1$.

\begin{definition}
\label{def:ad}
We shall call a potential $\phi$ {\em admissible} if either: (a) $\phi$ is H\"older continuous and
$\phi < P(\phi)$ on $I$; or (b) $\phi = - t \log |Df|$ with $t \in (t^-, t^+)$.
\end{definition}

For each admissible $\phi$, 
\begin{equation}
P(\phi) = \sup \left\{ h_\mu(f) + \int \phi \, d\mu: \mu\in\M, \  \mu(-\phi) < \infty \text{ and } h_\mu(f)>0\right\}, \label{eq:pressureLE}
\end{equation}
and there is a unique equilibrium state which is exponentially mixing: for geometric potentials with $t \in (t^-, t^+)$, this follows for example by \cite[Theorem A]{IomTod10}; in the H\"older case this follows from \cite[Theorem A]{LiRiv14}.  Moreover, each equilibrium state is absolutely continuous with respect to a unique conformal measure, which is shown to exist in, for example, \cite[Appendix B]{IomTod13}.
Throughout, we will denote the normalised potential by $\vf = \phi - P(\phi)$, and say that $\vf$ is admissible whenever $\phi$ is.  Moreover, we let $m_\vf$ and $\mu_\vf$ denote the $\vf$-conformal measure and the 
 equilibrium state, respectively.  We may drop the $\vf$ when the potential is clear.

\subsection{Puncturing the system}
\label{sec:puncture}

Choose $z \in I$, and let $H_\ve = (z-\ve, z+\ve) \subset I$ be an interval.  
Denote by $\hI = I \setminus H_\ve$, and in general by $\hI^n = \cap_{i=0}^n f^{-i} \hI$, the set of points
that do not enter $H_\ve$ in the first $n$ iterates.  The sequence of maps
$\hf^n := f^n|_{\hI^n}$ defines the corresponding open system.

We define the upper and lower escape rates through $H_\ve$ by
\[
\log \overline\lambda_\eps:=\limsup_{n\to \infty}\frac1n\log \mu_\vf(\hI^n )
\quad
\mbox{and}
\quad
\log \underline\lambda_\eps:=\liminf_{n\to \infty}\frac1n\log \mu_\vf(\hI^n )
\]
When the two quantities coincide, we denote them by $\log \lambda_\ve$, and call $- \log \lambda_\ve$ the {\em escape rate}
with respect to $\mu_\vf$.

Given a potential $\phi$, once a hole $H_\ve$ is introduced, the \emph{punctured potential} is defined by $\phi^{H_\ve} = \phi$ on $\hI$ and 
$\phi^{H_\ve} = - \infty$ on $H_\ve$.  $P(\phi^{H_\ve})$ denotes the pressure of the punctured potential, and it follows from
the requirement $\mu(-\phi^{H_\ve}) < \infty$ that the supremum for this pressure is restricted to $f$-invariant measures
that are supported on the {\em survivor set} $\hI^\infty := \cap_{n=0}^\infty \hI^n$.

We will be interested in establishing convergence for limits of the form,
$\hf^n_*\mu/\hf^n_*\mu(I)$ for measures $\mu$ which are absolutely continuous with respect
to the conformal measure $m_\vf$.  To this end, define the transfer operator corresponding to the
potential $\vf$ by,
\[
\Lp_\vf \psi(x) = \sum_{y \in f^{-1}x} \psi(y) e^{\vf(y)}, \quad \mbox{for $\psi \in L^1(m_\vf)$.}
\]
Similarly, the punctured transfer operator for the open system
is defined by
\[ 
\hLp_{\vf^{H_\ve}} \psi(x) = \Lp_\vf(1_{\hI^1} \psi)(x)
= \sum_{y \in \hf^{-1}x} \psi(y) e^{\vf(y)} \, .
\]
Due to the conformality of $m_\vf$, we have
\[
\int_I \hLp_{\vf^{H_\ve}}^n \psi \, dm_\vf = \int_{\hI^n} \psi \, dm_\vf, \quad \mbox{for all $n \geq 1$},
\]
which relates the escape rate with respect to the measure $\psi \, dm_\vf$ to the spectral radius
of $\hLp_{\vf^{H_\ve}}$.


\section{Results}
\label{sec:results}

\subsection{Small hole, general placement}

Theorem~\ref{thm:accim} proves the standard suite of strong hyperbolic properties for the open system. 
As noted in the introduction, it is a significant improvement over \cite{BDM}, \cite{DemTod17} 
and \cite{PolUrb} which had similar results under much more
restrictive assumptions on the map, the potential and the hole.

\begin{theorem}
\label{thm:accim}
 Let $f\in \F$ and $\phi$ be an admissible potential, with normalised version $\vf = \phi - P(\phi)$.
Let $z \in I$, and for $\ve>0$, set $H_\ve(z) = (z-\ve, z+\ve)$. 
Suppose that $\ve^*>0$ is sufficiently 
small so that 
 $-\log\overline\lambda_{\ve^*}< \alpha$, where $\alpha>0$ is the tail decay rate
 from Theorem~\ref{thm:tails}.
 Then the following hold for all $0 < \ve \le \min \{ \ve_1^*, \ve^* \}$, where
 $\ve_1^*$ is from Lemma~\ref{lem:mixing D}.
\begin{enumerate}
\item[(a)] The escape rate $- \log \lambda_\ve$ exists, and $\lambda_\ve<1$ is the spectral radius of the punctured transfer operator 
on the associated Young tower.
The associated eigenvector projects to a nonnegative function $\hg_\ve$, which is bounded
away from zero
on $I \setminus H_\ve$ and satisfies $\hLp_{\vf^{H_\ve}} \hg_\ve = \lambda_\ve \hg_\ve$.
\item[(b)] There is a unique $(\phi^{H_\ve}-P(\phi^{H_\ve}))$-conformal measure $m_{H_\ve}$. This is singular with respect to
$m_\vf$ and supported on $\hI^\infty$.
\item[(c)] The measure $\nu_{H_\ve} :=\hg_\ve m_{H_\ve}$ is the unique
equilibrium state for $\phi^{H_\ve} - P(\phi^{H_\ve})$; in particular,
\[
\log \lambda_\ve = P(\phi^{H_\ve}) - P(\phi) = P(\vf^{H_\ve}) 
= h_{\nu_{H_\ve}}(f) + \int \vf^{H_\ve} \, d\nu_{H_\ve}.
\]
Moreover, $\nu_{H_\ve}$ is supported on $\hI^\infty$ and can be realised as the limit,
 \[
 \nu_{H_\ve}(\psi) = \lim_{n\to \infty} \lambda_\ve^{-n} \int_{\hI^n} \psi \hg_\ve \, dm_\vf, \qquad
 \mbox{for all $\psi \in C^0(I)$.}
 \]
 \item[(d)] The measure $\mu_\vf^{H_\ve} := \hg_\ve m_\vf$ is a conditionally invariant measure supported on 
 $I \setminus H_\ve$ with eigenvalue 
$\lambda_\ve$
and is a limiting distribution in the following sense.  Fix $\varsigma >0$ and
let $\psi \in C^\varsigma(I)$ satisfy $\psi \ge 0$, with $\nu_{H_\ve}(\psi) >0$.  
Then
 \begin{equation}
 \label{eq:conv to accim}
\left| \frac{\hLp^n_{\vf^{H_\ve}} \psi}{|\hLp^n_{\vf^{H_\ve}} \psi|_{L^1(m_\vf)}} - \hg_\ve \right|_{L^1(m_\vf)} \le C \vartheta^n |\psi|_{C^\varsigma(I)} 
\end{equation}
for some $C>0$ independent of $\psi$, and $\vartheta<1$ depending only on $\varsigma$. 
\end{enumerate}
\end{theorem}

The techniques also imply that $\nu_{H_\ve}, \mu_\vf^{H_\ve}  \to \mu_\phi$ as $\ve\to 0$.
Note that the techniques of the proof also extend to holes comprised of finitely many intervals
as the only condition required
on the hole in \cite{DemTodMP} is $-\log \overline\lambda_\ve < \alpha$.  We prove this theorem in Section~\ref{accim pf}. 

\begin{remark}
\label{rmk:wider conv}
In fact, we prove convergence to the conditionally invariant measure $\mu_\vf^{H_\ve}$ for a larger class of initial densities than $C^\varsigma(I)$.  It only matters that
$\psi$ satisfies $\nu_{H_\ve}(\psi)>0$ and that it can be realised as the projection of an element in a certain function space on the
related Young tower.  So for example, any function of the form $\psi = \tilde \psi g_\vf$ also satisfies \eqref{eq:conv to accim}, where
$\tilde \psi \in C^\varsigma(I)$ and $g_\vf = \frac{d\mu_\vf}{dm_\vf}$.
\end{remark}

The following lemma shows that we can always choose $\ve^*>0$ small enough so that $-\log\overline\lambda_{\ve^*}< \alpha$, and hence the theorem applies to all small holes.

\begin{lemma}
\label{lem:lambda cont}
Suppose $\phi$ is an admissible potential and $f \in \mathcal{F}$.  For any $z \in I$, and
hole $H_\ve(z) = (z-\ve, z+\ve)$, it holds that $\lim_{\ve \to 0} \overline\lambda_\ve = 1$.
\end{lemma}
\begin{proof}
This is a simple consequence of Corollary~\ref{cor:spectral decomp},  since the escape
rate for the related induced system, $-\log \Lambda_\ve$, is continuous in $\ve$, and by
monotonicity, $\overline\lambda_\ve \ge \Lambda_\ve$.  For details, see the verification of
property (P2) in Section~\ref{tower transfer}.
\end{proof}

\begin{remark}
\label{rmk:hausdorff}
(Bowen formula for Hausdorff dimension of $\hI^\infty$.)  If we take $\phi = - \log |Df|$, then under the assumptions of Theorem~\ref{thm:accim}
and for $\ve$ sufficiently small,
$\mbox{\textnormal{Hdim}}(\hI^\infty) = t^*$,
where $t^*$ is the unique value of $t$ such that $P(t\phi^{H_\ve}) = 0$.  This follows as in
\cite[Theorem~8.1]{DemTod17}, using the uniform bounds for $t$ close to 1 on the tail of the return time function from Theorem~\ref{thm:tails} to show that
any set of Hausdorff dimension greater than some constant $D<1$ lifts to our inducing scheme.  Then Theorem~\ref{thm:accim}(c)
implies that the dimension of the equilibrium measure $\nu^1_{H_\ve}$, corresponding to $t=1$, equals $1 + \log \lambda_\ve$,
and so  is greater than $D$ for $\ve$ small.   Thus the Hausdorff dimension of
$\hI^\infty$ equals that of the survivor set in our inducing scheme.
\end{remark}


\subsection{Zero-hole limits}

Here we consider the asymptotic scaling limit for the escape rate,
$\frac{-\log\lambda_\eps}{\mu_\vf(H_\eps)}$ as $\ve\to 0$.  
This limit was first computed in the context of escape rates for full shifts
in \cite{BuniYur}, then extended to (piecewise) uniformly expanding systems in \cite{KelLiv09} and to more
general potentials in the symbolic setting in \cite{FerPol} (see also \cite{AfBu, BahVai13, FreFreTod15}).
Its extension to unimodal and multimodal maps followed with added assumptions on the centre of the hole $z$,
 either assuming that the post-critical orbits
approach $z$ slowly \cite{BDM, DemTod17}, or are bounded away from $z$ \cite{PolUrb}.

By contrast, for H\"older continuous potentials, we prove our results for all nonperiodic $z \in I$,
with an additional assumption required only if $z$ is periodic and lies in the post critical orbit.  
For geometric potentials, we require a (generic) 
slow approach condition to $z$, and present an example (Section~\ref{ex}) to show that the scaling limit
can fail for geometric potentials if no condition on $z$ is imposed.
The proofs of Theorems~\ref{thm:zerohole_Holder} and \ref{thm:zerohole_geom} are in Section~\ref{sec:zerohole}.

\subsubsection{H\"older potentials}

The asymptotic escape rate depends on whether the chosen centre $z$ is periodic 
or not.

\begin{theorem}
\label{thm:zerohole_Holder}
Let $f \in \mathcal{F}$, $\phi$ be H\"older continuous and $z \in I$.  

\begin{itemize}
  \item[a)]  If $z$ is not periodic, then $\displaystyle \lim_{\ve \to 0} \frac{-\log \lambda_\ve}{\mu_\vf(H_\ve)} = 1$.
  \item[b)]  If $z$ is periodic with prime period $p$ and $\{ f^n(c) : c \in \Crit, n \ge 1 \} \cap \{ z \} = \emptyset$, then  
    $\displaystyle \lim_{\ve \to 0} \frac{-\log \lambda_\ve}{\mu_\vf(H_\ve)} = 1 - e^{S_p\vf(z)}$.
  \item[c)]  Suppose $z$ is periodic with prime period $p$ and $\{ f^n(c) : c \in \Crit, n \ge 1 \} \cap \{ z \} \neq \emptyset$.   If in addition,
  either $f^p$ is orientation preserving in a neighbourhood of $z$, or $\lim_{\ve\to 0} \frac{m_\vf(z+\ve, z)}{m_\vf(z, z-\ve)} = 1$, then
  $\displaystyle \lim_{\ve \to 0} \frac{-\log \lambda_\ve}{\mu_\vf(H_\ve)} = 1 - e^{S_p\vf(z)}$.  
  \end{itemize}
\end{theorem}

\begin{remark}
\label{rmk:sub}
Even when both conditions fail in part (c) of Theorem~\ref{thm:zerohole_Holder} fail, we can still find a subsequence of $\ve \to 0$
so that the scaling limit converges to $1 - e^{S_p\vf(z)}$.  Thus we expect that the scaling limit holds for all periodic points in
the case of H\"older continuous potentials.
\end{remark}


\subsubsection{Geometric potentials}
\label{sssec:geo}

For the remainder of this section we let $\phi=-\log|Df|$. 
The geometric case requires a condition on slow approach to the critical set as well as a polynomial rate of growth of the derivative
along the post-critical orbit.  For simplicity, for a given $\order> 1$ we will consider the set $\F_\order\subset \F$ with the defining property that for each $f\in \F_\order$ all critical points have order $\order$.

For $t \in (t^-, t^+)$, let $s_t := t + \frac{p_t}{\lambda(\mu_t)} \in (0,1]$ denote the local scaling exponent
for $m_{t\phi - p_t}$, see \cite[Lemma~9.5]{DemTod17}.
Define
\[
D_n(c) = |Df^n(f(c))| \quad \mbox{for each } c \in \Crit \, .
\]
We assume that for each $c \in \Crit$, 
\begin{equation}
\label{eq:Dn}
D_n(c) \ge \mbox{const.} n^q \quad \mbox{for some $q > \order + \frac{\order - 1}{s_t}$ and all $n \ge 1$.}
\end{equation}

With $q$ given as above, we choose $\theta \in (0,1)$ and $r \in \left(\frac{1}{(1-\theta)s_t}, \frac{q-\order}{\order -1} \right)$, 
and define a sequence $\gamma_n = n^{-r}$, $n \ge 1$.
We make the following assumption on the 
centre of the hole, $z$, in terms of this sequence:
\begin{equation}
\label{eq:slow approach}
\exists \delta_z > 0 \mbox{ s.t. } \min_{c \in \Crit} d(f^n(c), z) \ge \delta_z \gamma_n^{1-\theta}, \quad \mbox{for all $n \in \mathbb{N}$.}
\end{equation}
In particular, we have 
$\sum_n \gamma_n^{(1-\theta)(s_t - \epsilon)} < \infty$ for some $\epsilon>0$, 
so that condition \eqref{eq:slow approach} is generic with respect to the measures $m_\vf$,
$\vf = t\phi - p_t$, as proved in  \cite[Lemma~9.3]{DemTod17}. 

The value of $s_t$ varies continuously with $t$, and is $>0$ for each $t\in (t^-, t^+)$, with $s_1=1$, but may tend to zero as $t$ tends to the boundary of $(t^-, t^+)$.  This means that in particular when the map $f$ satisfies the (CE) condition, 
we will restrict to a subinterval $(t^-, t_1)$ where $t_1 \in (1, t^+]$ is determined by \eqref{eq:t1};
if $f$ does not satisfy (CE), we let $t_1=1$.

\begin{theorem}
For $\order > 1$, let $f\in \F_\order$ and $t \in (t^-, t_1)$.  Suppose
\eqref{eq:Dn} is satisfied and  $z\in I$ satisfies
\eqref{eq:slow approach}.  
Then for  $\vf=-t\log|Df|-p_t$,
\begin{enumerate}
\item[(a)] if $z$ is not periodic then 
$\displaystyle \lim_{\eps\to 0} \frac{-\log\lambda_\eps}{\mu_\vf(H_\eps)} =1$;
\item[(b)] if $z$ is periodic with (prime) period $p$, then  
$\displaystyle \lim_{\eps\to 0} \frac{-\log \lambda_\eps}{\mu_\vf(H_\eps)} =1-e^{S_p\vf(z)}$. 
\end{enumerate}
\label{thm:zerohole_geom}
\end{theorem}

\begin{remark}
\cite[Section 6]{FreFreTod13} shows that there are examples of maps $f\in \F_\order$ and periodic points $z$ satisfying \eqref{eq:slow approach}.
\end{remark}

\begin{remark}
It is not clear what the optimal condition on $z$ is so that the scaling limits of 
Theorem~\ref{thm:zerohole_geom} hold, but it is clear that the limits can fail without some
assumption on $z$ in the case of geometric potentials.  To illustrate this point, we present an
example in Section~\ref{ex} using the map $f(x) = 4x(1-x)$ for which \eqref{eq:slow approach}
does not hold, and the relevant scaling limit fails. 
\end{remark}


\subsection{Escape rate function}

The asymptotics in the previous subsection can be seen as a type of derivative of the escape rate
at $\ve = 0$.  Our next result addresses the regularity of the escape rate $- \log \lambda_\ve$
from Theorem~\ref{thm:accim} for $\ve > 0$.

\begin{theorem}
\label{thm:devil}
Let $f \in \F$ and $\phi$ be an admissible potential.  Suppose $z \in I$ and 
let $\ve^* > 0$ be from Theorem~\ref{thm:accim}.
Then $\ve \mapsto - \log \lambda_\ve$ is continuous on $[0, \ve^*]$ and
forms a devil's staircase:  i.e., $\frac{d\log\lambda_\ve}{d\ve}$ exists and equals 0 on an 
open and full measure subset of $[0, \ve^*]$.
\end{theorem}

That the escape rate function forms a devil's staircase has been shown in 
uniformly hyperbolic settings, namely for expanding systems in \cite{KelLiv09}, and for
Anosov diffeomorphisms in \cite{DemWri}.  The present result is the first in the setting
of nonuniformly hyperbolic maps.  
It stands in contrast to Theorems~\ref{thm:zerohole_Holder}
and \ref{thm:zerohole_geom}, which prove that $\left. \frac{d\log\lambda_\ve}{d\ve} \right|_{\ve = 0}$ exists
and is nonzero.
Once Theorem~\ref{thm:accim} is established,
it is a direct consequence of the continuity of $f$ and the ergodicity of the 
measure $\mu_\vf$, so we give this short proof immediately.

\begin{proof}[Proof of Theorem~\ref{thm:devil}]
The continuity of $\ve\mapsto - \log \lambda_\ve$ follows from Corollary~\ref{cor:spectral decomp}
and \eqref{eq:abramov}.  We proceed to prove the statement about the derivative of this map.
Denote the survivor set by $\hI^\infty_\ve = \cap_{n=0}^\infty f^{-n}(I \setminus H_\ve(z))$.

If $\hI^\infty_\ve \cap \partial H_\ve = \emptyset$, then dist$(\hI^\infty_\ve, \partial H_\ve) > 0$.
This follows from the continuity of $f$ and the fact that $H_\ve(z)$ is open:  If
$\hI^\infty_\ve \cap \partial H_\ve = \emptyset$ then there exists $n> 0$ such that
$f^n(z+\ve) \in H_\ve(z)$; by the continuity of $f$, there exists a neighbourhood of $z + \ve$,
$N_\delta(z+\ve)$,
such that $f^n(N_\delta(z+\ve)) \subset H_\ve(z)$.  A similar argument holds for $z-\ve$.

Thus if $\hI^\infty_\ve \cap \partial H_\ve = \emptyset$, then 
$\hI^\infty_{\ve'} \cap \partial H_{\ve'} = \emptyset$ for all $\ve' \in (\ve-\delta', \ve+\delta')$ for
some $\delta' >0$, i.e.
the fact that the boundary of the hole falls into the hole is an open condition.
It follows from this that $\hI^\infty_\ve = \hI^\infty_{\ve'}$ for all $\ve' \in (\ve-\delta', \ve+\delta')$, and
thus that $P(\vf^{H_\ve}) = P(\vf^{H_{\ve'}})$ and by Theorem~\ref{thm:accim}(c),
$\lambda_\ve = \lambda_{\ve'}$ for all $\ve' \in (\ve-\delta', \ve+ \delta')$.

Thus $\log\lambda_\ve$ is locally constant whenever $\hI^\infty_\ve \cap \partial H_\ve = \emptyset$.

Finally, since $\mu_\vf(H_\ve)>0$, ergodicity implies that generic $z \pm \ve$ fall in the hole, so the condition $\hI^\infty_\ve \cap \partial H_\ve = \emptyset$
is generic.  Therefore, 
$$\mu_\vf\left\{x=z+\ve\in I: \ve\in (0, \eps^*) \text{ and } \frac{d\log\lambda_\ve}{d\ve} \neq 0\right\}= 0,$$
as required. 
\end{proof}


\subsection{An example of scaling limit failure}
\label{ex}

In this section we present an example of a map in our class $\mathcal{F}$
and choice of $z$ such that condition \eqref{eq:slow approach} is violated and 
the conclusion of Theorem~\ref{thm:zerohole_geom} fails.

Let $f: I \circlearrowleft$ be defined by $f(x) = 4x(1-x)$.  Let $X$ also denote the unit interval,
and $T:X \circlearrowleft$ be the tent map $T(x) = 2x$, $x \in [0,1/2]$, and $T(x) = 2(1-x)$,
$x \in [1/2,1]$.

The well-known conjugacy between $f$ and $T$ is $g:X \to I$, $g(x) = \sin^2(\frac{\pi x}{2})$, so that
$f \circ g(x) = g \circ T(x)$ for all $x \in X$.

Let $m$ denote Lebesgue measure on $X$, which is $T$-invariant and the equilibrium
state for the potential $- \log |DT|$.  The absolutely
continuous invariant probability measure for $f$ can then be written as $\mu = g_*m$,
which is the equilibrium state for the potential $- \log |Df|$. 

We choose $z = 0$, a fixed point for $f$, and
define $H_\ve = [0, \ve)$.  
It is clear that \eqref{eq:slow approach} fails, since $\Crit = \{ \frac 12 \}$ and $f^2(\frac 12) = 0$.

Now $g^{-1}(H_\ve) = [0, \ve')$, where $\ve' = \frac{2}{\pi} \sin^{-1}(\sqrt{\ve})$.
Note that since $X = g^{-1}(I)$, we have
\[
\begin{split}
m(\hX^n) & := m(\cap_{i=0}^n T^{-i}(X \setminus g^{-1}(H_\ve)) 
= m(\cap_{i=0}^n T^{-i}(g^{-1}(I \setminus H_\ve))) \\
& = m(\cap_{i=0}^n g^{-1} \circ f^{-i}(I \setminus H_\ve))
= m(g^{-1}(\cap_{i=0}^n f^{-i}(I \setminus H_\ve))) = \mu(\hI^n) \, , 
\end{split}
\]
where $\hX^n$ and $\hI^n$ denote the $n$-step survivor sets for $T$ and $f$, respectively.

Thus the escape rate $-\log \lambda_\ve$ for $(f, \mu, H_\ve)$ is the same as the escape rate for $(T, m, g^{-1}(H_\ve))$.

Now applying \cite[Theorem~4.6.1 and Section~5]{BuniYur} 
(see also \cite[Theorem~2.1 and Section~3.1]{KelLiv09}) to $T$,
we compute the scaling limit,
\[
\lim_{\ve \to 0} \frac{- \log \lambda_\ve}{\mu(H_\ve)} = 
\lim_{\ve \to 0} \frac{- \log \lambda_\ve}{m(g^{-1}(H_\ve))}
= 1 - \frac{1}{DT(0)} = \frac 12 \, .
\]
Yet $Df(0) = 4$, so that the expected scaling limit for $f$
would be $1 - \frac{1}{Df(0)} = \frac 34 \neq \frac 12$.

\begin{remark}
Although the scaling limit of Theorem~\ref{thm:zerohole_geom} fails in this case, we note
that an alternate formulation is possible.  Indeed, the invariant density for $f$
with respect to Lebesgue measure has a spike of order $x^{-1/2}$ at $z=0$.  So the limit of
$\frac 1 2$ that we compute is compatible with the formula,
\[
\lim_{\ve \to 0} \frac{- \log \lambda_\ve}{\mu(H_\ve)} = 1 - \left( \frac{1}{Df(0)} \right)^{1/2}
= 1 - \left( \frac{1}{4} \right)^{1/2} = 1 - \frac 1 2 = \frac 1 2,
\]
where the scaling exponent of $1/2$ matches the exponent in the spike of the invariant density.
Such relations follow from O'Brien's formula for the extremal index (see \cite[(2.6)]{FreFreTod15} for a dynamical setting of this), and given the connection
between extremal indices and scaling limits for escape rates established in 
\cite{BDT}, we conjecture that it holds in greater generality for scaling limits.
\end{remark}


\section{Construction of extensions and preliminary results}
\label{sec:construct}

\subsection{Distortion and contraction}

As is standard in this field we wish to recover some uniform expansion and uniform distortion from a system which is non-uniformly hyperbolic.  
We will use versions of the Koebe Lemma often, so state it here (see \cite[Theorem IV.1.2]{MelStr93}) recalling that elements of $\F$ have negative Schwarzian derivative.

\begin{lemma}[Koebe Lemma]
For any $\epsilon>0$, there exists $K(\epsilon)\ge 1$ such that the following hold.  If $f\in \F$ and $U\Subset U'$ is such that $U'\sm U$ consists of two intervals length $\ge \epsilon|U|$ and $f^n: U'\to f^n(U')$ is a diffeomorphism then,
\begin{enumerate}
\item[(a)]  for $x, y\in U$,
$$\frac{Df^n(x)}{Df^n(y)}\le K(\epsilon);$$
\item[(b)]   for $x, y\in U$,
$$\left|\frac{Df^n(x)}{Df^n(y)}-1\right|\le K(\epsilon)\frac{|x-y|}{|U|}.$$
\end{enumerate}
\label{lem:Koebe}
\end{lemma}

For expansion/backward contraction we use `polynomial shrinking'.  That is, 
   for $\beta>0$,
   
\begin{itemize}
\item (PolShr)$_\beta$: there are constants $\delta, C>0$  such that for each $x\in I$ and every integer $n \ge 1$, any  connected component $W$ of $f^{-n}(B_\delta(x))$ has $|W| \le Cn^{-\beta}$.
\end{itemize}

 Combining \cite[Theorem A]{RivShe14} and \cite[Theorem 1]{BruRivSheStr08},  for each $f\in \F$ this holds for any $\beta>0$.\footnote{In fact, these results imply that to obtain (PolShr)$_\beta$ for a particular $\beta$,
one does not need $|Df^n(f(c))|\to \infty$ for all $c\in \Crit$, but rather a specific lower bound for $|Df^n(f(c))|$ depending on $\beta$ suffices.}
Notice that for intervals of size larger than $\delta$ in our setting, we can simply chop these up into smaller intervals at the cost of adding a multiplicative constant.


\subsection{Hofbauer extensions}

Hofbauer extensions are Markov extensions of $f : I \circlearrowleft$ usually defined by introducing
cuts at (images of) critical points, but in fact we can cut at arbitrary points: in Section~\ref{ssec:puncture} we will give a definition of our `extended critical set'.  So we let $\crit\subset I$ be a finite set of points with $\Crit\subset\crit$.  Set $\P_0:=I$, let $\P_1$ be the partition defined by $\crit$, and define $n$-cylinders by
$$\P_n:=\bigvee_{i=0}^{n-1}f^{-i}\P_1.$$
We will denote the $n$-cylinder which $x\in I$ lies in by $Z_n[x]$ (note that if there are two, then we can make an arbitrary choice).
Now define $\D:=\{f^k(Z):Z\in \P_k, k \ge 0\}$.  
As $\D$ is a set, each element $D\in \D$ appears once (i.e., if $f^k(Z)=f^j(Z')$ then these elements are naturally identified as the same set).  The \emph{Hofbauer extension} is defined as the disjoint union 
$$\hat I=\hat I(\crit):=\sqcup_{D\in \D}D.$$  
We call each $D$ a \emph{domain} of $\hat I$.  There is a natural projection map $\pi:\hat I \to I$, so each point $\hat x\in \hat I$ can be represented as $(x, D)$ where $x=\pi \hat x$.   The map $\hat f:\hat I \circlearrowleft$ is defined by $\hat f(\hat x)=\hat f(x, D)=(f(x), D')$ if there are cylinder sets $Z'\subset Z$ with $Z'\in \P_{k+1}$ and $Z\in \P_k$ such that 
\begin{equation}
x\in f^k(Z')\subset f^k(Z)=D \text{ and } D'=f^{k+1}(Z').\label{eq:Hof_cyl}
\end{equation}
In this case we write $D\to D'$, so $(\D, \to)$ has the structure of a directed graph.  With this setup, $\pi$ acts as a semiconjugacy between $\hat f$ and $f$:
$$\pi\circ \hat f=f\circ \pi.$$

We can think of points in $\crit$ as `cut points' since if an open interval $\hat A=(A, D)\subset \hat I$ and $\#\{A\cap \crit\}=k \ge 1$, then $\hat A$ gets cut at each element of $\crit$ (strictly speaking, of $\pi^{-1}(\crit)$) so that $\hat f(\hat A)$ lies in $k+1$ different elements of $\D$.

Let $D_0$ be the \emph{base}  of $\hat I$, that is the copy of $I$ in the extension.   Define $\iota$ to  be the natural inclusion map sending $I$ to $D_0$.
For $D\in \D$ we let $\lev(D)$ be the length of the shortest path $D_0\to \cdots \to D$ in $(\D, \to)$.  Then for $L\in \N$, the truncated extension at level $L$ is  
$$\hat I(L):=\sqcup\{D\in \D:\lev(D)\le L\}.$$

The following lemma and proof are well-known in the area, but we include them for illustrative purposes and for use later.

\begin{lemma}
Suppose that $\hat x, \hat y\in \hat I\sm\bd\hat I$ have $\pi \hat x=\pi\hat y$.  Then there exists $n\in \N$ such that $\hat f^n(\hat  x) =\hat f^n(\hat y)$.
\label{lem:fibres meet}
\end{lemma}

\begin{proof}
Let $w=\pi \hat x$.  Observe that since $\pi$ is a semiconjugacy, $\hat f^k(\hat x), \hat f^k(\hat y)\in \pi^{-1}(f^k(w))$ for all $k\ge 0$.  Let $D_{\hat x}$ and $D_{\hat y}$ denote the 
domains of $\hat I$ which contain $\hat x$ and $\hat y$ respectively.  Then choose $n$ so large 
that $(\pi|_{D_{\hat x}})^{-1}(Z_n[w])$ and  $(\pi|_{D_{\hat y}})^{-1}(Z_n[w])$ are both compactly contained inside  $D_{\hat x}$ and $D_{\hat y}$ respectively, where
$Z_n[w]$ denotes the element of $\mathcal{P}_n$ containing $w$.  Now notice that $f^n(Z_n[w])$ is a domain of the Hofbauer extension, and indeed it follows from the construction in \eqref{eq:Hof_cyl} that $\hat f^n(\hat x)$ and $\hat f^n(\hat y)$ must lie in $f^n(Z_n[w])$.  Since these iterates must also both lie on the fibre $\pi^{-1}(f^n(w))$ by the conjugacy property, the points must coincide, as required.
\end{proof}

In general, Hofbauer extensions split into a collection of transitive components and a non-transitive set, see \cite{HofRai89}, but the above lemma and the leo property imply that there is a unique transitive component.  Since any points outside this must map into it and stay there forever, we will adopt the convention that $\hat I(L)$ is always restricted to the transitive component.

 Given a set $A\subset I$, the set $\hat A= \pi^{-1}(A)$ is called the \emph{lift} of $A$.
We now consider how to lift measures to $\hat I$.  Suppose that $\mu$ is an ergodic $f$-invariant probability measure.  Set 
$\hat\mu^{(0)}:=\mu\circ \iota^{-1}$, and for $n\in \N$,
$$\hat\mu^{(n)}:=\frac1n\sum_{k=0}^{n-1}\hat\mu^{(0)}\circ \hat f^{-k}.$$
As in \cite{Kel89}, if  $h_\mu(f)>0$, then $\hat\mu^{(n)}$ converges in the vague topology\footnote{Recall that $\hat\mu^{(n)}$ converges to $\hat\mu$ vaguely if $\hat\mu^{(n)}(\psi)$ converges to $\hat\mu(\psi)$ for all continuous $\psi$ with compact support in $\hat I$.} to $\hat\mu$, which is an $\hat f$-invariant ergodic measure with $$\hat\mu\circ\pi^{-1}=\mu.$$
Also, \cite{Kel89} shows that $h_{\hat\mu}(\hat f)= h_\mu(f)$.

We will also be interested in lifting conformal measures.
Given a conformal measure $m_\phi$ on $I$, define $\hat m_\phi:=m_\phi\circ\pi^{-1}$.  
Clearly $\hat m_\phi$ is $\hat\phi$-conformal for $\hat\phi:= \phi\circ\pi$ on $\hat I$.  
Note that in general it could be the case that $\hat m_\phi(\hat I)=\infty$.

\begin{remark}
We can define pressure $P(\hat\phi)$ analogously to \eqref{eq:pressure}.  As in \eqref{eq:pressureLE}, for admissible potentials we need only consider measures with positive entropy, so we deduce that $P(\hat \phi)=P(\phi)$.
This implies that when we lift the normalised potential, $\hat \vf := \vf \circ \pi$, then
the relation $\hat \vf = \hat \phi - P(\hat \phi)$ continues to hold.
\label{rmk:lifted pressure}
\end{remark}


\subsection{Inducing schemes}

We wish to define inducing schemes via first return maps to truncated domains in the Hofbauer extension, whose partition we will refine further below: it will also be useful to set this up for our punctured systems, though there will be a small difference in the structure there.  To this end, let $\hat\P_n$ be the set of intervals $\{(\pi|_D)^{-1}(Z):D\in \D, \ Z\in \P_n\}$.  For a domain $D\in \D$, let $D_\ell^L$ be the left-most interval of $\hat\P_L$ in $D$ and $D_r^L$ be the right-most,
 \begin{equation}
\label{eq:trim}
\hat I'(L):=\hat I(L)\cap \left[\sqcup_{D\in \D}\left(D\sm (D_\ell^L \cup D_r^L)\right)\right].
\end{equation}
It follows, for example from \cite[Lemma 8.2]{DobTod15} that, so long as $\hat I$ has more than one domain, then for all $\epsilon>0$ there exists $L\in \N$ such that if $h_\mu(f)>\epsilon$ then 
$\hat\mu(\hatI'(L))>0$.

We further partition $\hat I'(L)$ into the elements of $\hat\P_L$ intersecting it and denote this collection by $\Q$, (i.e., $\Q=\{Q\in \hat\P_L: Q\subset \hat I'(L)\}$), see Figure~\ref{fig:Hofbasic}.  Letting $R$ be the first return time {to $Y := \hat I'(L)$, 
the map $F=\hat f^R$ is the first return map.  
We denote the domains of $F$ by $\{ Y_i \}_i$.
These are the maximal sets $U$ such that $U\subset Q$ 
and $F(U) \subset Q'$ for some $Q, Q' \in \Q$, so that $F$ is monotonic and $R$ is constant on $U$.
We set $R_i=R|_{Y_i}$.  The cylinder structure of $\Q$ ensures that the $\{Y_i\}_i$ are disjoint 
and the Markov structure ensures that the image of such a domain is an interval $Q$ 
of $\Q$, see \cite[Lemma 4.9]{DobTod15}.  We give a short proof of this fact to explain how the changes we make later will not affect this structure.

\begin{figure}[h]
\begin{tikzpicture}[thick, scale=0.5]


\draw[gray, very thin] (0.6, -1.8) -- (0.6, 10);
\draw (0.6, -2,0.5pt) node[below] {$c_{-3}^1$};

\draw[gray, very thin] (2.9, -1.8) -- (2.9, 10);
\draw (2.9, -2,0.5pt) node[below] {$c_{-4}^1$};

\draw[gray, very thin] (5.1, -1.8) -- (5.1, 10);
\draw (5.0, -2,0.5pt) node[below] {$c_{-2}^1$};

\draw[gray, very thin] (6.4, -1.8) -- (6.4, 10);
\draw (6.4, -2,0.5pt) node[below] {$c_{-4}^2$};

\draw[gray, very thin] (13.1, -1.8) -- (13.1, 10);
\draw (13.1, -2,0.5pt) node[below] {$c_{-4}^3$};

\draw[gray, very thin] (14.5, -1.8) -- (14.5, 10);
\draw (14.5, -2,0.5pt) node[below] {$c_{-2}^2$};

\draw[gray, very thin] (16.9, -1.8) -- (16.9, 10);
\draw (16.9, -2,0.5pt) node[below] {$c_{-4}^4$};

\draw[gray, very thin] (21.6, -1.8) -- (21.6, 10);
\draw (21.6, -2,0.5pt) node[below] {$c_{-3}^2$};

\draw[gray, very thin] (25.5, -1.8) -- (25.5, 10);
\draw (25.5, -2,0.5pt) node[below] {$c_{-1}$};

\draw[gray, very thin] (26.6, -1.8) -- (26.6, 10);
\draw (26.7, -2,0.5pt) node[below] {$c_{-3}^3$};

\draw[gray, very thin] (29.7, -1.8) -- (29.7, 10);
\draw (29.7, -2,0.5pt) node[below] {$c_{-4}^5$};

\draw[thin, |-| ] (0,0) -- (30,0);
\draw (0,0.5pt) node[below] {$c_2$};
\draw (30,0.5pt) node[below] {$c_1$};
\draw[gray, dashed] (9, -0.3) -- (9, 10);
\draw (9,0.5pt) node[below] {$c_0$};
\draw[black, very thick, |-|]  (0.6,0) -- (2.9,0);
\draw[black, very thick, |-|]  (2.9,0)--(5.1, 0);
\draw[black, very thick, |-|]  (5.1,0) -- (6.4,0);
\draw[black, very thick, |-|]  (6.4,0) -- (9,0);
\draw[black, very thick, |-|]  (9,0) -- (13.1,0);
\draw[black, very thick, |-|]  (13.1,0) -- (14.5,0);
\draw[black, very thick, |-|]  (14.5,0) -- (16.9,0);
\draw[black, very thick, |-|]  (16.9,0) -- (21.6,0);
\draw[black, very thick, |-|] (21.6,0) -- (25.5,0);
\draw[black, very thick, |-|] (25.5,0) -- (26.6,0);
\draw[black, very thick, |-|] (26.6,0) -- (29.7,0);

\draw[thin, |-| ] (14,2) -- (30,2);
\draw (14, 2,0.5pt) node[below] {$c_3$};
\draw (30, 2, 0.5pt) node[below] {$c_1$};
\draw[ -> , blue, very thin] (5, 0.5) -- (21, 1.5);
\draw[ ->, red, very thin] (16, -0.5)--(16, -0.9)--(8, -0.9)--(8, -0.5);

\draw[black, very thick, |-|]  (14.5,2) -- (16.9,2);
\draw[black, very thick, |-|]  (16.9,2) -- (21.6,2);
\draw[black, very thick, |-|] (21.6,2) -- (25.5,2);
\draw[black, very thick, |-|] (25.5,2) -- (26.6,2);
\draw[black, very thick, |-|] (26.6,2) -- (29.7,2);

\draw[thin, |-| ] (0,4) -- (26.5,4);
\draw (0, 4,0.5pt) node[below] {$c_2$};
\draw (26.5, 4,0.5pt) node[below] {$c_4$};
\draw[ ->, blue, very thin ] (23, 2.5) -- (11, 3.5);
\draw[ ->, blue, very thin ] (7, 3.5) -- (13, 2.5);
\draw[black, very thick, |-|]  (0.6,4) -- (2.9,4);
\draw[black, very thick, |-|]  (2.9,4)--(5.1, 4);
\draw[black, very thick, |-|]  (5.1,4) -- (6.4,4);
\draw[black, very thick, |-|]  (6.4,4) -- (9,4);
\draw[black, very thick, |-|]  (9,4) -- (13.1,4);
\draw[black, very thick, |-|]  (13.1,4) -- (14.5,4);
\draw[black, very thick, |-|]  (14.5,4) -- (16.9,4);
\draw[black, very thick, |-|]  (16.9,4) -- (21.6,4);
\draw[black, very thick, |-|] (21.6,4) -- (25.5,4);

\draw[thin, |-| ] (6.5,6) -- (30,6);
\draw (6.5, 6,0.5pt) node[below] {$c_5$};
\draw (30, 6,0.5pt) node[below] {$c_1$};
\draw[ -> , blue, very thin] (14, 4.5) -- (12, 5.5);
\draw[ -> , blue, very thin] (18, 5.5) -- (14.5, 0.5);
\draw[black, very thick, |-|]  (9,6) -- (13.1,6);
\draw[black, very thick, |-|]  (13.1,6) -- (14.5,6);
\draw[black, very thick, |-|]  (14.5,6) -- (16.9,6);
\draw[black, very thick, |-|]  (16.9,6) -- (21.6,6);
\draw[black, very thick, |-|] (21.6,6) -- (25.5,6);
\draw[black, very thick, |-|] (25.5,6) -- (26.6,6);
\draw[black, very thick, |-|] (26.6,6) -- (29.7,6);

\draw[thin, |-| ] (27.5,8) -- (30,8);
\draw (27.5, 8,0.5pt) node[below] {$c_6$};
\draw (30, 8,0.5pt) node[below] {$c_1$};
\draw[ ->, blue, very thin ] (8, 6.5) -- (26, 7.5);

\end{tikzpicture}
\caption{A sketch of the first few levels of a Hofbauer extension for a unimodal map.  The 
dashed line shows where we cut at the critical point $c_0$, a blue arrow shows movement between domains in different levels and a red arrow shows movement between domains in the same level (the colouring will be most useful when we have extra cuts, as we will later).  
We denote by $c_n = f^n(c_0)$.  We also indicate the boundaries of the cylinder sets, denoted $c^i_{-j}$, in $\P_4$, and draw vertical lines to indicate how this lifts to $\hat\P_4$.  
Thick vertical lines imply these points are doubled.
These endpoints are then used to determine $\hat I'(4)$, 
as well as the domains $Q$ of $Y$, which are drawn with thick black lines.
}
\label{fig:Hofbasic}
\end{figure}

\begin{lemma}[Markov property of $F$]
\label{lem:Markov}
If $Y_i$ is a domain of 
$F$ with $F(Y_i) \subset Q\in \Q$ then $F(Y_i) = Q$.
\end{lemma}

\begin{proof}
Let $D\in \D$ denote the domain in which $Y_i$ lies and suppose $R_i=n$.  By the Markov structure of the Hofbauer extension there must exist $Y_i'\subset D$ such that $\hat f^n(Y_i')=Q$.  If $Y_i\neq Y_i'$ then  the only constraint that $Y_i$ must satisfy which $Y_i'$ does not need to is that $Y_i$ must be contained in some $Q'\in \hat\P_L$.  This means that  $Y_i$ must have an element of 
$\bd\hat\P_L$ as a boundary point: indeed it must be adjacent to some $D_\ell^L$ or $D_r^L$.  Denote such a point by $a_{-j}$ where $\pi(a_{-j})\in f^{-j}\crit$.  In particular $j\le L$.   So if $n>j$ then in fact $\hat f^n(a_{-j})$ must be a boundary point of some $D\in \D$, which is a contradiction.  On the other hand if $n\le j$ then $\hat f^n(a_{-j})$ is a boundary point of an element of $\hat\P_L$ so in fact $\hat f^n(Y_i)=Q$ and $Y_i= Y_i'$.
\end{proof}

\begin{remark}
In the construction above, we used $\hat\P_L=\hat\P_L(\crit)$ to firstly arrange for $\hat I(L)$ to be trimmed to $\hat  I'(L)$ and then secondly to partition the domains of $\hat  I'(L)$ into $\Q$.  We observe here that if a subset $\crit'\subset \crit$ is instead used to produce $\hat\P_L(\crit')$ and this set used in place of  $\hat\P_L(\crit)$, the setup above, and in particular the conclusion of Lemma~\ref{lem:Markov}, still holds.
We will employ such a construction in Section~\ref{ssec:puncture}.
\label{rmk:still Markov}
\end{remark}

Note that the set of domains generate a cylinder structure for $F$, which we will denote by
$\{ Y^{(n)}_i \}_i$ for the collection of $n$-cylinders.
The Markov structure of the Hofbauer extension implies for that each domain of $F$, if it maps onto $D(L)\in \hat I'(L)$, where $D(L)\subset D\in \D$,  then there is  an extension so that $F$ extends to a map onto $D$.
As in Lemma~\ref{lem:Koebe}, this extension property gives us bounded distortion for $F$: there exists $K\ge 1$ such that for $Y_i$ a domain of $F$, if $x, y\in Y_i$ then
$$\frac{|DF(x)|}{|DF(y)|}\le K$$
(we improve on this estimate in Lemma~\ref{lem:distort}). Note that $K$ depends
on $L$ since $L$ determines the constant $\epsilon$ in Lemma~\ref{lem:Koebe}.

We also note that by \cite[Lemma 10.7]{DobTod15}, $F$ is uniformly hyperbolic, i.e., there exist $C_F>0$ and $\sigma_F>1$ such that for $x\in Y$ and any $n\ge 1$,
\begin{equation}
|DF^n|\ge C_F\sigma_F^n.
\label{eq:unif_hyp}
\end{equation}

Given a potential $\phi:I\to [-\infty, \infty]$, and its normalised lift $\hat \vf = \vf \circ \pi$ as in 
Remark~\ref{rmk:lifted pressure}, we define the \emph{induced potential} 
$$\Phi(x)=\hat\vf(x)+\hat\vf(\hat f(x))+\cdots+ \hat \vf(\hat f^{R(x)}(x)), \quad x \in \hat I'(L).$$  
As in \eqref{eq:conf_meas}, if $\hat m_\vf$ is $\hat \vf$-conformal for $\hat f$, then it is 
also $\Phi$-conformal for $F$.

By Kac's Lemma, since $F$ is a first return map to $Y$, if $\hat \mu$ is a $\hat f$-invariant probability measure then 
\begin{equation} \hat\mu_{Y}=\frac{\hat\mu|_{Y}}{\hat\mu(Y)} \text{ is an $F$-invariant probability measure and }
\hat\mu(Y)=\frac1{\int R~d\hat\mu_Y}.
\label{eq:Kac}
\end{equation}
  We also note that 
  \begin{equation}
\label{eq:Kac2}
\hat\mu(A) = \hat\mu(Y) \sum_i \sum_{j=0}^{R_i-1} \hat\mu_{Y}(Y_i \cap \hatf^{-j}A)
= \sum_i \sum_{j=0}^{R_i-1} \hat\mu (Y_i \cap \hatf^{-j}A), \mbox{ for any Borel $A \subset \hatI$},
\end{equation}
where the sum over $i$ is taken over all 1-cylinders $Y_i$ for $F$, and $R_i = R|_{Y_i}$.  

We close this subsection with the following distortion result, which is primarily due to Lemma~\ref{lem:Koebe}.

\begin{lemma}
\begin{enumerate}
\item[(a)] Suppose that $\phi:I\to \R$ is H\"older continuous with H\"older exponent $\eta \le 1$.  
Then there exists $K_{F, \phi}>0$ such that for any $n$-cylinder $Y^{(n)}_i$, 
 and all $x, y\in Y^{(n)}_i$,
$$|S_n\Phi(x)-S_n\Phi(y)|\le K_{F, \phi}|F^n(x)-F^n(y)|^\eta.$$
\item[(b)] There exists $K_F>0$ such that for any $n$-cylinder of the scheme $Y^{(n)}_i$, and all $x, y\in Y^{(n)}_i$,
 $$\left|\frac{DF^n(x)}{DF^n(y)}-1\right|\le K_F|F^n(x)-F^n(y)|.$$
\end{enumerate}
\label{lem:distort}
\end{lemma}

\begin{proof}
We prove (a) first.  We begin by taking a 1-cylinder $Y_i$ and $x, y \in Y_i$.  Then

\begin{align*}
|\Phi(x)-\Phi(y)| &\le \sum_{k=0}^{R_i-1}|\phi\circ f^{k}(x)-\phi\circ f^{k}(y)|\le C\sum_{k=0}^{R_i-1}|f^{k}(x)-f^{k}(y)|^\eta \\
&= C\sum_{k=0}^{R_i-1}\left(\frac{|f^{k}(x)-f^{k}(y)|}{|F(x)-F(y)|}\right)^\eta |F(x) -F(y)|^\eta 
\le KC |F(x) - F(y)|^\eta \sum_{k=0}^{R_i-1} \left(\frac{|f^k(Y_i)|}{|f^{R_i}(Y_i)|}\right)^\eta, 
\end{align*}
where $K$ is a distortion constant coming from Lemma~\ref{lem:Koebe}.
So for a H\"older condition on the induced potential it suffices to have a bound on 
$ \sum_{k=0}^{R_i-1} \left(\frac{|f^{k}(Y_i)|}{|f^R(Y_i)|}\right)^\eta$, which follows from (PolShr)$_\beta$ for $\beta>1/\eta$.

Note that since $F$ is uniformly hyperbolic as in \eqref{eq:unif_hyp}, this result passes to $n$-cylinders, proving (a). 

Part (b) is an immediate consequence of Lemma~\ref{lem:Koebe}(b).  Note that when considering a cylinder $Y^{(n)}_i$, the switch from $\frac{|x-y|}{|Y^{(n)}_i|}$ to  $|F^n(x)-F^n(y)|$ follows by Lemma~\ref{lem:Koebe}(a) and that $|F^n(Y^{(n)}_i)|\asymp 1$.
\end{proof}

\begin{remark}
\label{rmk:ind press}
The above lemma,  Remark~\ref{rmk:lifted pressure} and the proof of \cite[Propostion 1.6]{DemTodMP} imply that for admissible normalised potentials $\vf$, the induced potential $\Phi$ has $P(\Phi)=0$ where pressure for the induced system is defined analogously to \eqref{eq:pressure}. 
\end{remark}


\subsection{Punctured extensions with uniform images and uniform tails}
\label{ssec:puncture}

In order to study open systems via the Hofbauer extension, once we fix a point $z \in I$ to be the centre
of our hole, we will introduce extra cuts during the construction of the extension.
Indeed, in order to compare Hofbauer extensions with different sets of cuts in a neighbourhood
of $z$, we will construct extensions with uniform images for the induced maps that are
independent of these
extra cuts.  

Our notation is as follows.  For $\ve_0 > 0$ to be chosen below and $0 < \ve < \ve_0$,
we will construct two related Hofbauer extensions:  $\hatI_{z, \ve_0}$ introducing cuts at $z$ and $z \pm \ve_0$;
and $\hatI_{z, \ve_0, \ve}$  introducing cuts at $z, z \pm \ve_0$, and $z \pm \ve$.
In particular, this means that we will add  $f^{-1}(z)$, $f^{-1}(z \pm \ve_0)$ and $f^{-1}(z \pm \ve)$ 
to our critical set. 
The corresponding dynamics are denoted by 
$\hatf_{z, \ve_0}$ and $\hatf_{z, \ve_0, \ve}$, respectively.  A simplified diagram is presented in Figure~\ref{fig:Hofextracuts}.

\begin{figure}[h]
\begin{tikzpicture}[thick, scale=0.5]


\draw[gray, very thin] (0.6, -1.8) -- (0.6, 20);
\draw (0.6, -2,0.5pt) node[below] {$c_{-3}^1$};

\draw[gray, very thin] (2.9, -1.8) -- (2.9, 20);
\draw (2.9, -2,0.5pt) node[below] {$c_{-4}^1$};

\draw[gray, very thin] (5.1, -1.8) -- (5.1, 20);
\draw (5.0, -2,0.5pt) node[below] {$c_{-2}^1$};

\draw[gray, very thin] (6.4, -1.8) -- (6.4, 20);
\draw (6.4, -2,0.5pt) node[below] {$c_{-4}^2$};

\draw[gray, very thin] (13.1, -1.8) -- (13.1, 20);
\draw (13.1, -2,0.5pt) node[below] {$c_{-4}^3$};

\draw[gray, very thin] (14.5, -1.8) -- (14.5, 20);
\draw (14.5, -2,0.5pt) node[below] {$c_{-2}^2$};

\draw[gray, very thin] (16.9, -1.8) -- (16.9, 20);
\draw (16.9, -2,0.5pt) node[below] {$c_{-4}^4$};

\draw[gray, very thin] (21.6, -1.8) -- (21.6, 20);
\draw (21.6, -2,0.5pt) node[below] {$c_{-3}^2$};

\draw[gray, very thin] (25.5, -1.8) -- (25.5, 20);
\draw (25.5, -2,0.5pt) node[below] {$c_{-1}$};

\draw[gray, very thin] (26.6, -1.8) -- (26.6, 20);
\draw (26.7, -2,0.5pt) node[below] {$c_{-3}^3$};

\draw[gray, very thin] (29.7, -1.8) -- (29.7, 20);
\draw (29.7, -2,0.5pt) node[below] {$c_{-4}^5$};

\draw[ |-| ] (0,0) -- (30,0);
\draw (0,0.5pt) node[below] {$c_2$};
\draw (30,0.5pt) node[below] {$c_1$};
\draw[gray, dashed] (9, -0.3) -- (9, 20);
\draw (9,0.5pt) node[below] {$c_0$};
\draw[gray, dashed] (21, -0.3) -- (21, 20);
\draw (21,0.5pt) node[below] {$y^{-\eps}_0$};
\draw[gray, dashed] (22.5, -0.3) -- (22.5, 20);
\draw (22.5,-0.1, 0.5pt) node[below] {$y_0$};
\draw[gray, dashed] (24, -0.3) -- (24, 20);
\draw (24,0.5pt) node[below] {$y^{+\eps}_0$};
\draw [decorate,decoration={brace,amplitude=5pt},rotate=0] (-0.5,-0.7) -- (-0.5,0.7);
\draw (-1, 0) node[left] {$0$};


\draw[blue, very thin, ->] (5.5, 0.5)--(13, 3.2);

\draw[blue, very thin, ->] (15.5, 0.5)--(18, 1.5);

\draw[ |-| ] (14,4) -- (30,4);
\draw (14, 4,0.5pt) node[below] {$c_3$};
\draw (30, 4, 0.5pt) node[below] {$c_1$};
\draw[ |-| ] (0,2) -- (10.7,2);
\draw (0.1, 2,0.5pt) node[below] {$c_2$};
\draw (30, 4, 0.5pt) node[below] {$c_1$};
\draw[ |-| ] (10.7,2) -- (12.8,2);
\draw (10.7, 2, 0.5pt) node[below] {$y_1^{+\eps}$};
\draw[ |-| ] (12.8,2) -- (14.9,2);
\draw (12.8, 2, 0.5pt) node[below] {$y_1$};
\draw[ |-| ] (14.9,2) -- (30,2);
\draw (14.9, 2, 0.5pt) node[below] {$y_1^{-\eps}$};
\draw (30, 2, 0.5pt) node[below] {$c_1$};
\draw [decorate,decoration={brace,amplitude=5pt},rotate=0] (-0.5,1) -- (-0.5,4.5);
\draw (-1, 3) node[left] {$1$};


\draw[red, very thin, ->] (5.5, 2.7)--(13, 3.7);

\draw[red, very thin,->] (23.2, 3.7)--(11.7, 2.4);

\draw[blue, very thin,->] (23.2, 0.3)--(11.7, 1.7);

\draw[blue,very thin, ->] (17.5, 4.5)--(18.5, 7.5);

\draw[blue, very thin,->] (17.8, 2.3)--(19.5, 5.5);


\draw[blue, very thin,->] (12.2, 2.3)--(27.5, 5.5);


\draw[red, very thin,->] (23.2, 1.8)--(23.2, 1.2)--(12.2, 1.2)--(12.2, 1.8);

\draw[ |-| ] (14.9,6) -- (24.4,6);
\draw (14.9, 6, 0.5pt) node[below] {$y_1^{-\eps}$};
\draw (24.4, 6, 0.5pt) node[below] {$y_2^{-\eps}$};
\draw[ |-| ] (24.4,6)--(26.8, 6);
\draw (26.8, 6, 0.5pt) node[below] {$y_2$};
\draw[ |-| ] (26.8,6)--(28.6, 6);
\draw (28.6, 6, 0.5pt) node[below] {$y_2^{+\eps}$};
\draw[ |-| ] (28.6, 6) -- (30, 6);
\draw (30, 6, 0.5pt) node[below] {$c_1$};
\draw[ |-| ] (14.9,8) -- (26.5,8);
\draw (14.9, 8, 0.5pt) node[below] {$y_1^{-\eps}$};
\draw (26.5, 8,0.5pt) node[below] {$c_4$};

\draw [decorate,decoration={brace,amplitude=5pt},rotate=0] (-0.5,5) -- (-0.5,8.5);
\draw (-1, 7) node[left] {$2$};


\draw[blue, very thin,->] (21.5, 5.7)--(13.8, 2.4);
\draw[blue, very thin,->] (23.2, 7.7)--(11.7, 2.4);
\draw[blue, very thin,->] (17.2, 7.7)--(17.7, 6.4);

\draw[blue,very thin, ->] (25.5, 8.2)--(11.85, 11.6);


\draw[blue, very thin,->] (27.5, 6.3) -- (4, 9.6);


\draw[blue, very thin,->] (24.1, 6.2)-- (10, 9.5);

\draw[ |-| ] (0,10) -- (2.3,10);
\draw (0.1, 10, 0.5pt) node[below] {$c_2$};
\draw (2.3, 10, 0.5pt) node[below] {$y_3^{+\eps}$};
\draw[ |-| ] (2.3,10) -- (4.8,10);
\draw (4.8, 10, 0.5pt) node[below] {$y_3$};
\draw[ |-| ] (4.8,10) -- (9.4, 10);
\draw (9.4, 10, 0.5pt) node[below] {$y_3^{-\eps}$};
\draw[ |-| ] (9.4,10) -- (10.7, 10);
\draw (10.7, 10, 0.5pt) node[below] {$y_1^{+\eps}$};
\draw[ |-| ] (6.7,12) -- (10.7,12);
\draw (6.7, 12, 0.5pt) node[below] {$c_5$};
\draw (10.7, 12, 0.5pt) node[below] {$y_1^{+\eps}$};

\draw [decorate,decoration={brace,amplitude=5pt},rotate=0] (-0.5,9) -- (-0.5,12.5);
\draw (-1, 11) node[left] {$3$};


\draw[blue, very thin,->] (1.3, 10.5)-- (15, 15.5);
\draw[blue,very thin, ->] (3.9, 10.5)-- (22, 15.5);

\draw[blue, very thin,  ->] (7.7, 10.3)-- (27.0, 15.5);


\draw[blue, very thin,->] (10.2, 10.2)-- (28.5, 13.7);

\draw[blue, very thin,->] (7.9, 12.2)-- (28.5, 17.7);

\draw[blue, very thin,->] (10.9, 11.6)-- (29.2, 6.5);


\draw[ |-| ] (28.6, 14)--(29.2, 14);
\draw (28.6, 14, 0.5pt) node[below] {$y_2^{+\eps}$};
\draw (29.7, 14, 0.5pt) node[below] {$y_4^{-\eps}$};
\draw[ |-| ] (29.2, 14)--(30,14);
\draw (30.7, 14, 0.5pt) node[below] {$c_1$};

\draw[ |-| ] (14,16) -- (19,16);
\draw (14, 16,0.5pt) node[below] {$c_3$};
\draw (19, 16,0.5pt) node[below] {$y_4^{+\eps}$};
\draw[ |-| ] (19,16)--(25, 16);
\draw (25, 16,0.5pt) node[below] {$y_4$};
\draw[ |-| ] (25, 16)--(30,16);
\draw (30, 16,0.5pt) node[below] {$c_1$};
\draw[ |-| ] (27.5, 18)--(30,18);
\draw (27.5, 18,0.5pt) node[below] {$c_6$};
\draw (30, 18,0.5pt) node[below] {$c_1$};

\draw [decorate,decoration={brace,amplitude=5pt},rotate=0] (-0.5,13) -- (-0.5,18.5);
\draw (-1, 16) node[left] {$4$};


\draw[blue, very thin, ->] (23.2, 15.7)--(11.7, 2.8);
\draw[blue, very thin, ->] (21.7, 15.7)--(13.8, 2.8);

\end{tikzpicture}

\caption{The Hofbauer extension based on the same unimodal map as in Figure~\ref{fig:Hofbasic}, but with new cuts at points $y_0^{+\eps}, \ y_0$ and $y_0^{-\eps}$.   These represent one set of preimages of $f^{-1}(z-\eps_0)$, $f^{-1}(z)$ and $f^{-1}(z+\eps_0)$: adding in all preimages adds to the complexity of the diagram significantly. 
  Similarly for simplicity we include only the preimages of $c_0$ at the bottom of the diagram, 
 and omit the preimages of $y_0^{+\eps}, \ y_0$ and $y_0^{-\eps}$ (therefore unlike in Figure~\ref{fig:Hofbasic} we do not mark out the domains $Q$ of $Y$ here).  In levels above 0 any marked point is a boundary point of $\D$: thicker markers imply that these points are doubled.  Since, in contrast to Figure~\ref{fig:Hofbasic}, the number of domains of a given level can be greater than one, we add in the numbers on the left to clarify the level of each domain.  Note that with $\eps_0$ fixed less than $\eps_0^*$, 
additional cuts can be introduced at $z \pm \eps$ for $\eps <\eps_0$ which do not affect the
structure of the cylinders outside the intervals $(y_n, y_n^{+\eps})$ and $(y_n^{-\eps}, y_n)$. }
\label{fig:Hofextracuts}
\end{figure}

We fix $z$ and at the beginning of Section~\ref{ssec:uniform}, we will choose 
the relevant quantities in the following order.  First, we will choose $L$ according to 
Theorem~\ref{thm:tails}, which will provide uniform control on the complexity of the
tail of the Hofbauer extension and will depend only on the cardinality of the critical set
plus $5 \deg(f)$.  Next, we will choose $\ve_0^*$
according to \eqref{eq:eps_0}, then finally we choose $\ve_0 \le \ve_0^*$, which will fix 
the return domain $Y$, and work with
$0 < \ve < \ve_0$ as the variable size of the hole.

\underline{$\hat I_z=\hat I(\Crit_{z})$}. Let $\Crit_{z}$ denote the expanded critical set, i.e., $\Crit\cup\{f^{-1}z\}$.
Next, consider the partition $\P_L=\P_L(\Crit_z)$ of $I$ into $L$-cylinders with endpoints at
$\{ f^{-j}(y): y \in \Crit_z, 0 \le j \le L \}$.  We choose
\begin{equation}
\label{eq:eps_0}
\ve^*_0 < \frac{1}{|Df^L|_\infty} \min \{ |x - y| : x \neq y , x \in \partial \P_L, y \in 
f^j(\Crit_z), 0 \le j \le L \} .
\end{equation}

\underline{$\hat I_{z, \ve_0}=\hat I(\Crit_{z, \ve_0})$}.  For $0 < \ve_0 \le \ve_0^*$,
we define $\hatI_{z, \ve_0}=\hat I(\Crit_{z, \ve_0})$ as above where $\Crit_{z, \ve_0}$ has $f^{-1}(z\pm \ve_0)$ added to $\Crit_z$.  Let $\hatI_{z, \ve_0}(L)$ denote the first $L$
levels of $\hatI_{z, \ve_0}$, and 
let $\hatI'_{z, \ve_0}(L)$ denote $\hatI_{z, \ve_0}(L)$
minus the elements of $\hat{\mathcal{P}}_L(\Crit_{z, \ve_0})$  adjacent to each boundary point in $\hatI_{z, \ve_0}(L)$, as in \eqref{eq:trim}, so that the new boundary points are of the form $f^{-j}(y)$ for some $y \in \Crit_{z, \ve_0}$ and $0 \le j \le L$.  
Note that by choice of $\eps^*_0$, we completely remove elements of the form $[f^k(z), f^k(z+\eps_0)]$ for $0 \le  k\le L$, and analogues, in going from $\hatI_{z, \ve_0}(L)$ to $\hatI_{z, \ve_0}'(L)$.

\underline{$\hat I_{z, \ve_0, \ve}=\hat I(\Crit_{z, \ve_0, \ve})$}.
For any $\ve \in (0, \ve_0)$, we define $\Crit_{z, \ve_0, \ve}$ to be $\Crit_{z, \ve_0}$ with $f^{-1}(z\pm\eps)$ added.  Let $\hat I_{z, \ve_0, \ve}=\hat I(\Crit_{z, \ve_0, \ve})$ and 
define $\hatI'_{z, \ve_0, \ve}(L)$ to be the first $L$ levels, 
$\hatI_{z, \ve_0, \ve}(L)$,
minus the elements of $\hat{\mathcal{P}}_L(\Crit_{z, \ve_0, \ve})$ adjacent to each boundary point in $\hatI_{z, \ve_0, \ve}(L)$ so that the new boundary points are of the form $f^{-j}(y)$ for some $y \in \Crit_{z, \ve_0, \ve}$ and $0 \le j \le L$.  As above, we completely remove elements of the form $[f^k(z), f^k(z+\eps_0)]$ for $0 \le k\le L$, and analogues, in going from $\hatI_{z, \ve_0, \ve}(L)$ to $\hatI_{z, \ve_0, \ve}'(L)$.

As can be seen from this construction, the domains of  $\hatI'_{z, \ve_0}(L)$ and  $\hatI'_{z, \ve_0, \ve}(L)$ are the same.  We choose $\Q$ to be the domains of  $\hatI'_{z, \ve_0, \ve}(L)$ further partitioned by $\hat{\mathcal{P}}_L(\Crit_{z, \ve_0})$.  We choose this partition rather than  $\hat{\mathcal{P}}_L(\Crit_{z, \ve_0, \ve})$ to ensure our $F$-images have size independent of $\ve$ and because, as in Remark~\ref{rmk:still Markov}, this does not affect the Markov structure for $F$ since
the extra cuts due to $\ve$ fall within intervals of the form $[f^k(z), f^k(z+\eps_0)]$ for $0 \le k\le L$,
which have already been removed from $\hatI'_{z, \ve_0, \ve}(L)$.

\begin{remark}
Here we explain how cutting at $f^{-1}(z)$ and our choice of $\ve_0^*$ ensures that the representatives of the holes in the Hofbauer extension are disjoint from our inducing domains.
\begin{enumerate}
\item[(a)] If $\hat f:U\to D$, $D \in \mathcal{D}$, is a homeomorphism, then since we cut at  $f^{-1}(z)$, the interior of $U$ cannot intersect $\pi^{-1}(f^{-1}(z))$, which also implies that the interior of $D$ cannot intersect $\pi^{-1}(z)$.  Therefore, this fact must be true for any $D$ in the transitive part of $\hat I_z$.  So we conclude that $\pi^{-1}(z)\cap \hat I'_z(L) =\es$ due to trimming of $L$-cylinders.

\item[(b)]
Suppose that $\hat J_{\ve_0}(z)\subset D\in \D$ where $\lev(D)=k\in \{0, \ldots, L\}$ and $\pi(\hat J_{\ve_0}(z)) \subset (z - \ve_0, z+\ve_0)$.  By \eqref{eq:eps_0},  
$\left(\hatf_{z, \ve_0}^j(\hat J_{\ve_0}(z))\right)  \cap \hatI_{z, \ve_0}'(L)= \es$,
for all $j = 0, \ldots, L-k$.   As a consequence $\pi^{-1}( (z - \ve_0, z+\ve_0)) \cap\hatI'_{z, \ve_0}(L)=\es$ and there is a one-to-one correspondence between elements of
$\hatI'_{z, \ve_0}(L)$ and $\hatI'_{z, \ve_0, \ve}(L)$; indeed, precisely the same domains appear
on each level.  Abusing notation slightly, we write $\hatI'_{z, \ve_0}(L) = \hatI'_{z, \ve_0, \ve}(L)$,
and once $L$ is fixed, simply refer to the common set of domains as $$Y = \sqcup_{Q\in \hat{\mathcal{P}}_L(\Crit_{z, \ve_0})} Q.$$
As a result of this construction, $Y \cap \hat H_\ve = \emptyset$ for all $\ve < \ve_0$, where $\hat H_\ve = \pi^{-1}(H_\ve)$.


\end{enumerate}
\label{rmk:def of Y}
\end{remark}

\begin{remark}
(Role of $\ve$ and $\ve_0$.)  The cuts at $z \pm \ve$ form the boundary of the hole
$H_\ve$, and defining $\hat I_{z, \ve_0, \ve}$ with respect to these cuts guarantees that
the Markov structure will respect the hole.  The extra cuts at $z \pm \ve_0$ are used to guarantee
uniform images and tails for returns to $Y$ as $\ve \to 0$.  Without loss of generality on $H_\ve$,
we may always choose $\ve_0$ to satisfy,
\begin{equation}
\label{eq:eps0 choice}
\{ f^\ell(z \pm \ve_0) \}_{\ell \ge 0} \cap \Crit_z = \emptyset \mbox{ and }
\{ f^\ell(z) \}_{\ell \ge 0} \cap \{ z \pm \ve_0 \} = \emptyset . 
\end{equation}

In fact, we will only need to invoke
\eqref{eq:eps0 choice} in Section~\ref{sec:zerohole} to prove convergence to the asymptotic escape rate in the case that $z$ is periodic (see Lemmas~\ref{lem:holder per} and \ref{lem:t per} ).  All results
in Sections~\ref{ssec:uniform}, \ref{sec:gap} and \ref{sec:Young tower} hold for all $\ve_0 \le \ve_0^*$.

The size of $\ve >0$ will be further reduced in Corollaries~\ref{cor:large connection}
and \ref{cor:spectral gap}, and Lemma~\ref{lem:mixing D} to satisfy $\ve < \ve_1^*$, where $\ve_1^* < \ve_0$ guarantees that the corresponding
induced maps are uniformly mixing and the associated transfer operators have a uniform spectral gap.
\end{remark}

As defined above, $\Q$ denotes the finite partition of 
$Y$ into its domains.
Define the induced maps
$F_{z, \ve_0} = \hatf_{z, \ve_0}^{R_{z, \ve_0}}$ and 
$F_{z, \ve_0, \ve} = \hatf_{z, \ve_0, \ve}^{R_{z, \ve_0, \ve}}$ acting on the domain
$Y$, where $R_\varkappa$ denotes the first return time
to $Y$ in the extension $(\hatf_\varkappa, \hatI_\varkappa)$, and $\varkappa$
stands for either of the indices $(z, \ve_0)$ or $(z, \ve_0, \ve)$.  By
construction, all images of elements of $\Q$ under $F_{\varkappa}$ are unions of elements of
$\Q$.  Thus $F_{\varkappa}$ has the finite images property.

We have a natural projection $\pi_\varkappa : \hatI_\varkappa \to I$ which commutes with the
dynamics, $\pi_\varkappa \circ \hatf_\varkappa = f \circ \pi_\varkappa$.  Note that from here on we will fix $\hat m=\hat m_\vf$ for the relevant $\vf = \phi - P(\phi)$.  As in Remark~\ref{rmk:def of Y}, $\hat m$ is the same for all $Y=\hatI_{\varkappa}'(L)$, and moreover $\hat m$ is always conformal for $\Phi_\varkappa$ under $F_\varkappa$ and we obtain $F_\varkappa$-invariant measures as in \eqref{eq:Kac} and \eqref{eq:Kac2}.

Define $\mathfrak{d} = \# \crit + 5 \deg(f) \ge \# \Crit_{z, \ve_0, \ve}$, and note that
by definition, $\mathfrak{d}$ does not depend on $\ve$ and $\ve_0$, just on the fact that we have
introduced extra cuts at the preimages of the 5 points, $z$, $z \pm \ve_0$ and $z \pm \ve$. 
Our first result provides uniform bounds on the tail of the return time functions $R_{z, \ve_0}$
and $R_{z, \ve_0, \ve}$.

\begin{theorem}
Suppose that either:
\begin{enumerate}
\item[(a)]  $\phi=-t\log|Df|$ for $t \in (t^-, t^+)$; or
\item[(b)] 
$\phi:I\to \R$ is H\"older continuous.
\end{enumerate}
 Then there exist $L\in \N$, $C>0$ and $\alpha>0$ such that 
 for all $0 < \ve < \ve_0 \le \ve_0^*$, $F_{\varkappa}$, 
 the first return map to $\hat I'_\varkappa(L)$, has tails 
 $\hat m_{\phi-P(\phi)}(R_\varkappa>n)\le Ce^{-\alpha n}$ where 
 $\varkappa = \{ z, \ve_0 \}$ or $\{z, \eps_0, \eps\}$. Here $L, C, \alpha$ depend only on $(f, \phi, \mathfrak{d})$.
\label{thm:tails}
\end{theorem}

\begin{proof}[Proof of Theorem~\ref{thm:tails}]
For ease of notation, we will drop the subscript $\varkappa$ in the proof, 
but all statements apply
equally well to $F_{z, \ve_0}$ and $F_{z, \ve_0, \ve}$.

As shown in \cite[Lemma 4.15]{DobTod15}, for each $\xi>0$ there exist $L=L(\xi)$  and $n_0=n_0(L)$ such that $\#\{i: R(Y_i) = n \}\le e^{\xi n}$ for all $n \ge n_0$.  Crucially these numbers only depend on $\mathfrak{d}$, so are independent of the actual values of 
$\eps$ and $\eps_0$.   Thus to prove the theorem, it suffices to show that there exists some 
$\bar\alpha>\xi$ such that for any 1-cylinder $Y_i$ of $F$, $\hat m(Y_i) \le e^{-\bar\alpha R_i}$ where $R_i:=R|_{Y_i}$.

In the geometric case i.e., case (a), we will set 
$$\bar\alpha=p_t+t\lmin.$$ 
The fact that $\bar\alpha>0$ follows immediately from our having set $t<t^+$.
In the H\"older case we obtain an analogous $\bar\alpha>0$ using the 
assumed pressure gap, $\phi < P(\phi)$.  In both cases we now can select $\xi<\bar\alpha$, which then fixes $L(\xi)$ and $\Q$.  We will see below that our estimates on the measures of the domains $Y_i$ yield $\alpha=\bar\alpha-\xi$. 

We will use the expansion on periodic orbits to estimate the measure of the domains $Y_i$.  The proof of this theorem would be simpler if we had $Y_i\subset F(Y_i)$ for all $i$, since then each $Y_i$ would contain a point of period $R_i$, which will allow us to 
connect $\bar\alpha$ and the measure of $Y_i$.  
To overcome this issue, we will first prove that $F$ is transitive on elements
of $\mathcal{Q}$.   Recall that by Lemma~\ref{lem:fibres meet}, 
if $O_1, O_2 \subset \hatI_{z, \ve_0}'(L)$
are two open sets such that $\pi(O_1) \cap \pi(O_2) \neq \emptyset$, then there exists
$n \in \mathbb{N}$ such that $F^n(O_1) \cap F^n(O_2) \neq \emptyset$. \\
\\
Now let $Q_1, Q_2 \in \Q$.  Since
$f$ is leo, there exists $n_1 \in \mathbb{N}$ such that $\pi(\hatf^{n_1}Q_1) \supset I \supset \pi(Q_2)$.
By Lemma~\ref{lem:fibres meet}, there exists $n_2 \in \mathbb{N}$ such that 
$\hat f^{n_1 + n_2}(Q_1) \cap \hat f^{n_2}(Q_2) \neq \emptyset$.
Since $Q_2$ is a recurrent element of $\Q$, 
there exists $n \in \mathbb{N}$ such that $F^n(Q_1) \cap Q_2 \neq \emptyset$.
Then the Markov property of $F$ implies that $F^n(Q_1) \supset Q_2$, and the
claimed transitivity follows.

Since $\Q$ is finite,
there exist $N\ge1$ and $C>0$ such that for each pair $Q_1, Q_2\in \Q$, there are $J \subset Q_1$ 
and  $n \le N$ such that $\hat f^n: J \to Q_2$ is a diffeomorphism with $\left|D\hat f^n|_J\right|\ge C$ .  Therefore, each domain $Y_i$ of the inducing scheme contains a periodic point $y_i$ 
with period $R_i\le p \le R_i+N$ for $\hat f$.  Then $|DF(y_i)|\ge C^{-1} |D\hat f^p(y_i)|\gtrsim e^{\lmin R_i}$.  
Throughout we will treat $\hat f^{R_i}(Y_i)$ as having uniform size, i.e., independent of $i$.

In case (a), Lemma~\ref{lem:Koebe} implies 
\begin{align}
 \hat m(\hat f^{R_i}(Y_i)) & =\int_{Y_i}e^{-S_{R_i}(t\hat \phi-p_t)}d\hat m=\int_{Y_i}|DF|^te^{R_ip_t}d\hat m 
 \asymp |DF(y_i)|^t\hat m(Y_i)e^{R_ip_t} .
 \label{eq:confcyl_new}
\end{align}
Therefore, $\hat m(Y_i)\lesssim e^{-R_i(t\lmin+p_t)}= e^{-\bar\alpha R_i}$.  

For the H\"older case, recall that we have assumed that $\phi<P(\phi)$, 
and thus by Remark~\ref{rmk:lifted pressure},  $\hat \phi < P(\hat \phi)$ on $\hat I$. 
 Our value of $\bar\alpha$ here is $\inf\left\{P(\phi)-\frac{S_p\phi(x)}p: f^p(x)=x\right\}\ge \inf\{P(\phi)-\phi(x):x\in I\}>0$. So again, using a slightly more elementary version of the estimate in \eqref{eq:confcyl_new}  in conjunction with Lemma~\ref{lem:distort},  to give us our requisite distortion property, the result follows.
\end{proof}


\subsection{Uniform mixing for $F_{z, \ve_0, \ve}$}
\label{ssec:uniform}

Now we choose $L$ large enough so that the conclusion of Theorem~\ref{thm:tails} is satisfied.  Furthermore, we enlarge
$L$ if necessary so that 
\begin{enumerate}
  \item[a)] $\displaystyle \kappa = \max\{ \hat m(\hatI_{z, \ve_0, \ve} \setminus \hatI'_{z,\ve_0}(L)) , 
\hat m(\hatI_{z, \ve_0} \setminus \hatI'_{z,\ve_0}(L)) \} < 1/3$; and
  \item[b)]  any ergodic invariant measure $\nu$ with entropy $h_\nu(f) > (\log \overline\lambda_{\ve^*} + \alpha)/2$ lifts to our inducing
  scheme on $\hat I_{z, \ve_0}'(L)$,  where $\ve^*$ is from Theorem~\ref{thm:accim}.
   \end{enumerate}
Item (b) is possible due to \cite[Lemma~8.2]{DobTod15}, and the fact that $\alpha$ does not decrease as $L$ increases.
   
With $L$ fixed, we define $\ve_0^*$ as in \eqref{eq:eps_0}, and for $\ve_0 \le \ve_0^*$, we let $Y = \hatI'_{z,\ve_0}(L)$ as in Remark~\ref{rmk:def of Y}.

Our next result proves a necessary mixing property for our return maps.

\begin{lemma}
\label{lem:mixing}
For all $\ve_0 \le \ve_0^*$ and $\ve < \ve_0$, the induced maps
$F_{z, \ve_0}$ and $F_{z, \ve_0, \ve}$ are
topologically mixing on $Y$.
\end{lemma}

\begin{proof}
We write our arguments for $F_{z, \ve_0, \ve}$, but the same proof holds for $F_{z, \ve_0}$. 

By the proof of Theorem~\ref{thm:tails},
$F_{z, \ve_0, \ve}$ is transitive on the finitely many elements of $\Q$.
The only way it can
fail to be mixing is if the images decompose into a periodic cycle.  Let $Q \in \Q$.  Since $f$ is leo, 
there exists $n_Q$ such that $f^n(\pi(Q)) \supset I$
for all $n \ge n_Q$.  By our choice of $L$, 
$\kappa = \hat m(\hatI_{z, \ve_0, \ve} \setminus Y) < 1/3$.  
Then since $f^n \circ \pi = \pi \circ \hatf^n_{z, \ve_0, \ve}$, we have
$m(\pi(\hatf^n_{z, \ve_0, \ve}(Q) \cap Y)) \ge 1 - \kappa$, for $n \ge n_Q$.

Applying this to $n = n_Q$ and $n = n_Q+1$, we conclude
\[
m\left(\pi\left(Y \cap \hatf_{z, \ve_0, \ve}^{n_Q}(Q)\right) \cap \pi\left( Y \cap \hatf_{z, \ve_0, \ve}^{n_Q+1}(Q)\right)\right) \ge 1 - 2\kappa >0  .
\]
Thus there must exist intervals $O_1 \subset Y \cap \hatf_{z, \ve_0, \ve}^{n_Q}(Q)$ and $O_2 \subset Y \cap \hatf_{z, \ve_0, \ve}^{n_Q+1}(Q)$ such that $\pi(O_1) \cap \pi(O_2) \neq \emptyset$.  By Lemma~\ref{lem:fibres meet}, there exists
$n_1 \in \mathbb{N}$ such that $F_{z, \ve_0, \ve}^{n_1}(O_1) \cap F_{z, \ve_0, \ve}^{n_1}(O_2) \neq \emptyset$.   
Since $Q, O_1, O_2 \subset Y$, there exists $k_Q \in \mathbb{N}$ such that 
$F_{z, \ve_0, \ve}^{k_Q}(Q) \cap F_{z, \ve_0, \ve}^{k_Q+1}(Q) \neq \emptyset$, 
so the period of $Q$ under $F_{z, \ve_0, \ve}$ is 1.  Thus $F_{z, \ve_0, \ve}$ is aperiodic and therefore mixing.
\end{proof}

Our next two lemmas show that the mixing established in Lemma~\ref{lem:mixing} is in fact
uniform in $\ve$.

\begin{lemma}
\label{lem:uniform mix}
Fix $\ve_0 \le \ve_0^*$ and suppose there exist $Q_1, Q_2 \in \Q$ and an interval $J \subset Q_1$
such that $F_{z, \ve_0}^{n_J}(J) = Q_2$ for some $n_J \in \mathbb{N}$.
Then there exists $\ve_1 < \ve_0$ such that
for all $\ve \in (0, \ve_1)$, $F_{z, \ve_0, \ve}^{n_J}(J) = Q_2$.
\end{lemma}

\begin{proof}
Fix $\ve_0$.
Suppose there exist $Q_1, Q_2 \in \Q$ and an interval $J \subset Q_1$ and $n_J \in \N$
such that $F^{n_J}_{z, \ve_0}(J) = Q_2$ as in the statement of the lemma.  Let $n_1 \in \N$ be such that
$\hatf_{z, \ve_0}^{n_1}(J) = Q_2$.  

A key property of our construction of $Y = \hatI'_{z, \ve_0}(L)$ is that we have `trimmed' the edges
of the domains at returns:  i.e., the endpoints of $Q_1$ and $Q_2$ are elements of  $\bd\P_L= \bd\P_L(\Crit_{z, \ve_0})$ and 
the Markov property of $F_{z, \ve_0}$, Lemma~\ref{lem:Markov}, implies that
there exist domains $Q_1' \supsetneq Q_1$ and $Q_2' \supsetneq Q_2$ in the extension $\hatI_{z, \ve_0}$ (note that $Q_2'$ is an element of $\D$) 
and an interval $J'$ with $J \subsetneq J' \subset Q_1'$ such that $\hatf_{z, \ve_0}^{n_1}(J') = Q_2'$.

Let $\widehat{z}$ and $\widehat{z \pm \ve_0}$ denote the fibres above $z$ and $z \pm \ve_0$, respectively.
Due to the Markov property and because we have treated $f^{-1}(z)$ and $f^{-1}(z \pm \ve_0)$ as 
cut points during our construction of $\hatI_{z, \ve_0}$ and $\hatI_{z, \ve_0, \ve}$, it follows that
$\partial(\hatf_{z, \ve_0}^j(J)) \cap \{ \widehat{z} , \widehat{z \pm \ve_0} \} 
= \emptyset$, for all $0 \le j \le n_1$.

Case 1:  $\bigcup_{j =0}^{n_1} \hatf_{z, \ve_0}^j(J) \cap \left(\widehat{z-\ve_0}, \widehat{z+\ve_0}\right) = \emptyset$.
Then introducing new cuts at $f^{-1}(z \pm \ve)$ in the construction of $\hatI_{z, \ve_0, \ve}$
does not affect the endpoints of either $J'$ or $Q_2'$, and the lemma holds with $\ve_1 = \ve_0$.

Case 2:  $\bigcup_{j =0}^{n_1} \hatf_{z, \ve_0}^j(J) \cap \left(\widehat{z-\ve_0}, \widehat{z+\ve_0}\right) \neq \emptyset$.
Choose
\[
\ve_1 < \min \big\{ d\big( \partial(\hatf_{z, \ve_0}^j(J)), \widehat{z} \big) : 0 \le j \le n_1 \big\}.
\]
It follows that for all $\ve < \ve_1$, $\hatf_{z, \ve_0, \ve}^{n_1}(J) = Q_2$.  Moreover, there exists
an interval $J'_\ve \supsetneq J$ and a domain $Q_{2, \ve}' \supsetneq Q_2$ in $\hatI_{z, \ve_0, \ve}$
such that $\hatf_{z, \ve_0, \ve}^{n_1}(J'_\ve) = Q_{2, \ve}'$.   Then, since $F_{z, \ve_0, \ve}$ is the first return
map to $Y$, and $Y$ is independent of $\ve$, it follows that $F_{z, \ve_0, \ve}^{n_J}(J) = Q_2$.
\end{proof}

\begin{corollary}
\label{cor:large connection}
For all $\delta > 0$ there exists $\ve_1 \in (0, \ve_0)$ such that for all $\ve \in (0, \ve_1)$,
\[
\hat m ( \hat x \in Y : F_{z, \ve_0}(\hat x) \neq F_{z, \ve_0, \ve}(\hat x) ) \le \delta .
\] 
\end{corollary}

\begin{proof}
Fix $\delta > 0$.
By Theorem~\ref{thm:tails}, we may choose $N$ such that 
$\hat m(R_{z, \ve_0} > N) \le C e^{-\alpha N} \le \delta$.
Considering the 1-cylinders for $F_{z, \ve_0}$, there are only finitely many with $R \le N$.

For each 1-cylinder $Y_i$, Lemma~\ref{lem:uniform mix} yields an $\ve_1(i) > 0$ 
such that for all $\ve \in (0, \ve_1(i))$, $Y_i$ is also a 1-cylinder for $F_{z, \ve_0, \ve}$; moreover,
$F_{z, \ve_0}(Y_i) = F_{z, \ve_0, \ve}(Y_i)$ and $R_{z, \ve_0}(Y_i) = R_{z, \ve_0, \ve}(Y_i)$.

Taking $\ve_1 = \min \{ \ve_1(i) : R_{z, \ve_0}(Y_i) \le N \} > 0$ completes the proof of the corollary.
\end{proof}


\section{A spectral gap for the induced punctured transfer operators}
\label{sec:gap}

In this section, we work with the induced maps $F_{z, \ve_0}$ and $F_{z, \ve_0, \ve}$ defined on the
common domain $Y = \hat{I}'_{z, \ve_0}(L)$.  Since $z$ and $\ve_0 \le \ve_0^*$ are fixed throughout
this section, for brevity, we will denote these
maps simply by $F_\ve := F_{z, \ve_0, \ve}$ and $F_0 := F_{z, \ve_0}$.  Related objects will also
be denoted by the subscript $\ve$ or 0.  One of the main points of this section
is to show that certain key properties are uniform for $\ve \in [0, \ve_0)$, where $\ve=0$ is understood
to correspond to the map $F_{z, \ve_0}$ whose Hofbauer extension is defined by introducing cuts
only at $z$ and $z \pm \ve_0$.

For $\ve \in [0, \ve_0)$, let 
$\mathcal{Y}_\ve = \{ Y_i \}_i$ denote the set of 1-cylinders for $F_\ve$ on which $R_\ve = R_{z, \ve_0, \ve}$
is constant.
As before, denote by $\Q$ the finite partition of $Y$ into intervals which comprise the finite images of
$\mathcal{Y}_\ve$ under $F_\ve$.  It is important that $Y$ and $\Q$ are independent of $\ve$.
Indeed, the uniformity of $\Q$ and $L$ allows us to take the constants in 
\eqref{eq:unif_hyp} and Lemma~\ref{lem:distort} uniformly in $\eps$.  This is formalised in properties 
\ref{GM2} and \ref{GM3} below.

Let $\Phi_\ve = S_{R_\ve}\vf$ be the induced version of $\vf$ on $Y$. Note that as in, for example \cite[Lemma 14.9]{DobTod15}, the fact that $\hat\mu_\eps(\hat I'(L))>0$ guarantees 
that $P(\Phi_\ve) =0$.
Also, the conformal measure $m_{\vf}$ lifted to $\hat I_{z, \ve_0, \ve}$, and denoted $\hat m_{\vf, \ve}$, 
depends on both $\ve$ and $\vf$.  However, $\hat m_{\vf, \ve}$ restricted to $Y$ is independent of $\ve$ since $Y$ is independent
of $\ve$.  Since we will work exclusively in $Y$ in this section, we suppress the dependence on $\vf$ and refer to this measure on 
$Y$ as simply $\hat m$.  For each $\ve \in [0, \ve_0)$, it is
a conformal measure for $F_\ve$ with respect to the potential $\Phi_\ve$.

The key properties of the Gibbs-Markov maps $F_\ve$, $\ve \in [0, \ve_0)$, are as follows: 
\begin{itemize}
  \item[\textbf{(GM1)}\labeltext{{(GM1)}}{GM1}]  $F_\ve(Y_i) \in \Q$ for each $Y_i \in \mathcal{Y}_\ve$; 
    \item[\textbf{(GM2)}\labeltext{{(GM2)}}{GM2}]  There exist $\sigma > 1$ and $C_e \in (0,1]$ (an expansion constant) such that for all $n \in \mathbb{N}$, 
  if $Y_i^{(n)}$ is an $n$-cylinder for $F_\ve$ and $x, y \in Y_i^{(n)}$, then
  $d(F_\ve^nx, F_\ve^ny) \ge C_e \sigma^n d(x,y)$, where $d(\cdot, \cdot)$ is the distance on each interval in
  $\hat{I}$ induced by the Euclidean metric on $I$.
  \item[\textbf{(GM3)}\labeltext{{(GM3)}}{GM3}] There exists $C_d > 0$ (a distortion constant) such that for all $n \in \mathbb{N}$, 
  if $Y_i^{(n)}$ is an $n$-cylinder for $F_\ve$ and $x, y \in Y_i^{(n)}$, then
  \[
 \left|e^{S_n \Phi_\ve(x) - S_n \Phi_\ve(y)} - 1\right| \le C_d d(F_\ve^nx, F_\ve^ny)^\eta, 
  \]
  for some $\eta > 0$. 
  \end{itemize}
 Note that \ref{GM3} follows from Lemma~\ref{lem:distort}, and that the constants in \ref{GM2} and \ref{GM3} are 
independent of $\ve$ by construction of $Y$.  Due to \ref{GM3}, conformality and large images,
\begin{equation}
\label{eq:distortion}
e^{S_n\Phi_\ve(x)} \le (1+C_d) \frac{\hm(Y_i^{(n)})}{\hm(F_\ve^n(Y_i^{(n)}))}
\le \frac{1+C_d}{q} \hm(Y_i^{(n)}), \quad \mbox{for all $x \in Y_i^{(n)}$,}
\end{equation}
where $q := \min_{Q \in \Q} \hm(Q) > 0$. 

Let $\cont^\eta(\Q)$ denote the set of H\"older continuous functions on elements of $\Q$, equipped with
the norm,
\[
\| \psi \|_{\cont^\eta} = \sup_{Q \in \Q} \left( |\psi|_{\cont^0(Q)} + \sup_{x,y \in Q} |\psi(x) - \psi(y)| d(x,y)^{-\eta} \right) =: \sup_{Q \in \Q} \big( |\psi|_{\cont^0(Q)} + H^\eta_Q(\psi) \big).
\]

We define the transfer operator $\Lp_\ve = \Lp_{\Phi_\ve}$ acting on $L^1(\hat{m}_\ve)$ by
\[
\Lp_\ve^n \psi(x) = \sum_{y \in F_\ve^{-n}(x)}  \psi(y) e^{S_n\Phi_\ve(y)} , \quad \mbox{for each $n \ge 1$.}
\]
Analogously, define $\Lp_0$ to be the transfer operator corresponding
to the map $F_0 = F_{z, \ve_0}$.

Given a hole $H_\ve = (z-\ve, z+\ve)$, $\ve \in (0, \ve_0)$, as in Remark~\ref{rmk:def of Y}, 
its lift $\hat H_\ve$ 
is disjoint from $Y$ due to our choice of $\ve_0^*$.  We denote by $\hat H_\ve' \subset Y$ 
the \emph{pre-hole}, the set of points in $Y$ which
do not return to $Y$ before entering $\hat H_\ve$.  
Due to our construction, $\hat H_\ve'$ is a (countable) union of 1-cylinders for $F_\ve$,
\[
\hat H_\ve' = \{ Y_i \in \mathcal{Y}_\ve : \hat f_\ve^n(Y_i) \subset \hat H_\ve \mbox{ for some } n < R_\ve(Y_i) \} \, .
\]
We will treat $\hat H_\ve'$ as our effective hole for $F_\ve$.  Let $\hY_\ve = Y \setminus \hat H_\ve'$,
and for $n \ge 0$, define
$$\hY^n_\ve = \cap_{i=0}^n F_\ve^{-i}(\hY_\ve)$$ 
to be the set of points which do not enter $\hat H_\ve'$
in the first $n$ iterates of $F_\ve$.  
The dynamics of the induced open system are defined by
$\hF_\ve^n = F_\ve^n|_{\hY^{n-1}_\ve}$. 
 Since $\hat H_\ve'$
is a union of 1-cylinders for $F_\ve$, the punctured map $\hF_\ve$
has the same finite image property:  $\hF_\ve(Y_i) \in \Q$ for each
$Y_i \subset \hY_\ve$.  The punctured transfer operator for the open system is defined for $n \ge 1$ by
\begin{equation}
\label{eq:hLp}
\hLp_\ve^n \psi(x) = \Lp_\ve^n (\psi 1_{\hY_\ve^{n-1}})(x) = \sum_{y \in \hF_\ve^{-n}(x)} \psi(y) e^{S_n\Phi_\ve(y)} .
\end{equation}
The punctured transfer operator is defined only for $\ve > 0$.  There is no analogous
object for $\Lp_0$.


\subsection{Spectral properties of $\Lp_\ve$}

In this subsection we prove that for sufficiently small $\ve$, all the operators $\Lp_\ve$ have a uniform
spectral gap.

\begin{proposition}
\label{prop:ly}
There exists $C>0$ such that for all $n \ge 0$,
\begin{eqnarray}
\| \hLp_\ve^n \psi \|_{\cont^\eta} & \le & C \sigma^{-\eta n} \| \psi \|_{\cont^\eta} \hm(\hY_\ve^{n-1}) 
+ C \int_{\hY_\ve^{n-1}} |\psi| \, d\hm \qquad \forall \psi \in \cont^\eta(\Q) ,
\label{eq:strong ly} \\
| \hLp_\ve^n \psi |_{L^1(\hm)} & \le & \int_{\hY_\ve^{n-1}} |\psi| \, d\hm 
\qquad \forall \psi \in L^1(\hm) .
\label{eq:weak ly}
\end{eqnarray}
The analogous inequalities hold for $\Lp_\ve^n \psi$ and $\Lp_0^n \psi$ with $\hY_\ve^{n-1}$ replaced by $Y$.
\end{proposition}

\begin{proof}
Due to the definition \eqref{eq:hLp}, $\int_Y \hLp_\ve^n \psi \, d\hm = \int_{\hY^{n-1}_\ve} \psi \, d\hm$,
so that \eqref{eq:weak ly} is immediate.  We focus on verifying \eqref{eq:strong ly} for $\psi \in \cont^\eta(\Q)$.

First, we estimate the H\"older constant of $\hLp_\ve^n \psi$.
Let $Q \in \Q$ and $x, y \in Q$.  For $n \ge 0$, notice that each $u \in \hF_\ve^{-n}(x)$ has a
(unique) corresponding $v \in \hF_\ve^{-n}(y)$ lying in the same $n$-cylinder $Y_i^{(n)}(u)$ as $u$.  Thus,
\[
\begin{split}
 \hLp_\ve^n \psi(x) & - \hLp_\ve^n \psi(y)   = \sum_{u \in \hF_\ve^{-n}(x)} (\psi(u) - \psi(v)) e^{S_n\Phi_\ve(u)}
			\; + \sum_{v \in \hF_\ve^{-n}(y)} \psi(v) \big(e^{S_n\Phi_\ve(u)} - e^{S_n\Phi_\ve(v)}\big) \\
			& \le \sum_{u \in \hF_\ve^{-n}(x)} H^\eta(\psi) d(u,v)^\eta \tfrac{1+C_d}{q} \hm\left(Y_i^{(n)}(u)\right) 
			 \; + \sum_{v \in \hF_\ve^{-n}(y)} |\psi(v)| \, \hm\!\left(Y_i^{(n)}(v)\right) \tfrac{1 + C_d}{q} 
			C_d d(x,y)^\eta,
\end{split}
\]
where we have used the bounded distortion property \ref{GM3} as well as \eqref{eq:distortion}.  Now using
the regularity of $\psi$ as well as the expanding property \ref{GM2}, for any $v \in \hF_\ve^{-n}(y)$,
\begin{equation}
\label{eq:average}
\left| |\psi(v)| - \frac{1}{\hm\left(Y_i^{(n)}(v)\right) } \int_{Y_i^{(n)}(v)} |\psi| \, d\hm \right|
\le H^\eta(\psi) \ \diam\left(Y_i^{(n)}(v)\right)^\eta \le H^\eta(\psi) C_e^{-\eta} \sigma^{-\eta n}  \, .
\end{equation}
Putting these estimates together, we obtain
\[
\begin{split}
\left|  \hLp_\ve^n \psi(x)  - \hLp_\ve^n \psi(y) \right|
			& \le C_e^{-\eta } \sigma^{-\eta n} H^\eta(\psi) d(x,y)^\eta \sum_{u \in \hF_\ve^{-n}(x)}  \tfrac{1+C_d}{q} (1+C_d) 
			\hm \!\left(Y_i^{(n)}(u)\right) \\
			& \quad + \sum_{v \in \hF_\ve^{-n}(y)} \int_{Y_i^{(n)}(v)} |\psi| \, d\hm \tfrac{1 + C_d}{q} 
			C_d d(x,y)^\eta.
\end{split}
\]
Due to the fact that the hole respects the Markov structure of our inducing scheme, it follows that
$\cup_{u \in \hF_\ve^{-n}(x)} Y_i^{(n)}(u) \subset \hY_\ve^{n-1}$ allowing us to evaluate both sums.
Now dividing through by $d(x,y)^\eta$ and taking the appropriate suprema yields 
the required inequality in \eqref{eq:strong ly} for the H\"older constant of
$\hLp_\ve^n \psi$ with $C = C_e^{-\eta } \frac{(1+C_d)^2}{q}$. 

Next, we estimate $|\hLp_\ve^n \psi|_\infty$.  Let $Q \in \Q$ and $x  \in Q$.   Now,
\begin{equation}
\label{eq:C0 bound}
|\hLp_\ve^n \psi(x)|  \le \sum_{u \in \hF_\ve^{-n}(x)} |\psi(u)| e^{S_n\Phi_\ve(u)}
\le \sum_{u \in \hF_\ve^{-n}(x)} \tfrac{1+C_d}{q} |\psi(u)|\,  \hat m \!\left( Y_i^{(n)}(u) \right)
\end{equation}
where we have used \eqref{eq:distortion} for the second inequality.  Using
\eqref{eq:average}, we estimate,
\[
\begin{split}
|\hLp_\ve^n \psi(x)| & \le \tfrac{1+C_d}{q} \sum_{u \in \hF_\ve^{-n}(x)} H^\eta(\psi) C_e^{-\eta} \sigma^{-\eta n} \hat m \!\left( Y_i^{(n)}(u) \right) + \int_{Y_i^{(n)}(u)} |\psi| \, d\hat m \\
& \le C' \sigma^{-\eta n}  H^\eta(\psi) \hat m(\hY_\ve^{n-1}) +  \int_{\hY_\ve^{n-1}} |\psi| \, d\hat m \, ,
\end{split}
\]
 so \eqref{eq:strong ly} holds 
with $C = 2 C_e^{-\eta } \frac{(1+C_d)^2}{q}$,  completing the proof of the proposition.
\end{proof}

Define the norm for $\Lp_\ve : \cont^\eta(\Q) \to L^1(\hm)$ by 
\[
\vertiii{ \Lp_\ve} = \sup \left\{ |\Lp_\ve \psi |_{L^1(\hm)} :  \| \psi \|_{\cont^\eta} \le 1 \right\} \, .
\]

\begin{lemma}
\label{lem:close}
For any $\delta > 0$, there exists $\ve_\delta > 0$ such that for all $\ve \in (0, \ve_\delta)$,
$\vertiii{\Lp_0 - \Lp_\ve} \le \delta$.
\end{lemma}

\begin{proof}
Fix $\delta > 0$.
Define $\G_\ve = \{ Y_i \in \mathcal{Y}_\ve : \hat f^k_{z, \ve_0}(Y_i) = \hat f^k_{z, \ve_0, \ve}(Y_i), \forall k = 1, \ldots, R_\ve(Y_i) \}$.  Note that if $Y_i \in \G_\ve$, then $\Phi_0 = \Phi_\ve$ on $Y_i$.

Next
define $B_\ve = \{ y \in Y : Y_{i, \ve}(y) \notin \G_\ve \}$,
where $Y_{i, \ve}(y)$ is the 1-cylinder with respect to $F_\ve$ containing $y$.
For $\psi \in \cont^\eta(\Q)$ and $x \in Y$,
\begin{equation}
\label{eq:error sum}
\begin{split}
|(\Lp_0  - \Lp_\ve) \psi(x)| & = \left| \sum_{y \in F_0^{-1}x} \psi(y)e^{\Phi_0(y)} - \sum_{y \in F_\ve^{-1}x} \psi(y) e^{\Phi_\ve(y)} \right| \\
& \le |\psi|_{\infty} \sum_{\substack{ y \in F_0^{-1}x \\ y \in B_\ve}} \tfrac{1+C_d}{q} \hm(Y_i(y))
+ |\psi|_{\infty} \sum_{\substack{ y \in F_\ve^{-1}x \\ y \in B_\ve}} \tfrac{1+C_d}{q} \hm(Y_i(y)).
\end{split}
\end{equation}
By the proof of Corollary~\ref{cor:large connection}, the total mass of 1-cylinders where
$F_0$ and $F_\ve$ do not agree can be made arbitrarily small.

Let $\delta' = \delta \frac{q}{2(1+C_d) \hm(Y)}$.  Choose $\ve_\delta > 0$ such that 
$\hm(B_{\ve_\delta}) \le \delta'$ by Corollary~\ref{cor:large connection}.  Then
\begin{equation}
\label{eq:infinity close}
|(\Lp_0  - \Lp_\ve) \psi(x)| \le |\psi|_\infty 2 \tfrac{1+C_d}{q} \hm(B_\ve) \le |\psi |_\infty \delta / \hm(Y).
\end{equation}
Integrating over $x \in Y$ proves the lemma:
$\displaystyle \int | (\Lp_0 - \Lp_\ve) \psi | \, d\hm \le | \psi |_\infty \delta$.
\end{proof}

\begin{corollary}
\label{cor:spectral gap}
There exists $\ve_1 \in (0, \ve_0]$ such that the family of operators $\Lp_\ve$, $\ve \in [0, \ve_1)$, acting on $\cont^\eta(\Q)$ have a uniform
spectral gap.  There exists $\beta >0$ such that $\Lp_\ve$ admits the following spectral decomposition
for all $\ve \in [0, \ve_1)$: 
There exist $G_\ve \in \cont^\eta(\Q)$,
a linear functional $e_\ve : \cont^\eta(\Q) \to \R$ and an operator $\cR_\ve : \cont^\eta(\Q) \circlearrowleft$
such that
\[
\Lp_\ve = G_\ve \otimes e_\ve + \cR_\ve, \quad \mbox{and } \cR_\ve G_\ve = 0 .
\]
The spectral radius of $\cR_\ve$ is at most $e^{-\beta}$ and
$e_\ve(\psi) = \int_Y \psi \, d\hm$ for all $\psi \in \cont^\eta(\Q)$.

Moreover, $G_\ve \to G_0$ in $L^1(\hm)$ and  $\vertiii{\cR_\ve - \cR_0} \to 0$ as $\ve \to 0$.
\end{corollary}

We may normalise the above so that $\hm(G_\ve)=1$, so $\hat\mu_{Y, \ve} = G_\ve \hm$ is the corresponding invariant probability measure for $F_\ve$.

\begin{proof}
The fact that all the operators $\Lp_\ve$, $\ve \in [0, \ve_0)$, are quasi-compact on $\cont^\eta(\Q)$
with essential spectral
radius bounded by $\sigma^{-1}$ follows from Proposition~\ref{prop:ly} and the fact that
the unit ball of $\cont^\eta(\Q)$ is compactly embedded in $L^1(\hm)$.  Moreover, the spectrum of
$\Lp_\ve$ on the unit circle is finite dimensional and forms a cyclic group.

Since $F_0$ is mixing by Lemma~\ref{lem:mixing}, $\Lp_0$ has a single simple eigenvalue at 1 and
the rest of the spectrum of $\Lp_0$ is contained in a disk of radius $e^{-2\beta} \ge \sigma^{-1}$ 
for some $\beta>0$.
Next, by Lemma~\ref{lem:close} and \cite[Corollary 1]{KelLiv99}, the spectrum of
$\Lp_\ve$ outside the disk of radius $\sigma^{-1}$ can be made arbitrarily close to that of $\Lp_0$
by choosing $\ve$ sufficiently small.  Thus we may choose $\ve_1 \in (0, \ve_0]$ such that
the spectrum of $\Lp_\ve$ outside the disk of radius $e^{-\beta}$ consists only of a simple
eigenvalue at 1, for all $\ve \in (0, \ve_1)$.  The closeness of $G_\ve$ and $\cR_\ve$ to
$G_0$ and $\cR_0$ follow similarly from \cite[Corollary 1]{KelLiv99}.
Finally, the fact that $e_\ve(\psi) = \hm(\psi)$ for all $\psi \in \cont^\eta(\Q)$ follows from the conformality of $\hm$.
\end{proof}


\subsection{Spectral properties of the punctured operators $\hLp_\ve$}

Due to the uniform Lasota-Yorke inequalities provided by Proposition~\ref{prop:ly}, it only remains
to show that $\Lp_\ve$ and $\hLp_\ve$ are close in the $\invertiii{ \cdot }$-norm.

\begin{lemma}
\label{lem:good bound}
For any $\ve \in (0, \ve_0)$, $\invertiii{ \Lp_\ve - \hLp_\ve} \le \hm(\hat H_\ve')$.
\end{lemma}

\begin{proof}
The proof is immediate using the definition of $\hLp_\ve$ and the conformality of $\hm$,
\begin{equation}
\label{eq:good bound}
\int (\Lp_\ve - \hLp_\ve) \psi \, d\hm = \int \Lp_\ve (1_{Y \setminus \hY^1_\ve} \psi) \, d\hm
= \int_{\hat H_\ve'} \psi \, d\hm \le | \psi |_\infty \hm(\hat H_\ve') ,
\end{equation}
since $\hat H_\ve' = Y \setminus \hY_\ve$.
\end{proof}

\begin{corollary}
\label{cor:spectral decomp}
There exists $\ve_2 \le \ve_1$ such that for all $\ve \in (0, \ve_2)$, the operators
$\hLp_\ve$ have a uniform spectral gap:  There exist $\Lambda_\ve \in (e^{-\beta/3}, 1)$, 
$\bhg_\ve \in \cont^\eta(\Q)$,
a functional $\he_\ve : \cont^\eta(\Q) \to \mathbb{R}$, and an operator 
$\hcR_\ve : \cont^\eta(\Q) \circlearrowleft$ such that
\begin{equation}
\label{eq:decomp}
\hLp_\ve = \Lambda_\ve \bhg_\ve \otimes \he_\ve + \hcR_\ve, \quad \mbox{and } \hcR_\ve \bhg_\ve = 0.
\end{equation}
The spectral radius of $\hcR_\ve$ is at most $e^{-2\beta/3} < \Lambda_\ve$.

Moreover, $\Lambda_\ve \to 1$, $\bhg_\ve \to G_0$ in $L^1(\hm)$ and $\invertiii{\hcR_\ve - \cR_0} \to 0$ as $\ve \to 0$.
\end{corollary}

\begin{proof}
Lemmas~\ref{lem:close} and \ref{lem:good bound} together with the triangle inequality show that
$\hLp_\ve$ and $\Lp_0$ are close in the $\invertiii{\cdot }$-norm.
The uniform Lasota-Yorke inequalities given by Proposition~\ref{prop:ly} together with
\cite[Corollary 1]{KelLiv99} imply that the spectrum (and corresponding spectral projectors)
of $\hLp_\ve$ outside the disk of radius $e^{-\beta}$ are close to those of $\Lp_0$.  
Without requiring a rate of approach, we may choose $\ve_2 > 0$ with the stated properties.
\end{proof}

We may normalise $\bhg_\ve$ and $\he_\ve$ so that $\hm(\bhg_\ve) = 1$ and $\he_\ve(\bhg_\ve)=1$,
so that $\hLp_\ve \bhg_\ve = \Lambda_\ve \bhg_\ve$.


\section{Young towers and proof of Theorem~\ref{thm:accim}}
\label{sec:Young tower}

The Markov structure of the return map $F_\ve = F_{z, \ve_0, \ve}$ to $Y$ immediately implies the existence of another,
related extension, called a Young tower.  These have been well-studied in the context of open systems,
so we will recall their structure in order to apply some results in our setting.

As in Section~\ref{sec:gap}, let $R_\ve = R_{z, \ve_0, \ve}$.  Define the Young tower over $Y$ with return time $R_\ve$ by,
\[
\Delta := \{ (y, \ell) \in Y \times \mathbb{N} : \ell < R_\ve(y) \} .
\]
We view $\Delta$ as a tower with $\Delta_\ell = \{ (y, n) \in Y \times \mathbb{N} : n = \ell \}$ as the $\ell$th level. 
The dynamics on the tower is defined by $f_\Delta(y, \ell) = (y, \ell+1)$ when $\ell + 1 < R_\ve(y)$, and
$f_\Delta(y, \ell) = (F_\ve(y), 0)$ otherwise.  Thus $\Delta_0$ corresponds to $Y$ and $F_\ve = f_\Delta^{R_\ve}$ can be viewed
as the first return map to $\Delta_0$.  With this definition, there is a natural projection 
$\hatpi_\Delta : \Delta \to \hatI$ satisfying $\hatpi_\Delta \circ f_\Delta = \hatf_\ve \circ \hatpi_\Delta$.
Then also defining $\pi_\Delta = \pi \circ \hatpi_\Delta : \Delta \to I$, we have
$\pi_\Delta \circ f_\Delta = f \circ \pi_\Delta$.

Clearly, $\Delta = \Delta(z, \ve_0, \ve)$ depends on $z$, $\ve_0$ and $\ve$ through the construction of 
$\hatI_{z, \ve_0, \ve}$, $Y = \hat I'_{z, \ve_0, \ve}(L)$ and $F_\ve$.  However, since we fix these three parameters in this section, we 
will drop explicit mention of this dependence in the notation we use for objects associated with $\Delta$.

The map $f_\Delta$ inherits a Markov structure as follows.  On $\Delta_0$, we use the elements of 
the finite partition $\mathcal{Q}$ as our partition elements, 
labelling them by $\Delta_{0,i}$.  On $\Delta_\ell$, $\ell \ge 1$, we define
$\Delta_{\ell, i} = f_\Delta^\ell(Y_i)$, $Y_i \in \mathcal{Y}_\ve$.  The collection $\{ \Delta_{\ell, i} \}_{\ell, i \ge 0}$ forms a 
countable Markov partition for $f_\Delta$.  Since at return times to $\Delta_0$, 
$f_\Delta$ maps the image of each 1-cylinder $Y_i$ to an element of the finite partition $\Q$ of
$Y = \Delta_0$, we will view $(f_\Delta, \Delta)$ as a Young tower with finitely many bases.  The partition 
$\{ \Delta_{\ell,i} \}$ is generating since $\{ Y_i \}_i$ is a generating partition for $F_\ve$.  
Moreover, the first return time $R_\ve$ to $\Delta_0$ under $f_\Delta$ is the same
as the first return to $Y$ under $\hat f_\ve$.

We make $\Delta$ into a metric space by defining a symbolic metric based on the Markov partition.
Let $R^n_\ve(x)$ denote the $n$th return time of $x$ to $\Delta_0$.  Define the {\em separation time}
on $\Delta_0$ by,
\[
s(x,y) = \min \left\{ n \ge 0 : f_\Delta^{R_\ve^n}(x) \text{ and } f_\Delta^{R_\ve^n}(y)
\mbox{ lie in different elements of $\mathcal{Y}_\ve$} \right\} \, .
\]
We extend the separation time to all of $\Delta$ by setting $s(x,y) = s(f_\Delta^{-\ell} x, f_\Delta^{-\ell} y)$ for $x, y \in \Delta_\ell$.
It follows that  
$s(x,y)$ is finite almost everywhere since $\{ \Delta_{i,j} \}$ is a generating partition.
For $\theta>0$, define a metric on $\Delta$ by
$d_\theta(x,y) = e^{-\theta s(x,y)}$.  We will choose $\theta$ according to property (P3) in Section~\ref{tower transfer}.
 
Given our (normalised) potential $\hat \vf$ on $\hatI$, and $\hat \vf$-conformal measure 
$\hat m = \hat m_\vf$, 
we define a reference measure 
$m_\Delta$ on 
$\Delta$ by setting $m_\Delta = \hat m$ on $\Delta_0$, and 
$m_\Delta|_{\Delta_\ell} := (f_\Delta)_*(m_\Delta|_{\Delta_{\ell-1} \cap f_\Delta^{-1}\Delta_\ell})$.

Similarly, we lift the potential $\hat \vf$ to a potential $\vf_\Delta$ on $\Delta$ as follows.
For $x \in \Delta_\ell$, let $x^- = f_\Delta^{-\ell}(x)$ denote the pullback of $x$ to $\Delta_0$.
Then,
\[
\vf_\Delta(x) := S_{R_\ve}\hat \vf(x^-) \mbox{ for } x \in f_{\Delta}^{-1}(\Delta_0) \mbox{ and }
\vf_\Delta = 0 \mbox{ on } \Delta \setminus f_\Delta^{-1}(\Delta_0) .
\]
With this definition, $m_\Delta$ is a $\vf_\Delta$-conformal measure.

We may also define a related invariant measure on $\Delta$.  Let $G_\ve \in \cont^\eta(\Q)$ be the invariant
density from Corollary~\ref{cor:spectral gap}.  Define 
\begin{equation}
\label{eq:gD def}
g_\Delta = G_\ve \mbox{ on } \Delta_0, \mbox{ and } g_\Delta(x) = G_\ve(x^-) \mbox{ for } x \in \Delta_\ell, \ell \ge 1,
\end{equation}
where $x^-$ is defined as above.

It follows that the measure $d\mu_\Delta = g_\Delta dm_\Delta/ \int_\Delta g_\Delta dm_\Delta$ is 
an invariant probability measure for $f_\Delta$.
Moreover, we have $(\hatpi_\Delta)_* \mu_\Delta = \hat \mu_\ve$. 
And since $\pi_* \hat \mu_\ve = \mu_\vf$, we have also that $(\pi_\Delta)_* \mu_\Delta = \mu_\vf$.
Note that here $\hat\mu_\ve$ is defined on $\hat I_{z, \ve_0, \ve}$ and depends on
$\ve$, while $\mu_\vf$ does not.

We lift the hole $H = H(z,\ve)$ to $\Delta$ by settting $H_\Delta := \pi_\Delta^{-1}H = \hatpi_\Delta^{-1}\hat H$.  Due to the construction
of $\hat I_{z, \ve_0, \ve}$, $H_\Delta$ comprises a countable collection of elements of the Markov
partition $\Delta_{\ell,j}$, which we shall denote by $H_{\ell,j}$.  Set $\hDelta = \Delta \setminus H_\Delta$,
and define the open system $\hf_\Delta = f_\Delta|_{\hDelta}$.

\begin{lemma}
\label{lem:same escape}
Define $\hDelta^{(n)} = \cap_{i=0}^{n} f_\Delta^{-1}(\hDelta)$.  Then,
\[
\log \overline\lambda_\ve := \limsup_{n \to \infty} \frac 1n \log \mu_\vf(\hI^n)
= \limsup_{n \to \infty} \frac 1n \log \mu_\Delta(\hDelta^{(n)}) 
 = \limsup_{n \to \infty} \frac 1n \log m_\Delta(\hDelta^{(n)}).
\]
\end{lemma}
\begin{proof}
The first equality follows immediately from the fact that $(\pi_\Delta)_*\mu_\Delta = \mu_\vf$ and
$\pi_\Delta \circ f_\Delta = f \circ \pi_\Delta$, so that $\mu_\Delta(\hDelta^{(n)}) = \mu_\vf(\hI^n)$
for each $n$.
The second equality follows from the fact that $\mu_\Delta = g_\Delta m_\Delta$,
and $g_\Delta$ is bounded (uniformly in $\ve$) 
away from 0 and $\infty$ on $\Delta$ by \eqref{eq:gD def} and
Lemma~\ref{lem:reg} below.
\end{proof}

Our final lemma of this subsection says that the open system $\hf_\Delta$ is mixing\footnote{Mixing for an open system is 
not generally defined, and topologically transitivity does not hold  
unless we restrict to the survivor set $\hDelta^{(\infty)} = \cap_{n=0}^\infty \hDelta^{(n)}$.
In the open systems context, a mixing property can be formulated
in terms of transitions between elements
of the Markov partition $\{ \Delta_{\ell,j} \}$, after removing those elements which lie above components of $H_\Delta$ in $\Delta$.} on partition elements under our assumptions on 
$f$ and our construction of $\hatf_{z, \ve_0, \ve}$.

Let $\ve_1^* = \min \{ \ve_1(Q_1, Q_2) : Q_1, Q_2 \in \Q \} > 0$, where $\ve_1(Q_1, Q_2)$ is from Lemma~\ref{lem:uniform mix}.

\begin{lemma}
\label{lem:mixing D}
For all $\ve \in (0, \ve_1^*)$, the open system $(\hf_\Delta, \hDelta)$ is transitive and 
aperiodic
on elements of $\{ \Delta_{\ell,j} \}$ that do not lie above
a component of $H_\Delta$.
\end{lemma}

\begin{proof}
Transitivity of $f_\Delta$ on elements of the Markov partition
is guaranteed by the transitivity of $F_{z, \ve_0, \ve}$, proved in Lemma~\ref{lem:mixing}.
That this property carries over to the open map $\hf_\Delta$ follows from Lemma~\ref{lem:uniform mix}.  Considering
Case 2 in the proof of that lemma, we see that for $\ve \in (0, \ve_1)$, the orbit of the desired interval $J$
connecting $Q_1$ to $Q_2$ is disjoint from $H_\ve$.  Thus the connection holds for the open system $(\hf_\Delta, \hDelta)$.

Next we show that $\hf_\Delta$ is aperiodic.
Due to the structure of the tower map, it suffices to show that there exists $n_0 \in \mathbb{N}$ such that for all 
$n \ge n_0$, $\hf_\Delta^n(\Delta_0) \supset \Delta_0$.  Since returns to $\Delta_0$ must be to one of the finitely
many elements of the partition $\Q$, this property is in turn implied by the following claim:
For all $Q \in \Q$, there exists $n_Q \in \mathbb{N}$ such that $\hf_\Delta^{n_Q}(Q) \supset Q$ and 
$\hf_\Delta^{n_Q+1}(Q) \supset Q$.
We proceed to prove the claim, which is a refinement of the proof of Lemma~\ref{lem:mixing}.

Let $Q \in \Q$.  Since $f$ is leo, there exists $\bar n \in \mathbb{N}$ such that $f^n(\pi(Q)) \supset I$ for all $n \ge \bar n$.
Thus as in the proof of Lemma~\ref{lem:mixing}, $m_\vf(\pi_\Delta(f_\Delta^n(Q) \cap \Delta_0)) \ge 1- \kappa$, by choice
of $L$.  Applying this to $n = \bar n$ and $n = \bar n + 1$, and recalling that we identify $Y = \hatI_{z, \ve_0}'$ with $\Delta_0$,
we obtain, 
\[
m_\vf \left(\pi_\Delta(\Delta_0 \cap f_\Delta^{\bar n}(Q)\right) \cap \pi_\Delta\left(\Delta_0 \cap f_\Delta^{\bar n +1}(Q))\right) \ge 1 - 2\kappa >0 .
\]
Thus there must exist intervals $O_1 \subset \Delta_0 \cap f_\Delta^{\bar n}(Q)$ and 
$O_2 \subset \Delta_0 \cap f_\Delta^{\bar n +1}(Q)$ such that $\pi(O_1) \cap \pi(O_2) \neq \emptyset$.  By 
Lemma~\ref{lem:fibres meet}, there exists $n_1 \in \mathbb{N}$ such that 
$\hatf_\ve^{n_1}(O_1) \cap \hatf_\ve^{n_1}(O_2) \neq \emptyset$, and we can choose this time $n_1$ so that this intersection
occurs in $Y=\Delta_0$.  This implies that also $f_\Delta^{n_1}(O_1) \cap f_\Delta^{n_1}(O_2) \neq \emptyset$.

Now using the transitivity of $F_\ve$, as well as its Markov property, there exists $k \in \mathbb{N}$ such that 
$F_\ve^k(f_\Delta^{n_1}(O_1) \cap f_\Delta^{n_1}(O_2)) \supset Q$.  Let $r_k$ denote the number of iterates of $f_\Delta$
contained in $F_\ve^k$ on this set.
This implies that both $f_\Delta^{r_k + n_1 + \bar n}(Q) \supset Q$, and 
$f_\Delta^{r_k + n_1 + \bar n +1}(Q) \supset Q$.

As a final step, we invoke Lemma~\ref{lem:uniform mix} as earlier.  We have constructed two times $k_1$ and $k_2$ for which
$F_\ve^{k_j}(Q) \supset Q$, $j=1,2$.  By case 2 of the proof of Lemma~\ref{lem:uniform mix}, for $\ve < \ve_1(Q,Q)$,
these connections still occur in the open system.  Thus we conclude
that both $\hf_\Delta^{r_k + n_1 + \bar n}(Q) \supset Q$, and 
$\hf_\Delta^{r_k + n_1 + \bar n +1}(Q) \supset Q$, as required.
\end{proof}


\subsection{Transfer Operator on $\Delta$ and a Spectral Gap}
\label{tower transfer}

In order to study the dynamics on the open tower, we define the transfer operator associated with the
potential $\vf_\Delta$,
\[
\Lp_\Delta \psi = \sum_{y \in f_\Delta^{-1}x} \psi(y) e^{\vf_\Delta(y)},
\]
and its usual punctured counterpart for the open system, $\hLp_\Delta \psi = \Lp_\Delta (\psi \cdot 1_{\hDelta^{(1)}})$.  We also define the corresponding punctured potential on the
tower by $\vf_\Delta^{H_\ve} = \vf_\Delta$ on $\hDelta$ and $\vf_\Delta^{H_\ve} = -\infty$
on $H_\Delta$. 

We will prove that for sufficiently small holes $H_\ve$, the transfer operator $\hLp_\Delta$ has a spectral gap on a certain Banach space $\B$, using the abstract result \cite[Theorem~4.12]{DemTodMP}.  Note that this result is not perturbative, but
rather relies on checking four explicit conditions (P1)-(P4) from \cite[Section~4.2]{DemTodMP}.
They are as follows.

\medskip
\noindent
(P1) {\em Exponential Tails. }  This follows from Theorem~\ref{thm:tails}, since by definition of $m_\Delta$,
\[
m_\Delta(\Delta_n) = m_\Delta(\Delta_0 \cap \{ R_\ve > n \}) 
= \hat m_\vf(Y \cap \{ R_\ve > n \})  \le C e^{-\alpha n} \, ,
\]
where $C$ and $\alpha$ are uniform for $\ve < \ve_0$.

\medskip
\noindent
(P2) {\em Slow Escape. } $ - \log \overline\lambda_\ve < \alpha$.  This can be guaranteed by noting that
$\overline\lambda_\ve \ge \Lambda_\ve$, where $\Lambda_\ve<1$ is from 
Corollary~\ref{cor:spectral decomp}.   This inequality is due to the fact that the escape from the
induced system cannot be slower than the escape from the uninduced system.
The requirement on the upper escape rate in \cite{DemTodMP}  is defined in terms of $m_\Delta$,
which in our case is equal to $\log \overline\lambda_\ve$ by Lemma~\ref{lem:same escape}.
Again using Corollary~\ref{cor:spectral decomp},  there exists $\ve^*>0$ such that 
$\Lambda_\ve > e^{-\alpha}$ for all $\ve \in (0, \ve^*)$.  This guarantees (P2).

\medskip
\noindent
(P3) {\em Bounded Distortion and Lipschitz Property for $e^{\vf_\Delta}$.}  
The potential $\vf_\Delta = 0$ on $\Delta \setminus f_\Delta^{-1}(\Delta_0)$ so we need only
to verify this property at return times.  This follows from Lemma~\ref{lem:distort} and the
following estimate linking the Euclidean metric on $I$ with the separation time metric
on $\Delta$.  If $s(x,y) = n$, then $F_\ve^i(\hat \pi_\Delta(x))$ and $F_\ve^i(\hat \pi_\Delta(y))$ lie in the same element 
of $\mathcal{Y}_\ve$
for each $i < n$, and $F_\ve^n(\hat \pi_\Delta(x))$ and $F_\ve^n(\hat \pi_\Delta(y))$ lie in the same element of $\Q$.  Then since $DF_\ve^n \ge C_e\sigma^n > 1$, 
\begin{equation}
\label{eq:metric}
\frac{|\hat\pi_\Delta(x) - \hat\pi_\Delta(y)|^\eta}{d_\theta(x,y)} = \frac{|\hat\pi_\Delta(x) - \hat\pi_\Delta(y)|^\eta}{e^{-\theta s(x,y)}} 
\le \frac{C_e^\eta\sigma^{-\eta n}}{e^{-\theta n}} \, .
\end{equation}
Choosing $\theta < \eta \log \sigma$
guarantees that a $\eta$-H\"older continuous function on $I$ (and $\hat I_{z, \ve_0, \ve}$) 
lifts to a Lipschitz function on 
$\Delta$.  Then Lemma~\ref{lem:distort}(a) implies the required bounded distortion 
for $\vf_\Delta$. 

\medskip
\noindent
(P4) {\em Subexponential Growth of Potential:}  For each $\delta>0$, there exists $C>0$
such that
\[
|S_{R_\ve} \vf_\Delta(x) | \le C e^{\delta R_\ve(x)} \quad
\mbox{for all $x \in \Delta_0$.} 
\]
This is immediate for H\"older continuous potentials since $\vf$ is bounded so 

\[
|S_{R_\ve} \vf_\Delta(x)| \le R_\ve(x) |\vf|_\infty \, \quad \mbox{for all $x \in \Delta_0$}.
\]

For geometric potentials, $\vf = -t\log |Df| - p_t$, (P4) is
guaranteed by the uniform expansion of $F_\ve$ at return times, noting that
\[
S_{R_\ve}\vf_\Delta(x) = S_{R_\ve} \vf(\hat \pi_\Delta(x)) 
= - t \log |DF_\ve(\hat \pi_\Delta(x))| - R_\ve(x) p_t \, .
\]
By \ref{GM2}, $C_e\sigma \le |DF_\ve| \le (\sup |Df|)^{R_\ve}$, and
since $\sup |Df|<\infty$,  we have
\[
|S_{R_\ve}\vf_\Delta(x)| \le R_\ve(x) |p_t| + t \max \{ |\log (C_e \sigma)|, R_\ve(x) \log |Df|_\infty \}
\quad \mbox{for all $x \in \Delta_0$.} 
\]


With (P1)-(P4) verified, we are prepared to study the action of $\hLp_\Delta$ on an
appropriate function space.  Using (P2), choose $\beta$ such that
$- \log \overline\lambda_\ve < \beta < \alpha$.  Define a weighted $L^\infty$ norm on $\Delta$ by,
\[
\| \psi \|_\infty = \sup_\ell e^{-\beta \ell}  \sup \{ |\psi(x)| : x \in \Delta_\ell \} \, ,
\] 
as well as the weighted Lipschitz norm,
\[
|\psi|_{Lip} = \sup_\ell e^{-\beta \ell} \sup \left\{ e^{-\theta s(x,y)}|\psi(x)-\psi(y)| : x,y \in \Delta_\ell \right\} \, .
\]
Then define $\B = \{ \psi \in L^1(m_\Delta) : \| \psi \|_\B < \infty \}$, where
$\| \psi \|_\B = \| \psi \|_\infty + |\psi|_{Lip}$. 
We define $\B_0 \subset \B$ to be the set of bounded functions on $\Delta$ whose
Lipschitz constant is also bounded, i.e., $\B_0$ uses the same definition as $\B$, but with $\beta = 0$.
Recall $\ve_1^*>0$ from Lemma~\ref{lem:mixing D} and $\ve^*>0$ from the verification of (P2).

\begin{theorem}{\em (\cite[Theorem~4.12]{DemTodMP})}
\label{thm:tower gap}
Since the open system $(f_\Delta, \Delta; H_\Delta)$ is mixing on partition elements and satisfies properties (P1)-(P4),
we conclude that $\hLp_\Delta$ has a spectral gap on $\B$ for all $\ve < \min \{ \ve^*, \ve_1^*\}$.  
Let $\lambda_\ve$ denote the largest eigenvalue of
$\hLp_\Delta$ and let $\hg_\Delta$ denote the corresponding normalised eigenfunction.
\begin{enumerate}
  \item[(a)]  The escape rate with respect to $m_\Delta$ exists and equals  
  $- \log \lambda_\ve$. 
\item[(b)] $ \displaystyle
\log \lambda_\ve = \sup \left\{ h_{\vartheta}(f_\Delta) + \int_{\hDelta} \vf_\Delta^{H_\ve} d\vartheta 
: \vartheta \in \mathcal{M}_{f_\Delta} , \vartheta(-\vf_\Delta^H) < \infty \right\}$, where $\mathcal{M}_{f_\Delta}$ 
is the set of $f_\Delta$-invariant probability measures on $\Delta$.
\item[(c)]   The following limit defines a probability measure $\nu_\Delta$, 
supported on $\cap_{n = 0}^\infty \hDelta^{(n)}$,
\[
\nu_\Delta(\varphi) = \lim_{n\to \infty} \lambda_\ve^{-n} \int_{\hDelta^{(n)}} \psi \, \hg_\Delta \, dm_\Delta
\qquad \mbox{for all $\psi \in \B_0$. }   
\]
Moreover, the measure 
$\nu_\Delta$ is the unique measure in $\mathcal{M}_{f_\Delta}$ that
attains the supremum in (b), i.e., it is the unique equilibrium state for
$\vf_\Delta^{H_\ve}$.

  \item[(d)] There exist constants
$D>0$ and $\sigma_0 <1$ such that for all $\psi \in \B$,
\[
\| \lambda_\ve^{-n}\hLp^n_\Delta \psi - d(\psi)\hg_\Delta \|_\B \leq D\|\psi \|_\B \sigma_0^n, \; \; \;
\mbox{where }
d(\psi) = \lim_{n\to \infty} \lambda_\ve^{-n} \int_{\hDelta^{(n)}} \psi \, dm_\Delta < \infty .
\]
Also, for any $\psi \in \B$ with $d(\psi) > 0$,
\[
\left| \frac{\hLp^n_\Delta \psi}{|\hLp^n_\Delta \psi|_{L^1(m_\Delta)}} - \hg_\Delta \right|_{L^1(m_\Delta)} 
\le D \| \psi \|_\B \sigma_0^n.
\]
\end{enumerate}
\end{theorem}


\subsection{Proof of Theorem~\ref{thm:accim}}
\label{accim pf}

In this section, we will prove the items of Theorem~\ref{thm:accim} using Theorem~\ref{thm:tower gap}.
The following lemma will allow us to lift H\"older continuous functions on $I$ to Lipschitz functions on $\Delta$.

\begin{lemma}
\label{lem:smooth lift}
Suppose $\theta/ \log \sigma \le \varsigma \le 1$, where $\sigma > 1$ is from \ref{GM2}.
Let $\psi \in C^\varsigma(I)$ and define $\tpsi$ on $\Delta$ by $\tpsi = \psi \circ \pi_\Delta$.
Then $| \tpsi |_\infty \le |\psi|_\infty$ and Lip$(\tpsi) \le C|\psi|_{C^\varsigma(I)}$ for some constant
$C$ depending on the minimum length of elements of $\Q$.
\end{lemma}

\begin{proof}
The bound $| \tpsi |_\infty \le |\psi|_\infty$ is immediate.  To prove the bound on the Lipschitz
constant of $\tpsi$, suppose $x, y \in \dlj$ and estimate,
\[
\frac{|\tpsi(x) - \tpsi(y)|}{d_\theta(x,y)}
= \frac{|\psi(\pi_\Delta(x))- \psi(\pi_\Delta(y))|}{|\pi_\Delta(x) - \pi_\Delta(y)|^\varsigma} 
\cdot \frac{|\pi_\Delta(x) - \pi_\Delta(y)|^\varsigma}{|\pi_\Delta(f_\Delta^{R_\ve}x) - \pi_\Delta(f_\Delta^{R_\ve}y)|^\varsigma}
\cdot \frac{|\pi_\Delta(f_\Delta^{R_\ve}x) - \pi_\Delta(f_\Delta^{R_\ve}y)|^\varsigma}{e^{-\theta s(x,y)}}.
\]
The first ratio above is bounded by $|\psi|_{C^\varsigma(I)}$. 
The second ratio is bounded due to bounded distortion and
the backward contraction condition (PolShr)$_\beta$
at return times to $Y$.
 For the third ratio, we use \eqref{eq:metric}, recalling that the
separation time only counts returns to $\Delta_0$, and that 
$\theta \le \varsigma \log \sigma$.
\end{proof}

In order to project densities from $\Delta$ to $I$, for $\psi \in L^1(m_\Delta)$, and $x \in I$, define
\begin{equation}
\label{eq:PD def}
\P_\Delta \psi(x) = \sum_{y \in \pi_\Delta^{-1}x} \frac{\psi(y)}{J\pi_\Delta(y)} ,
\end{equation}
where $J\pi_\Delta$ is the Jacobian of $\pi_\Delta$ with respect to the measures $m_\vf$ and $m_\Delta$.
Note that for $y \in \Delta_\ell$, with $y = f^\ell_\Delta(z)$ for $z \in \Delta_0$, the conformality of $m_\vf$ implies,
\begin{equation}
\label{eq:lip J}
\frac{1}{J\pi_\Delta(y)} = \frac{dm_\Delta(y)}{dm_\vf(\pi_\Delta y)} = \frac{dm_\Delta(y)}{dm_\vf(\pi_\Delta(f^\ell_\Delta z))}
= \frac{dm_\vf(\pi_\Delta z)}{dm_\vf(f^\ell(\pi_\Delta z))} = e^{S_\ell \vf(\pi_\Delta z)} \, .
\end{equation}
Then the proof of Lemma~\ref{lem:smooth lift} implies that $1/J\pi_\Delta$ is Lipschitz continuous on each $\Delta_{\ell,j}$
with Lipschitz constant depending only on the level $\ell$. 

It follows from the definition of $m_\Delta$ that $\P_\Delta \psi \in L^1(m_\vf)$, and 
$\int_I \P_\Delta \psi \, dm_\vf = \int_\Delta \psi \, dm_\Delta$.  Moreover, 
\begin{equation}
\label{eq:commute}
\hLp_{\vf^{H_\ve}}^n (\P_\Delta \psi) = \P_\Delta( \hLp_\Delta^n \psi) \, , \mbox{ for each $n \in \mathbb{N}$.}
\end{equation}

The final step in translating Theorem~\ref{thm:tower gap} to Theorem~\ref{thm:accim} is the following.

\begin{lemma}
\label{lem:good lift}
$C^\varsigma(\hI) \subset \P_\Delta \B_0$ for all $\varsigma \ge \theta/\log \sigma$.
\end{lemma}

\begin{proof}
Let $Q\in \Q$ and note that by the leo property there exists $N\in \N$ such that $f^N(\pi(Q))=I$.  
This implies that $\pi(\hat I_{z,\ve_0, \ve}(N)) = I$, where $\hat I_{z,\ve_0,\ve}(N)$
denotes the first $N$ levels of $\hat I_{z, \ve_0, \ve}$ as in Section~\ref{ssec:puncture}.
This in turn implies that $\pi_\Delta(\cup_{\ell \le N} \Delta_\ell) = I$ (mod 0 with respect to $m_\vf$).  

Next, we select a collection $\mathcal{K}$ of $\Delta_{\ell,i}$, $\ell \le N$, such that
$\pi_\Delta(\cup_{\Delta_{\ell,i} \in \mathcal{K}} \Delta_{\ell,i}) = I$ and
$\pi_\Delta(\Delta_{\ell,i}) \cap \pi_\Delta(\Delta_{\ell',j}) = \emptyset$ except for at most finitely many pairs
$\Delta_{\ell,i}, \Delta_{\ell',j} \in \mathcal{K}$.  Such a collection exists since $f^N$ has at most finitely many intervals
of monotonicity, so that when the images of two branches overlap, we  may eliminate all the $\Delta_{\ell,i}$ in one branch
from our set $\mathcal{K}$.  The only time when we may be forced to retain two overlapping $\Delta_{\ell,i}$ occurs
at the end of one of the branches of monotonicity.  In this way, we are guaranteed the existence of a set $\mathcal{K}$
with the property that only finitely many elements have projections that overlap. 

With the set $\mathcal{K}$ established, the rest of the proof follows along the lines of \cite[Proposition~4.2]{BDM}.
Essentially, it amounts to inverting the projection operator $\mathcal{P}_{\Delta}$ defined in \eqref{eq:PD def}.

Let $\psi \in C^\varsigma(I)$ be given.  
Define $\tilde \psi \equiv 0$ on $\Delta \setminus \cup_{\Delta_{\ell,i} \in \mathcal{K}} \Delta_{\ell,i}$. 
Next, if $\Delta_{\ell,i} \in \mathcal{K}$ and $\pi_\Delta(\Delta_{\ell,i})$ does not overlap
the projection of
any other $\Delta_{\ell',j} \in \mathcal{K}$, then for $x \in \Delta_{\ell,i}$, we may define
$\tilde\psi(x) = \psi(\pi_\Delta x) J\pi_\Delta(x)$.  It follows that $\P_\Delta \tilde\psi(x) = \psi(\pi x)$ for $x \in \Delta_{\ell,i}$, and
by \eqref{eq:lip J} and Lemma~\ref{lem:smooth lift}, $\tilde \psi$ is Lipschitz with norm depending on the level $\ell$.

Finally, for elements of $\mathcal{K}$ whose projections overlap, we proceed as follows.  Suppose 
$\pi_\Delta(\Delta_{\ell,i}) \cap \pi_\Delta(\Delta_{\ell',j}) \neq \emptyset$.  Let $A = \pi_\Delta(\Delta_{\ell,i}) \cup \pi_\Delta(\Delta_{\ell',j})$
and choose a partition of unity $\{ \rho_1, \rho_2 \}$ for the interval $A$ such that $\rho_1, \rho_2 \in C^\varsigma(A)$,
and $\rho_1 =1$ on $\pi_\Delta(\Delta_{\ell,i}) \setminus \pi_\Delta(\Delta_{\ell',j})$, while $\rho_2 = 1$ on
$\pi_\Delta(\Delta_{\ell',j}) \setminus \pi_\Delta(\Delta_{\ell,i})$.

Define $\tilde \psi$ for $x \in \Delta_{\ell,i}$ by
\[
\tilde \psi(x) = \psi(\pi_\Delta x) J\pi_\Delta(x) \rho_1(\pi x) \, ,
\]
and similarly define $\tilde \psi$ on $\Delta_{\ell',j}$ using $\rho_2$.  It is clear that $\P_\Delta \tilde\psi(y) = \psi(y)$
for $y \in A$.  This construction using partitions of unity $\rho_i$ can be modified to account for finitely may overlaps in
$\pi(\Delta_{\ell,i})$, $\Delta_{\ell,i} \in \mathcal{K}$, while keeping a uniform bound on the $C^\varsigma$-norm of $\rho_i$.

In this way, we define $\tilde \psi$ on $\Delta_{\ell,i}$ for all $\Delta_{\ell,i} \in \mathcal{K}$.  Since
$\pi_\Delta(\cup_{\Delta_{\ell,i} \in \mathcal{K}} \Delta_{\ell,i}) = I$, we have $\P_\Delta \tilde \psi = \psi$ (mod 0).
And since $\mathcal{K}$ contains only elements on level at most $N$, by \eqref{eq:lip J} and Lemma~\ref{lem:smooth lift}, 
$\tilde \psi \in \B_0$.
\end{proof}

We proceed to prove the items of Theorem~\ref{thm:accim}.

Recall that $\eta \in (0,1]$ is the relevant H\"older exponent for $\vf$.  For geometric potentials, we take
$\eta = 1$ due to Lemma~\ref{lem:distort}(b).  Fix $\varsigma \in (0, \eta]$.  Then we may choose
$\theta \le \varsigma \log \sigma$, so that Lemma~\ref{lem:good lift} holds.  Then also $\theta \le \eta \log \sigma$ as required by (P3).  Choosing  $\beta$ such that $- \log \overline\lambda_\ve < \beta < \alpha$ then fixes the appropriate
Banach space $\B$ for Theorem~\ref{thm:tower gap}.  In what follows, we assume 
$\ve < \min\{ \ve^* , \ve_1^* \}$.

\medskip
\noindent
(a) The existence of the escape rate $-\log \lambda_\ve$ follows from Theorem~\ref{thm:tower gap}(a) and 
Lemma~\ref{lem:same escape}.  Define 
\[
\hg_\ve = \P_\Delta \hg_\Delta \, .
\]
By \eqref{eq:commute}, one has $\hg_\ve \in L^1(m_\vf)$ and for each $n$,
\[
\hLp_{\vf^{H_\ve}}^n \hg_\ve = \P_\Delta( \hLp_\Delta^n \hg_\Delta)
= \P_\Delta (\lambda_\ve^n \hg_\Delta) = \lambda_\ve^n \hg_\ve \, ,
\]
so that $\hg_\ve dm_\vf$ defines a conditionally invariant probability measure on $I$ with eigenvalue $\lambda_\ve$.

\medskip
\noindent
(b) We define the required conformal measure $m_{H_\ve}$, using the by-now standard procedure,
\begin{equation}
\label{eq:conformal}
m_{H_\ve}(\psi)  := \lim_{n \to \infty} \lambda_\ve^{-n} \int_{\hI^n} \psi \, dm_\vf, \quad \mbox{for } \psi \in C^\varsigma(I).
\end{equation}
Using Lemma~\ref{lem:good lift}, we find $\tpsi \in \B_0$ such that $\P_\Delta \tpsi = \psi$.  Then by \eqref{eq:commute},
\[
\int_{\hI^n} \psi \, dm_\vf = \int_I \hLp_{\vf^{H_\ve}}^n \psi \, dm_\vf = \int_\Delta \hLp_\Delta^n \tpsi \, dm_\Delta
= \int_{\hDelta^{(n)}} \tpsi \, dm_\Delta \, ,
\]
so that the limit in \eqref{eq:conformal} exists by Theorem~\ref{thm:tower gap}(d), using the spectral gap enjoyed by
$\hLp_\Delta$.  Indeed, $d(\tpsi) = m_{H_\ve}(\psi)$.  The fact that $m_{H_\ve}$ defined in this way is $\vf$-conformal
follows from the same calculation as in the proof of \cite[Theorem~1.7]{DemTodMP}.  The fact that $m_{H_\ve}$
is supported on $\hI^\infty$ follows from its definition in \eqref{eq:conformal}.

\medskip
\noindent
(c) Defining $\nu_{H_\ve} := \hg_\ve m_{H_\ve}$, we see that
\begin{equation}
\label{eq:nu limit}
\nu_{H_\ve}(\psi) = \lim_{n \to \infty} \lambda_\ve^{-n} \int_{\hI^n} \hg_\ve \psi \, dm_\vf, \quad \mbox{for } \psi \in C^\varsigma(I),
\end{equation}
since $\P_\Delta (\hg_\Delta \tpsi) = \hg_\ve \psi$, and $\hg_\Delta \tpsi \in \B$ 
by Lemma~\ref{lem:good lift}. 
This extends to $\psi \in C^0(I)$ by approximation:  for each $\epsilon > 0$, 
we may choose $\psi_\delta \in C^\varsigma(I)$ such that $|\psi - \psi_\delta|_{C^0(I)} \le \epsilon$
and $|\psi_\delta|_{C^\varsigma(I)} \le \delta^{-\varsigma}$.  (This can be accomplished, for example,
through convolution of $\psi$ with a $C^\infty$ mollifier.)  
Then $\nu_{H_\ve}(\psi_\delta - \psi_{\delta'}) \le 2\epsilon$ for each $\delta' < \delta$, so that
$( \nu_{H_\ve}(\psi_\delta) )_{\delta > 0}$ forms a Cauchy family as $\delta \to 0$.  
Moreover, 
\[
\lim_{n \to \infty} \lambda_\ve^{-n} \int_{\hI^n} \hg_\ve \psi \, dm_\vf
= \lim_{n \to \infty} \lambda_\ve^{-n} \int_{\hI^n} \hg_\ve (\psi - \psi_\delta) \, dm_\vf
+ \nu_{H_\ve}(\psi_\delta)
= \nu_{H_\ve}(\psi_\delta) + \mathcal{O}(\epsilon),
\]
since $\lambda_\ve^{-n} \int_{\hI^n} \hg_\ve \, dm_\vf = 1$ for each $n \in \mathbb{N}$.  Since
$\epsilon>0$ was arbitrary, $\nu_{H_\ve}(\psi)$ exists and is given by the limit in \eqref{eq:nu limit}.

Next, again using the commutivity given by \eqref{eq:commute}, we see that
$\nu_{H_\ve} = (\pi_\Delta)_* \nu_\Delta$, where $\nu_\Delta$ is from Theorem~\ref{thm:tower gap}(c).
It follows that
\begin{equation}
\label{eq:var}
\log \lambda_\ve = h_{\nu_\Delta}(f_\Delta) + \int \vf_\Delta \, d\nu_\Delta = h_{\nu_{H_\ve}}(f) + \int \vf \, d\nu_{H_\ve},
\end{equation}
since $\pi_\Delta: \Delta \to I$ is at most countable-to-one,
so that $\nu_{H_\ve}$ achieves the supremum in the variational principle among all invariant probability measures on $\hI^\infty$
that lift to an invariant probability measure on $\Delta$, and $\nu_{H_\ve}$
is unique in this class.

In order to conclude that in fact $\nu_{H_\ve}$ achieves the supremum over all invariant probability
measures $\nu$ with $\nu(-\phi^{H_\ve}) < \infty$, i.e., that are supported on $\hI^\infty$, we note the following inequality,   
 taking our notation from Theorem~\ref{thm:tails},
\begin{equation}
\label{eq:lower pressure}
P(\phi) - \int \phi \, d\nu = \int (P(\phi) - \phi) \, d\nu \ge \bar\alpha =\alpha+\xi   \ge  \alpha,
\end{equation}
for any such measure $\nu$, 
which follows from the proof of Theorem~\ref{thm:tails} for all classes of our admissible
potentials.  Note also that $\int \phi^{H_\ve} \, d\nu = \int \phi \, d\nu$ whenever
$\nu(-\phi^{H_\ve}) < \infty$.

By choice of $L$ in Section~\ref{ssec:uniform}, any ergodic invariant measure $\nu$ with entropy 
$h_\nu(f) >  (\log \lambda_{\ve^*} + \alpha)/2$
lifts to our inducing scheme.
For an $f$-invariant measure $\nu$  with $\nu(-\phi^{H_\ve}) < \infty$,
define the pressure of $\nu$ to be $P_\nu(\phi^{H_\ve}) = h_\nu(f) + \int \phi \, d\nu$.  Now if $P_\nu(\phi^{H_\ve}) \ge P_{\nu_{H_\ve}}(\phi^{H_\ve})$,
then,
\[
h_\nu(f) + \int \phi \, d\nu - P(\phi) \ge h_{\nu_{H_\ve}}(f) + \int \phi \, d\nu_{H_\ve}  - P(\phi) = \log \lambda_\ve ,
\]
by \eqref{eq:var} so that $h_\nu(f) \ge \log \lambda_\ve + \alpha \ge \log \lambda_{\ve^*} + \alpha$, using \eqref{eq:lower pressure}, and so 
$\nu$ lifts
to our inducing scheme by our choice of $L$.  Thus 
$P_\nu(\phi^{H_\ve}) \le P_{\nu_{H_\ve}}(\phi^{H_\ve})$, and $\nu_{H_\ve}$ achieves the supremum
among all invariant measures $\nu$ satisfying $\nu(-\phi^{H_\ve}) < \infty$ 
(so in fact $\nu=\nu_{H_\ve}$).
Thus, $\nu_{H_\ve}$ is the unique equilibrium state for $\phi^{H_\ve}$,
proving item (c) of the theorem.

\medskip
\noindent
(d) The characterisation of the limit proving item (d) now follows from Theorem~\ref{thm:tower gap}(d), again using
Lemma~\ref{lem:good lift} to lift any $\psi \in C^\varsigma(\hI)$ to a function $\tpsi \in B_0$, and then evolving that function according
to \eqref{eq:commute}.  The convergence extends to any $\psi \in C^\varsigma(I)$ since in one iterate,
$\hLp_{\vf^{H_\ve}} \psi$ is supported on $\hI$ so the values of $\psi$ on $H_\ve = I \setminus \hI$ are irrelevant to the
value of the limit. 

To justify Remark~\ref{rmk:wider conv}, note that the convergence in (d) holds for any 
$\psi \in \mathcal{P}_\Delta\B_0$ with $\nu_{H_\ve}(\psi) > 0$ due to \eqref{eq:commute}.  In particular, since the invariant density
$g_\vf = \frac{d\mu_\vf}{dm_\vf}$ satisfies $g_\vf = \mathcal{P}_\Delta g_\Delta$ for some $g_\Delta \in \B_0$,
for any $\psi \in C^\varsigma(I)$, we may define $\tilde\psi = \psi \circ \pi_\Delta$, and then conclude
that $\tilde\psi g_\Delta \in \B_0$ by  Lemma~\ref{lem:smooth lift}.  Thus $\psi g_\vf \in \mathcal{P}_\Delta\B_0$, and so
$\hLp^n_{\vf^{H_\ve}}(\psi g_\vf)/|\hLp^n_{\vf^{H_\ve}}(\psi g_\vf)|_{L^1(m_\vf)}$ converges to $\hg_\ve$ as $n \to \infty$.


\section{Zero-hole limit}
\label{sec:zerohole}

In this section, we will focus on the limit $\lim_{\eps\to 0}-\frac{\log\lambda_\eps}{\mu_\vf(H_\eps)}$, the content of Theorems~\ref{thm:zerohole_Holder} and \ref{thm:zerohole_geom}.  
We assume throughout that $\ve \in (0, \ve_2)$, so that the conclusions of
Corollary~\ref{cor:spectral decomp} hold.  Indeed, we will use the spectral gap for $\hLp_\ve$
to construct a canonical invariant measure $\hat\nu_\ve$ for $\hF_\ve$, supported on the survivor set,
$\hY_\ve^\infty = \cap_{n=0}^\infty \hF_\ve^{-n}(Y)$.

For $\psi \in \cont^\eta(\Q)$, define
\begin{equation}
\label{eq:nu def}
\hat\nu_\ve(\psi) := \lim_{n \to \infty} \Lambda_\ve^{-n} \int_{\hY_\ve^{n-1}} \psi \, \bhg_\ve \, d\hat m \, .
\end{equation}
The limit exists since
\[
\Lambda_\ve^{-n} \int_{\hY_\ve^{n-1}} \psi \,  \bhg_\ve \, d\hat m
= \int_Y \Lambda_\ve^{-n} \hLp_\ve^n(\psi \,  \bhg_\ve) \, d\hat m
\; \xrightarrow[n \to \infty]{} \; \he_\ve(\psi \, \bhg_\ve) ,
\]
where $\he_\ve$ is from Corollary~\ref{cor:spectral decomp}.
Since $|\hat \nu_\ve(\psi)| \le \hat \nu_\ve(1) |\psi|_\infty$, $\hat\nu_\ve$ extends to a bounded linear functional
on $\cont^0(\Q)$, i.e., $\hat\nu_\ve$ is a Borel measure.  Moreover, $\hat\nu_\ve(1)=1$, 
so $\hat\nu_\ve$
is a probability measure, clearly supported on $\hY_\ve^\infty$.  

Let $\hPhi_\ve$ denote the punctured version of the induced potential $\Phi_\ve$, i.e., $\hPhi_\ve = \Phi_\ve$ on $Y \setminus \hat H_\ve'$, and $\hPhi_\ve = - \infty$ on $\hat H_\ve'$.  Recall $P(\Phi_\ve)=0$ by Remark~\ref{rmk:ind press}. According to \cite[Section~6.4.1]{DemTodMP}, $\hat\nu_\ve$ is 
an equilibrium state for the potential $\hPhi_\ve - \log \Lambda_\ve$; on the other hand, by \cite[Lemma~5.3]{BDM},
$\hat\nu_\ve$ is a Gibbs measure  
for the potential $\hPhi_\ve - R_\ve \log \lambda_\ve$, with pressure
$P_{\hat\nu_\ve}(\hPhi_\ve - R_\ve \log \Lambda_\ve) = 0$. We conclude,
\begin{equation}
\label{eq:abramov}
\log\Lambda_\ve= \left(\int R_\eps~d\hat\nu_\eps\right)\log\lambda_\ve.
\end{equation}

Recalling $\hat\mu_{Y, \ve} = G_\ve \hm$ is the invariant probability measure for $F_{z, \ve_0, \ve}$, 
supported on  $Y$, Kac's Lemma in \eqref{eq:Kac}  implies 
$\hat\mu_\ve(\hat H_\ve') = \frac{\hat\mu_{Y, \ve}(\hat H_\ve') }{\int R_\ve~d\hat\mu_{Y, \ve}}$.  So putting these together yields

\begin{equation}
\label{eq:string}
\frac{\log\lambda_\eps}{\hat\mu_\ve(\hat H_\eps)} = \frac{\log\lambda_\eps}{\hat\mu_\eps(\hat H_\eps')}\cdot\frac{\hat\mu_\eps(\hat H_\eps')}{\hat\mu_\ve(\hat H_\eps)} =\frac{\log\Lambda_\eps}{\hat\mu_{Y, \ve}(\hat H_\eps')}\cdot
 \frac{\int R_\eps~d\hat{\mu}_{Y, \ve}}{\int R_\eps~d\hat\nu_\eps} \cdot 
 \frac{\hat\mu_\ve(\hat H_\eps')}{\hat\mu_\ve(\hat H_\eps)} \, .
\end{equation}

Therefore to prove Theorems~\ref{thm:zerohole_Holder} and \ref{thm:zerohole_geom} we must show that as $\ve \to 0$,

\begin{equation}
\label{eq:three amigos}
-\frac{\log\Lambda_\eps}{\hat\mu_{Y, \ve}(\hat H_\eps')} \to 1, 
\quad \frac{\int R_\eps~d\hat\mu_{Y, \ve}}{\int R_\eps~d\hat \nu_\eps}  \to 1, 
\quad \mbox{and}
\quad \frac{\hat\mu_\ve(\hat H_\eps')}{\hat\mu_\ve(\hat H_\eps)} \to 1-e^{-S_p\vf(z)}
\end{equation}
(we take $e^{-S_p\vf(z)}=0$ when $z$ is aperiodic).
These are Theorem~\ref{thm:induced limit},  Proposition~\ref{prop:ratio} and then Lemmas~\ref{lem:holder aperiodic} and \ref{lem:holder per} in the H\"older case and Lemmas~\ref{lem:t aperiodic} and Lemma~\ref{lem:t per} in the geometric case.


\subsection{An asymptotic for $\Lambda_\ve$}

In this subsection, we 
obtain a precise
asymptotic for $\Lambda_\ve$ in terms of the quantity $\hat\mu_{Y,\ve}(\hat H_\ve')$, 
proving the first limit in \eqref{eq:three amigos}.

We remark that we are not able to apply the results of \cite{KelLiv09} in our setting since it does not fit into 
the assumptions of that paper.  In \cite{KelLiv09}, it is assumed that there is a sequence of operators $P_\ve$, 
with a decomposition
similar to that given by Corollary~\ref{cor:spectral decomp} and having largest eigenvalue $\rho_\ve$.  
These operators approach a fixed
operator $P_0$ with eigenvalue 1 and the derivative of $\log \rho_\ve$ is expressed in terms of the
size of the perturbation $P_0 - P_\ve$.

In our setting, the only candidate for $P_0$ is our transfer operator $\Lp_0 = \Lp_{z, \ve_0}$, the transfer operator
corresponding to $F_{z, \ve_0}$,
which does not depend on $\ve$.  However, the relation between $\delta$ and $\ve$ given
by Lemma~\ref{lem:close} is not explicit, so that a good asymptotic expression for $\Lambda_\ve$
is not available starting from $\Lp_0$ (indeed, the relation between $\ve$ and $\delta$ depends in part
on the rate of approach of the orbit of $z$ to itself, which is not guaranteed to be proportional to
the measure of $\hat H_\ve'$).  Instead, as suggested by Lemma~\ref{lem:good bound}, the difference
between $\Lp_\ve$ and $\hLp_\ve$ has the correct order for the asymptotic we want.  In order
to exploit this, we consider then two sequences of operators, $(\Lp_\ve)_{\ve > 0}$ and
$(\hLp_\ve)_{\ve >0}$, and use their uniform spectral properties to prove the required asymptotic for
the maximal eigenvalues $\Lambda_\ve$ of the latter sequence in terms of the maximal 
eigenfunctions of the former sequence.

We begin by establishing the following improved regularity for the functions $G_\ve$ and $\bhg_\ve$. 
\begin{lemma}
\label{lem:reg}
For all $\ve \in (0, \ve_2)$, where $\ve_2>0$ is from Corollary~\ref{cor:spectral decomp},
\begin{equation}
\label{eq:log}
H^\eta(\log G_\ve) \le C_d, \quad \mbox{and} \quad H^\eta(\log  \bhg_\ve ) \le C_d .
\end{equation}
As a consequence, there exists $c_0 > 0$ such that for all $\ve \in (0,\ve_2)$,
\begin{equation}
\label{eq:bounds}
c_0 \le \inf_Y \bhg_\ve \le \| \bhg_\ve \|_{\cont^\eta} \le c_0^{-1} ,
\end{equation}
and similar bounds hold for $G_\ve$.
\end{lemma}

\begin{proof}
Suppose $\psi \in \cont^\eta$ satisfies $H^\eta(\log \psi) \le K$.  Then $\psi(x)/\psi(y) \le e^{Kd(x,y)^\eta}$, for any
$x,y$ belonging to the same element of $\Q$.

We follow the notation in the proof of Proposition~\ref{prop:ly}.  Let $x, y \in Q \in \Q$.  For $n \ge 0$ and
$u \in \hF_\ve^{-n}(x)$, let $Y_i^{(n)}(u)$ denote the $n$-cylinder containing $u$.  For each $u$, there
is a unique $v \in \hF_\ve^{-n}(y) \cap Y_i^{(n)}(u)$.

Using  the log-H\"older regularity of $\psi$ as well as the bounded distortion property \ref{GM3}, we estimate,
\[
\begin{split}
\hLp_\ve^n \psi(x) & = \sum_{u \in \hF_\ve^{-n}(x)} \psi(u) e^{S_n\Phi_\ve(u)}
\le \sum_{u \in \hF_\ve^{-n}(x)} \psi(v) e^{Kd(u,v)^\eta} e^{S_n\Phi_\ve(v)} (1 + C_d d(x,y)^\eta) \\
& \le \hLp_\ve^n \psi(y) e^{KC_e^{-\eta} \sigma^{-n\eta} d(x,y)^\eta} (1+C_d d(x,y)^\eta),
\end{split}
\]
where for the last inequality, we have used property \ref{GM2}.  Now taking logs, and using the inequality
$\log(1+t) \le t$ for all $t \ge 0$, we have
\begin{equation}
\label{eq:log contract}
H^\eta(\log \hLp_\ve^n \psi) \le KC_e^{-\eta} \sigma^{-n \eta} H^\eta(\log \psi) + C_d, \qquad \mbox{for all $n \ge 1$.}
\end{equation}

This implies that for $n$ large enough, $\hLp_\ve^n$ preserves the set of functions
$\{ \psi \in C^\eta(\Q) : H^\eta(\log \psi) \le 1+C_d \}$.  Thus $\bhg_\ve$ must belong to this set.  
Since $\hLp_\ve \bhg_\ve = \Lambda_\ve \bhg_\ve$, substituting
$\bhg_\ve$ into \eqref{eq:log contract} and taking $n \to \infty$ implies that $H^\eta(\log \bhg_\ve ) \le C_d$, proving \eqref{eq:log}.

By a nearly identical argument, \eqref{eq:log contract} applies to $\Lp_\ve$ as well, and so its fixed point
$G_\ve$ satisfies \eqref{eq:log}.

Finally, we show how \eqref{eq:log} implies \eqref{eq:bounds}.  The uniform upper bounds on 
$|\bhg_\ve|_{\cont^\eta}$
and $|G_\ve|_{\cont^\eta}$ follow immediately from Proposition~\ref{prop:ly}; we can set $c_1^{-1} = C$ from that
proposition, so we focus on the lower bounds.

Since $\int \bhg_\ve  \, d\hat m =1$, there exists $Q_0 \in \Q$ such that $\sup_{x \in Q_0} \bhg_\ve(x) \ge 1$.  By \eqref{eq:log},
$\inf_{x \in Q_0} \bhg_\ve(x) \ge e^{-C_d}$.  Now by the mixing property of $F_\ve$ together with 
Lemma~\ref{lem:uniform mix}, there exists $n_0 \in \mathbb{N}$, independent of $\ve \in (0, \ve_1)$, such that
$\hatf^{n_0}_\ve(Q_0) \supset Y$.  Thus for any $y \in Y$, there exists $n(y) \le n_0$ such that
$R^{n(y)}(y) = n_0$.  Then, 
\[
\bhg_\ve(y) = \Lambda_\ve^{-n(y)} \hLp_\ve^{n(y)} \bhg_\ve(y) \ge \Lambda_\ve^{-n(y)} e^{-C_d} \inf_{x\in Q_0 \cap \hatf_\ve^{-n_0}(Y)} e^{S_{n_0}\vf(x)} =: c_2.
\]
Let $c_0:=\min\{c_1, c_2\}$.
Note that by our assumptions on $f$, we have $\inf_{x\in Q_0 \cap \hatf_\ve^{-n_0}(Y)} e^{S_{n_0}\vf(x)}>0$ even when 
$\vf$ is of the form $-t\log|Df|-P(-t\log|Df|)$ because the orbit $x, f(x), \ldots, f^{n_0-1}(x)$ avoids a neighbourhood of $\Crit$ 
for any $x\in Q_0  \cap \hatf_\ve^{-n_0}(Y)$ since $n_0$ is a return time to $Y$ on this set.  
Thus $c_0$ is strictly positive and is also independent of $\ve$
by Lemma~\ref{lem:uniform mix}.  This proves \eqref{eq:bounds} for $\bhg_\ve$ and an identical argument
can be used for $G_\ve$.
\end{proof}

\begin{theorem}
\label{thm:induced limit}
\[
\lim_{\ve \to 0} \frac{1 - \Lambda_\ve}{\hat{\mu}_{Y, \ve}(\hat H'_\ve)} = 1.
\]
\end{theorem}

\begin{proof}
We assume $\ve \in (0, \ve_2)$ since we are interested in the limit $\ve \to 0$.
Iterating \eqref{eq:decomp} for $n \ge 1$,
\[
\hLp_\ve^n G_\ve = \Lambda_\ve^n \he_\ve(G_\ve) \, \bhg_\ve +  \hcR_\ve^n G_\ve
 \implies \bhg_\ve = \frac{1}{\he_\ve(G_\ve)} \big( \Lambda_\ve^{-n} \hLp_\ve^n G_\ve - \Lambda_\ve^{-n}
 \hcR_\ve^n G_\ve \big). 
\]

Using this identity and \eqref{eq:good bound}, we estimate 
\begin{equation}
\label{eq:three}
\begin{split}
1 - \Lambda_\ve & =  \int \bhg_\ve \, d\hm - \int \hLp_\ve \bhg_\ve \, d\hm = \int (\Lp_\ve - \hLp_\ve) \bhg_\ve \, d\hm 
=   \int_{\hat H_\ve'} \bhg_\ve \, d\hm  \\
& =  \frac{1}{\he_\ve(G_\ve)} \int_{\hat H_\ve'} \big( \Lambda_\ve^{-n} \hLp_\ve^n G_\ve 
- \Lambda_\ve^{-n} \hcR_\ve^n G_\ve \big) \, d\hm  \\
& = \frac{1}{\he_\ve(G_\ve)} \left( \int_{\hat H_\ve'} G_\ve \, d\hm
- \int_{\hat H_\ve'} (1- \Lambda_\ve^{-n} \hLp_\ve^n) G_\ve \, d\hm
- \int_{\hat H_\ve'} \Lambda_\ve^{-n} \hcR_\ve^n G_\ve \, d\hm \right) . 
\end{split}
\end{equation}

Using
Corollary~\ref{cor:spectral decomp}, 
we estimate the third term on the right side of \eqref{eq:three} by
\[
\| \Lambda_\ve^{-n} \hcR_\ve^n G_\ve \|_{\cont^\eta} \le e^{-2\beta n/3} \Lambda_\ve^{-n} \| G_\ve \|_{\cont^\eta}
\le e^{-\beta n/3} \| G_\ve \|_{\cont^\eta} . 
\]
Due to \eqref{eq:bounds},  $\| G_\ve \|_{\cont^\eta} \le c_0^{-1}$ and $G_\ve \ge c_0$ 
uniformly in $\ve$.
Thus
\begin{equation}
\label{eq:three three}
 \left| \int_{\hat H'_\ve} \Lambda_\ve^{-n} \cR_\ve^n G_\ve \, d\hm \right|
= \int_{\hat H_\ve'} G_\ve \, d\hm \cdot \mathcal{O}(e^{-\beta n/3}) . 
\end{equation}

Next, the second term on the right hand side of \eqref{eq:three} can be rewritten as,
\[
\int_{\hat H_\ve'} (1- \Lambda_\ve^{-n} \hLp_\ve^n) G_\ve \, d\hm
= (1-\Lambda_\ve^{-n}) \int_{\hat H_\ve'} G_\ve \, d\hm + \Lambda_\ve^{-n} \int_{\hat H_\ve'} (\Lp_0^n - \hLp_\ve^n) G_\ve \, d\hm , 
\]
recalling that $\Lp_0$ is the transfer operator corresponding to $F_{z, \ve_0}$ which also has $\hm$
as a conformal measure.
Now  the maps $F_{z, \ve_0}$ and $\hF_{z, \ve_0, \ve}$ differ on the
1-cylinders contained in $B_\ve \cup \hat H_\ve'$, where $B_\ve$ is defined in the proof of 
Lemma~\ref{lem:close}.  Thus $F^n_{z, \ve_0}$ and $\hF^n_{z, \ve_0, \ve}$ differ on the
$n$-cylinders contained in $B'_{\ve, n} := \big( \cup_{i=0}^{n-1} F_{z, \ve_0}^{-i} (B_\ve \cup \hat H_\ve') \big)\bigcup \big( \cup_{i=0}^{n-1} F_{z, \ve_0, \ve}^{-i} (B_\ve \cup \hat H'_\ve) \big)$.  
Now  following \eqref{eq:error sum} and \eqref{eq:infinity close}, we have
\[
| ( \Lp_0^n - \hLp_\ve^n) G_\ve |_\infty \le \frac{1+C_d}{q} 2 \, \hm(B'_{\ve,n})  |G_\ve|_\infty. 
\]

 Then the second term
on the right side of \eqref{eq:three} can be bounded by
\begin{equation}
\label{eq:three two}
\left| \int_{\hat H_\ve'} (1- \Lambda_\ve^{-n} \hLp_\ve^n) G_\ve \, d\hm \right|
= \int_{\hat H_\ve'} G_\ve \, d\hm \cdot \mathcal{O}\left( (1-\Lambda_\ve^{-n}) + \Lambda_\ve^{-n} \hm(B'_{\ve,n}) \right), 
\end{equation}
using  \eqref{eq:bounds} again to estimate,
$\int_{\hat H_\ve'} |G_\ve|_\infty \, d\hm \le c_0^{-2} \int_{\hat H_\ve'} G_\ve \, d\hm$.

Putting \eqref{eq:three three} and \eqref{eq:three two} together with \eqref{eq:three} and dividing through
by $\int_{\hat H'_\ve} G_\ve \, d\hm  = \hat{\mu}_{Y, \ve}(\hat H'_\ve)$ yields,
\[
\frac{1-\Lambda_\ve}{\hat{\mu}_{Y, \ve}(\hat H'_\ve)} =  \frac{1}{\he_\ve(G_\ve)} 
 \left( 1 + \mathcal{O}\left( (1-\Lambda_\ve^{-n}) + \Lambda_\ve^{-n}\hm(B'_{\ve, n}) + e^{-\beta n/3} \right) \right) .
\]
The quantity $\he_\ve(G_\ve)$ can be made arbitrarily close to $e_0(G_\ve)=1$  by 
Corollary~\ref{cor:spectral decomp}.

Now fix $\delta>0$ and first choose $n$ sufficiently large that
$e^{-\beta n/3} < \delta$.  Next choose $\ve$ sufficiently small so that $|\he_\ve(G_\ve) - 1| < \delta$, 
$|1-\Lambda_\ve^{-n}| < \delta$
and $\Lambda_\ve^{-n} \le 2$ by Corollary~\ref{cor:spectral decomp},
and $\hm(B'_{n,\ve}) < \delta$ by Corollary~\ref{cor:large connection}.
Then the error term is
$\mathcal{O}(\delta)$, and since $\delta$ was arbitrary, the theorem follows.
\end{proof}


\subsection{Convergence of the integral of the return time}

In this subsection, we prove the convergence of the second limit in \eqref{eq:three amigos}, regarding the
integral of the return time. As before, 
we assume $\ve \in (0, \ve_2)$, so that the conclusions of
Corollary~\ref{cor:spectral decomp} hold.  

Recall the invariant measure $\hat \nu_\ve$ from \eqref{eq:nu def} supported on $\hY^\infty$, and 
that $\hat\mu_{Y,\ve} = G_\ve  \hat m$ is the invariant measure for $F_\ve$ given by Corollary~\ref{cor:spectral gap}.
The main result of this subsection is the following.

\begin{proposition}
\label{prop:ratio}
Let $R_\ve = R_{z, \ve_0, \ve}$.  Then,
\[
\lim_{\ve \to 0} \frac{\int R_\eps~d\hat\mu_{Y, \ve}}{\int R_\eps~d\hat\nu_\eps}  = 1 \, .
\]
\end{proposition}

\begin{proof}
First we show that for $\psi \in \cont^\eta(\Q)$, $|\hat\nu_\ve(\psi) - \hat\mu_{Y, 0}(\psi)| \to 0$ as $\ve \to 0$. 
Let $\hPi_\ve$ be the projector defined by $\bhg_\ve \otimes \he_\ve$,  i.e.
\[
\hPi_\ve(\psi) = \he_\ve(\psi) \, \bhg_\ve , \qquad \mbox{for all $\psi \in \cont^\eta(\Q)$,}
\]
and similarly for $\Pi_0$.  Recall that we have normalised the eigenvectors so that $\hat m(\bhg_\ve ) = \hat m(G_0) =1$.

Notice that since $\Lp_0^* \hat m = \hat m$, $e_0(\psi)$ is simply 
$\hat m(\psi)$.  Thus $\hat\mu_{Y,0}(\psi) = e_0(\psi G_0 )$.  Now,
\[
\begin{split}
|\hat\nu_\ve(\psi) - \hat\mu_{Y,0}(\psi)| &  \le |\he_\ve(\psi \bhg_\ve) - e_0(\psi \bhg_\ve)| + |e_0(\psi \bhg_\ve) - e_0(\psi g_0)|  \\
& \le \left| \int_Y \hPi_\ve(\psi  \bhg_\ve) - \Pi_0(\psi \bhg_\ve) \, d\hat m \right|
+ \left| \int_Y \psi( \bhg_\ve - G_0 ) \, d\hat m \right| \\
& \le \vertiii{ \hPi_\ve - \Pi_0 } \, \| \psi  \bhg_\ve \|_{\cont^\eta(\Q)} + |\psi|_\infty |\bhg_\ve - G_0|_{L^1(\hat m)}, 
\end{split}
\]
and both terms go to zero as $\ve \to 0$ by 
Corollary~\ref{cor:spectral decomp} (which in turn uses \cite{KelLiv99}).

It also follows from Corollary~\ref{cor:spectral gap}, that $\hat\mu_{Y,\ve}(\psi) \to \hat\mu_{Y,0}(\psi)$ as $\ve \to 0$.
Thus by the triangle inequality, $|\hat\nu_\ve(\psi) - \hat\mu_{Y,\ve}(\psi)| \to 0$ as $\ve \to 0$, for all $\psi \in \cont^\eta(\Q)$.

This does not immediately imply the proposition since $R_\ve \notin \cont^\eta(\Q)$.
However, we claim that $\hLp_\ve(R_\ve) \in \cont^\eta(\Q)$.  First, $\hLp_\ve(R_\ve)$ is bounded for all $x \in Y$ by
\begin{equation}
\label{eq:R}
\hLp_\ve R_\ve(x) = \sum_{u \in \hF_\ve^{-1}(x)} R_\ve(u) e^{\Phi_\ve(u)}
\le \frac{1+C_d}{q} \sum_{u \in \hF_\ve^{-1}(x)} R_\ve(u) \hat m(Y_i(u)),
\end{equation}
by \eqref{eq:distortion}, where $Y_i(u)$ is the 1-cylinder containing $u$.  The last sum
is simply bounded by $\hat m (R_\ve) = \hat m_\ve(\hatI_\ve)$, since $F_\ve$ is a first return map to $Y$ in
the Hofbauer extension.  This is uniformly  bounded in $\ve$ by Theorem~\ref{thm:tails}.
Next, since $R_\ve$ is constant on 
1-cylinders, $Y_i \in \mathcal{Y}_\ve$, using \ref{GM3} the H\"older constant of $\hLp_\ve(R_\ve)$ is bounded by,
\[
\hLp_\ve R_\ve(x) - \hLp_\ve R_\ve(y) = \sum_{ u \in \hF_\ve^{-1}(x)} R_\ve(u) (e^{\Phi_\ve(u)} - e^{\Phi_\ve(v)} ) \le C_d d(x,y)^\eta  \sum_{u \in \hF_\ve^{-1}(x)} R_\ve(u) e^{\Phi_\ve(u)},
\]
for all $x , y \in Q \in \mathcal{Q}$, where each $v \in \hF_\ve^{-1}(y)$ is paired with $u \in \hF_\ve^{-1}(x)$ lying in the same
1-cylinder.  The sum is again uniformly bounded in $\ve$ as in \eqref{eq:R}, proving the claim.

It follows that, $\hLp_\ve(R_\ve \bhg_\ve) \in \cont^\eta(\Q)$
and by Lemma~\ref{lem:reg}, also $\frac{\hLp_\ve ( R_\ve \bhg_\ve)}{\bhg_\ve} \in \cont^\eta(\Q)$. 
Now by \eqref{eq:nu def},
\[
\lim_{n \to \infty}  \Lambda_\ve^{-n} \int_{\hY_\ve^n} R_\ve \, \bhg_\ve \, d\hat m
= \lim_{n \to \infty} \Lambda_\ve^{-n} \int_{\hY_\ve^{n-1}} \frac{ \hLp_\ve(R_\ve \bhg_\ve) }{\bhg_\ve} 
\, \bhg_\ve \, d\hat m
\xrightarrow[n \to \infty]{}
\hat\nu_\ve \left( \frac{ \hLp_\ve(R_\ve \bhg_\ve) }{ \Lambda_\ve \bhg_\ve} \right) . 
\]
Thus $\hat \nu_\ve(R_\ve)$ exists and is defined by \eqref{eq:nu def}.

For $N \in \mathbb{N}$, define the truncation $R_\ve^{(N)} = \min \{ R_\ve, N \}$.   For $R_0 = R_{z, \ve_0}$,
define $R_0^{(N)}$ similarly.
By the above arguments, it follows that $\hLp_\ve(R_\ve^{(N)} \bhg_\ve) \in \cont^\eta(\Q)$ and that
$\hat \nu_\ve(R_\ve^{(N)})$ exists and is defined by \eqref{eq:nu def}. Similarly for the complementary function,
$\hLp_\ve(1_{R_\ve > N} \cdot R_\ve) \in \cont^\eta(\Q)$, and
$\hat \nu_\ve(1_{R_\ve > N} \cdot R_\ve)$ exists and is defined by \eqref{eq:nu def}.

Next, we claim that $R_\ve$ is uniformly integrable with respect to $\hat \nu_\ve$; in
particular, $\hat \nu_\ve(1_{R_\ve > N} \cdot R_\ve) \to 0$ as $N \to \infty$ uniformly in $\ve$.  To see this, note that
by \eqref{eq:nu def},
\[
\begin{split}
\hat\nu_\ve (1_{R_\ve > N} \cdot R_\ve) & = \lim_{n \to \infty} \Lambda_\ve^{-n} \int_Y \hLp_\ve^n(1_{R_\ve > N} \cdot R_\ve \, \bhg_\ve)
\, d\hat m
\le \lim_{n \to \infty}  \left|\Lambda_\ve^{-n}  \hLp_\ve^{n-1} (\hLp_\ve(1_{R_\ve > N} \cdot R_\ve \, \bhg_\ve))\right|_\infty  \\
& \le C |\hLp_\ve(1_{R_\ve > N} \cdot R_\ve \, \bhg_\ve)|_\infty \, ,
\end{split}
\]
where we have used \eqref{eq:C0 bound} for the last inequality, together with the fact that $\Lambda_\ve^{-n} \hat m(\hY_\ve^{n-1})$
is bounded uniformly in $\ve$ and $n$ by Corollary~\ref{cor:spectral gap}.  Then estimating as in \eqref{eq:R},
\[
\left|\hLp_\ve(1_{R_\ve > N} \cdot R_\ve \, \bhg_\ve)(x)\right|
\le C \sum_{\substack{u \in \hF_\ve^{-1}(x) \\ R_\ve(u) > N}} R_\ve(u) \hat m(Y_i(u))
\le C \sum_{k > N} k \, \hat m(R_\ve = k) \le C' e^{-\alpha N} \, ,
\]
by Theorem~\ref{thm:tails}, and the claim is proved.

It follows from the proof of Corollary~\ref{cor:large connection}, that for each $N > 0$, there exists
$\ve_N >0$ such that for $\ve \le \ve_N$, all one cylinders $Y_i$ for $F_\ve$ with $R_\ve(Y_i) \le N$, are also
one cylinders for $F_0$ with the same return time.  This implies that $R_\ve^{(N)} = R_0^{(N)}$ for
$\ve \le \ve_N$.

Let $\delta>0$ be arbitrary.  Choose $N$ such that $\hat\nu_\ve(R_\ve > N) < \delta$, 
$\hat\mu_{Y,\ve}(R_\ve > N) < \delta$,
and $\hat\mu_{Y,0}(R_0 > N) < \delta$, for all $\ve < \ve_1$, which is possible by the claim and Theorem~\ref{thm:tails}.
Then for $\ve \le \ve_N$, we have
\[
\hat\nu_\ve(R_\ve) = \hat\nu_\ve\left(R_0^{(N)}\right) + \mathcal{O}(\delta) 
\xrightarrow[\ve \to 0]{} 
\hat\mu_{Y,0}\left(R_0^{(N)}\right) + \mathcal{O}(\delta) = \hat\mu_{Y,0}(R_0) + \mathcal{O}(\delta).
\]
Similarly,
\[
\hat\mu_{Y,\ve}(R_\ve) = \hat\mu_{Y,\ve}\left(R_0^{(N)}\right) + \mathcal{O}(\delta) 
\xrightarrow[\ve \to 0]{} 
\hat\mu_{Y,0}\left(R_0^{(N)}\right) + \mathcal{O}(\delta) = \hat\mu_{Y,0}(R_0) + \mathcal{O}(\delta).
\]
Since $\delta$ was arbitrary, this proves the proposition.
\end{proof}


\subsection{Final step of the proof of Theorem~\ref{thm:zerohole_Holder}: the H\"older continuous case } 
\label{sec:holder}

In the next two sections, we prove the third limit in \eqref{eq:three amigos} in both the periodic and nonperiodic cases.
In the present section we address the case when $\vf$ is H\"older continuous, and in Section~\ref{ssec:geo limit} 
we will address the case
when $\vf$ is a geometric potential.  As a prelimary result, we prove the following lemma.

\begin{lemma}
For $f\in \F$ and a H\"older potential $\vf$, we have $\inf_{x\in I}\frac{d\mu_\vf}{dm_\vf}(x)>0$, where $\mu_\vf$ and $m_\vf$ are 
the relevant invariant and conformal measures.
\label{lem: low dens Hold}
\end{lemma}

\begin{proof}
For simplicity we write $g(x)=\frac{d\mu_\vf}{dm_\vf}(x)$ and note that $m_\vf$ is a $\vf$-conformal measure so $\L_\vf g=g$,
where $\L_\vf$ is the transfer operator associated to $\vf$ and $f$ (not the induced dynamics), 
defined in Section~\ref{sec:puncture}.
Since $\pi_*\hat \mu_\ve = \mu_\vf$ (we take any $\ve \in (0, \ve_2)$), Lemma~\ref{lem:reg} implies that there is an open set $U$ such that 
$\inf_{x\in U}g(x)>0$.  By leo, there is some $n\in \N$ such that $f^n(U)=I$.  Hence for any $x\in I$ we can estimate 
\begin{equation}
\label{eq:lower g}
g(x)=\L_\vf^ng(x)=\sum_{y\in f^{-n}(x)}g(y) e^{S_n\vf(y)} \ge \sum_{y\in \{f^{-n}(x)\} \cap U}g(y) e^{S_n\vf(y)},
\end{equation}
So we conclude by noting that $\inf S_n\vf>-\infty$.
\end{proof}

We first address the case in which $z$ is aperiodic.

\begin{lemma}
\label{lem:holder aperiodic}
Let $z$ be an aperiodic point for $f$ and suppose $\vf$ is H\"older continuous.  
Then,
\[
\lim_{\ve \to 0} \frac{ \hat\mu_\ve(\hat H_\ve')}{\hat\mu_\ve(\hat H_\ve)} = 1.
\]
\end{lemma}

\begin{proof}
Recall from  \eqref{eq:Kac} and \eqref{eq:Kac2} that $\hat \mu_{Y, \ve}$ and $\hat \mu_\ve$ are related by the following: 
$\hat \mu_{Y, \ve} = \frac{\hat\mu_\ve|_Y}{\hat\mu_\ve(Y)}$, and 
\begin{equation}
\label{eq:meas project}
\hat\mu_\ve(A) 
= \sum_i \sum_{j=0}^{R_i-1} \hat\mu_\ve (Y_i \cap \hatf^{-j}A), \mbox{ for any Borel $A \subset \hatI$}.
\end{equation}

We will apply the above expression to $A = \hat H_\ve$.  Note that due to our construction of
$\hatI_{z, \ve_0, \ve}$, for each $j$ if $\hatf^j(Y_i) \cap \hat H_\ve \neq \emptyset$, then 
$\hatf^j(Y_i) \subset \hat H_\ve$.  Thus each term in the above sum is either $0$ or $\hat\mu_\ve(Y_i)$.
Define for $k \ge 1$,
\begin{equation}
\label{eq:Hk}
\hat H_\ve'(k) = \left\{ Y_i \subset \hat H_\ve' : Y_i \mbox{ enters $\hat H_\ve$ exactly $k$ times before time
$R_i$} \right\} .
\end{equation}
Now using \eqref{eq:meas project} and our observation about $Y_i$,
\begin{equation}
\label{eq:unwrap}
\hat\mu_\ve(\hat H_\ve) = \sum_{k \ge 1} \sum_{Y_i \in \hat H_\ve'(k)} k \hat \mu_\ve(Y_i)
= \hat \mu_\ve(\hat H_\ve') + \sum_{k \ge 2} \sum_{Y_i \in \hat H_\ve'(k)} (k-1) \hat \mu_\ve(Y_i) .
\end{equation}
We proceed to estimate the double sum over $k$ and $Y_i$.

By \eqref{eq:meas project}, since $F_\ve$ is the first return map to $Y$ in $\hatI$, the
invariant density $G_\ve$ from Corollary~\ref{cor:spectral gap} is also the density for $\hat\mu_\ve$
on $Y$, up to a normalising constant.   Applying the uniform bounds on $G_\ve$ from Lemma~\ref{lem:reg}, we
replace
$\hat \mu_\ve(Y_i)$ with $\hat m(Y_i)$ in \eqref{eq:unwrap}, up to a uniform constant.
For $Y_i \subset \hat H_\ve'(k)$, let $T_i$ denote the time of the $k$th entry of $Y_i$ to
$\hat H_\ve$ under iteration of $\hatf$.  By the conformality of $\hat m = \hat m_{\hat \vf}$,
\begin{equation}
\label{eq:entry}
\hat m_\ve(Y_i) \le C e^{S_{T_i} \hat\vf(y_i)} \hat m_\ve(\hatf^{T^i}Y_i) \le C e^{-\bar\alpha T_i} \hat m_\ve(\hatf^{T_i}Y_i),
\end{equation}
for any $y_i \in Y_i$, where $\bar\alpha>0$ is from the proof of Theorem~\ref{thm:tails}.

Fixing $Y_i \in \hat H_\ve'(k)$, we wish to estimate 
$\# \{ Y_j \in \hat H_\ve'(k) : T_j = T_i \mbox{ and } \hatf^{T_i}Y_i \cap \hatf^{T_j}Y_j \neq \emptyset \}$.
Due to our construction of the Hofbauer extension, such a $Y_j$ is contained in a set $Z_j \in \hat{\mathcal P}_{T_j+L}$,
such that $f^{T_j}$ maps $Z_j$ injectively into a connected component of $\hat H_\ve$.  $Z_j$
can be associated with a word of length $T_j +L$, the first symbol of which lies in $Y$, while
the remaining symbols lie in $\hatI_{z, \ve_0, \ve} \setminus Y$.  We divide this word into blocks of length $L$,
and note there are $\lfloor \frac{T_j}{L} \rfloor$ of them.  They are all {\em external blocks}
according to the terminology of \cite{DobTod15}.  According to \cite[Lemma~4.6]{DobTod15}, there are
at most $16 \mathfrak{d}^2 L^3$ external blocks of length $L$.  In addition, since $f^{T_j}Z_j \subset \hat H_\ve$,
we may choose $\ve$ sufficiently small that any remaining symbols between
$\lfloor \frac{T_j}{L} \rfloor L$ and $T_j$ also belong to an external block of length $L$.  Finally, there
are at most $(2\mathfrak{d}L)^2$ choices for the first symbol of $Z_j$ since this is an upper bound on the number
of elements in $\hat I'_{z, \ve_0, \ve}(L)$.  Putting these estimates together, we conclude that
\begin{equation}
\label{eq:card}
\# \{ Y_j \in \hat H_\ve'(k) : T_j = T_i \mbox{ and } \hatf^{T_i}Y_i \cap \hatf^{T_j}Y_j \neq \emptyset \}
\le (2\mathfrak{d}L)^2 (16 \mathfrak{d}^2 L^3)^{\frac{T_j}{L} + 1} \le C e^{\xi T_j},
\end{equation}
where $\xi < \bar\alpha$ (by choice of $L$) is the same as in the proof of Theorem~\ref{thm:tails}.

Next, due to the aperiodicity of $z$ and the continuity of $f$, for each $\ve > 0$
there exists $N = N(\ve) \in \mathbb{N}$, such that $\hatf^j \hat H_\ve \cap \hat H_\ve = \emptyset$ for all $j < N$,
and $N(\ve) \to \infty$ as $\ve \to 0$.  This implies in particular that if $Y_i \in \hat H_\ve'(k)$, then
$T_i \ge (k-1)N$.

We organise our estimate for $Y_i \subset \hat H'_\ve(k)$ by considering 
$\hat H'_\ve(k) = \cup_{t \ge (k-1)N} \{ Y_i \in \hat H'_\ve(k) : T_i = t \}$.  Then using \eqref{eq:entry} and \eqref{eq:card},
\[
\sum_{\substack{ Y_i \in \hat H_\ve'(k) \\ T_i = t }} (k-1) \hat \mu_\ve(Y_i)
\leq \sum_{\substack{ Y_i \in \hat H_\ve'(k) \\ T_i = t }}  C (k-1) e^{-\bar\alpha t} \hat m_\ve (\hat f^t(Y_i))
\leq C (k-1) e^{-(\bar\alpha - \xi) t} (t+L)^3 m_\vf(H_\ve),
\]
where for the last inequality, we have used the fact that $\hat f^t(Y_i)$ lies in a component of $\hat H_\ve$ in level at most
$t + L$ in the Hofbauer extension.  Since there are at most $\mathfrak{d}\ell^2$ connected components on level $\ell$
according to the proof of \cite[Lemma~4.6]{DobTod15}, we obtain that projecting $\hat H_\ve|_{\mbox{\scriptsize level } \ell}$ down to $H_\ve$, we have
$\hat m_\ve(\hat H_\ve|_{\mbox{\scriptsize level } \ell}) \le \mathfrak{d}\ell^2 m_\vf(H_\ve)$ and summing over $\ell \le (t+L)$ yields the required bound. 

Using this estimate in the double sum in \eqref{eq:unwrap}, we obtain,
\begin{equation}
\label{eq:error bound}
\begin{split}
\sum_{k \ge 2} \sum_{Y_i \subset \hat H_\ve'(k)} (k-1) \hat \mu_\ve(Y_i) 
& \le \sum_{k \ge 2} \sum_{t \ge N(k-1)} \sum_{\substack{Y_i \in \hat H_\ve'(k) \\ T_i = t}}
C e^{- \bar \alpha t} (k-1) \hat m_\ve(\hatf^tY_i) \\
& \le \sum_{k \ge 2} \sum_{t \ge N(k-1)} C e^{-(\bar\alpha - \xi) t} (k-1) (t+L)^3 m_\vf(H_\ve) \\
& \le C' m_\vf(H_\ve) \sum_{k \ge 2} e^{-\alpha N(k-1)} (k-1) 
\le C'' \mu_\vf(H_\ve) e^{-\alpha N} \, ,
\end{split}
\end{equation}
where in the last step we have used Lemma~\ref{lem: low dens Hold}.
Combining this estimate with \eqref{eq:unwrap} and dividing through by $\hat \mu_\ve(\hat H_\ve)$ 
(using that $\hat \mu_\ve(\hat H_\ve) = \mu_\vf(H_\ve)$)
yields,
\[
\frac{\hat \mu_\ve(\hat H_\ve')}{\hat \mu_\ve(\hat H_\ve)} = 1
- \mathcal{O}\!\left(e^{-\alpha N}\right) \, .
\] 
Since $N(\ve) \to \infty$ as $\ve \to 0$, this completes the proof of the lemma.
\end{proof}

Our next lemma addresses the case in which $z$ is periodic with prime period $p$.

\begin{lemma}
\label{lem:holder per}
Suppose $z$ is a periodic point for $f$ of prime period $p$, and that $\vf$ is H\"older continuous.
\begin{itemize}
  \item[a)]  If  $\{ f^n(c) : c \in \Crit, n \ge 1 \} \cap \{ z \} = \emptyset$, then  
    $\displaystyle \lim_{\ve \to 0} \frac{-\log \lambda_\ve}{\mu_\vf(H_\ve)} = 1 - e^{S_p\vf(z)}$.
  \item[b)]  Suppose  $\{ f^n(c) : c \in \Crit, n \ge 1 \} \cap \{ z \} \neq \emptyset$.   If in addition,
  either $f^p$ is orientation preserving in a neighbourhood of $z$, or $\lim_{\ve\to 0} \frac{m_\vf(z+\ve, z)}{m_\vf(z, z-\ve)} = 1$, then
  $\displaystyle \lim_{\ve \to 0} \frac{-\log \lambda_\ve}{\mu_\vf(H_\ve)} = 1 - e^{S_p\vf(z)}$.  
\end{itemize}
\end{lemma}

\begin{proof}
Fix $N_0$ arbirarily large.  Due to \eqref{eq:eps0 choice}, we may choose $\ve'>0$ sufficiently small so that
for all $\ve < \ve'$, the following properties hold.
\begin{itemize}
  \item[(i)]  If $y \in H_\ve(z)$, then for all $j = 1, \ldots, pN_0$, $f^j(y) \in H_\ve(z)$ only if $j = kp$ for some $k = 1, \ldots N_0$.
  \item[(ii)]  If $y \in H_\ve(z)$ and there exists $k_1 \le N_0$ such that $f^{k_1p}(y) \in H_\ve$, then $f^{kp}(y) \in H_\ve(z)$
  for all $k = 1, \ldots, k_1$.
  \item[(iii)]  Each 1-cylinder $Y_i \subset \hat H_\ve'$ whose first entry time $\ell$ to $\hat H_\ve$ is less than $N_0$ is contained
  in an interval $Z_j \subset Q \in \mathcal{Q}$ such that $\hatf^\ell(Z_j)$ maps injectively onto a connected component
  of $\hat H_\ve$, which we will denote by $\hat H_\ve(Z_j)$.
  \item[(iv)] $\hatf^{pN_0}$ is injective and continuous on each connected component of 
$\hat H_\ve \cap \hat f^{-p N_0}(\hat H_\ve)$ that occurs below level $N_0$ in $\hatI_{z, \ve_0, \ve}$. 
\end{itemize}

Properties (i) and (ii) follow from the periodicity of $z$ and the uniform continuity of $f^n$ for each orbit segment of
length $n \le pN_0$.
To deduce Property (iii), since $f^k(z \pm \ve_0) = z$ is not allowed by choice of $\ve_0$ in \eqref{eq:eps0 choice}, it suffices to choose
\[
\ve' \le \tfrac 12 \min \{ d(z, f^k(w)) : w \in \Crit_{z, \ve_0}, f^k(w) \neq z, k \le N_0 \} \, .
\] 
With this choice of $\ve'$, no boundary points of $\hat I_{z, \ve_0, \ve}$ for $\ve < \ve'$ may fall in the interior of a connected 
component of $\hat H_\ve$ with a first entry time less than $N_0$.  Finally, Property (iv) holds since
the orbit of $z$ must be disjoint from $\Crit$; otherwise $f$ would have an attracting periodic orbit, which is forbidden
in our class of maps $\mathcal{F}$.  Thus we may choose
\[
\ve' \le |Df^{pN_0}|_\infty^{-1} \min \{ d(f^k(z), \Crit) : k = 0, \ldots, p-1 \} \, ,
\]
in order to guarantee (iv).

Starting from \eqref{eq:meas project}, we group the 1-cylinders  $Y_i \subset \hat H_\ve'$ as follows.  Let $\ell_i \in \mathbb{N}$ denote the
greatest $\ell \le N_0$ such that $\hatf^{\ell p}(Y_i) \subset \hat H_\ve$.  By (i) and (ii) above, if $j \le pN_0$, then 
$\hatf^j(Y_i) \subset \hat H_\ve$ if and only
if $j = \ell p$ for some $\ell \le \ell_i$.  Recalling \eqref{eq:Hk}, we let $\hat H_\ve'(k, N_0)$ denote the set of $Y_i \subset \hat H_\ve'(k)$
such that the first entry of $Y_i$ to $\hat H_\ve$ occurs before time $N_0$, while $\hat H_\ve'(k, \sim ) = \hat H_\ve'(k) \setminus \hat H_\ve'(k,N_0)$.  Moreover, $\hat H_\ve'(*, N_0) := \cup_{k \ge 1} \hat H_\ve'(k, N_0)$. Then,
\begin{equation}
\label{eq:period expand}
\hat\mu_\ve(\hat H_\ve) = \sum_{\ell=1}^{N_0} \sum_{\substack{Y_i \subset \hat H_\ve'(*,N_0) \\ \ell_i = \ell }} \ell \hat\mu_\ve(Y_i) 
+ \sum_{k > N_0} \sum_{Y_i \in \hat H_\ve'(k, N_0)} (k-\ell_i) \hat\mu_\ve(Y_i)
+ \sum_{k \ge 1} \sum_{Y_i \in \hat H_\ve'(k, \sim )} k \hat\mu_\ve(Y_i) \, .
\end{equation}
Since the entry times to $\hat H_\ve$ are greater than $N_0$ for each of the sets counted in the second and third sums above,
we may use \eqref{eq:entry} and \eqref{eq:error bound} to estimate that these two sums are of order
$\mathcal{O}(e^{-\alpha N_0} \hat\mu_\ve(\hat H_\ve))$.  It remains to estimate the first sum above.
We rewrite \eqref{eq:period expand} as,
\begin{equation}
\label{eq:expand}
\hat \mu_\ve(\hat H_\ve) = \sum_{\ell = 1}^{N_0} \sum_{\substack{Y_i \subset \hat H_\ve'(*, N_0) \\ \ell_i \ge \ell}} \hat\mu_\ve(Y_i)
+ \mathcal{O}\!\left( e^{-\alpha N_0} \hat \mu_\ve(\hat H_\ve) \right) \, .
\end{equation}
For $\ell = 1$, we have simply,
\[
\sum_{\substack{Y_i \subset \hat H_\ve'(*, N_0) \\ \ell_i \ge 1}} \hat\mu_\ve(Y_i) = \hat \mu_\ve(\hat H_\ve') + 
\mathcal{O}\!\left( e^{-\alpha N_0} \hat \mu_\ve(\hat H_\ve) \right) \, ,
\]
since any $Y_i \subset \hat H_\ve'$ not counted in the sum for $\ell = 1$ has first entry time to $\hat H_\ve$ greater than $N_0$.

To estimate the contribution for the terms corresponding to $\ell = 2$, we use property (iii) above.  If $Y_i \subset \hat H_\ve(*, N_0)$
with $\ell_i \ge 2$, then $Y_i$ is contained in an interval $Z_j$ such that $\hatf^k(Z_j)$ maps injectively onto a connected 
component of $\hat H_\ve$ (for the first time) at some time $k = k(Z_j) \le N_0$.  Let us denote this component of $\hat H_\ve$ by $\hat H_\ve(Z_j)$.
Let $A_j$ denote those indices $i$ for which $Y_i \subset Z_j$.  Then,
\[
\sum_{\substack{Y_i \subset \hat H_\ve'(*, N_0) \\ \ell_i \ge 2}} \hat\mu_\ve(Y_i)
=\sum_{Z_j} \hat\mu_\ve(Z_j) \sum_{i \in A_j} \frac{\hat\mu_\ve(Y_i)}{ \hat\mu_\ve(Z_j)} \, .
\] 
Notice that since $\ell = 2$, for each $i \in A_j$, $\hatf^{k(Z_j)}(Y_i) \subset \hat H_\ve(Z_j) \cap \hatf^{-p}(\hat H_\ve)$.
Recalling Lemma~\ref{lem:reg} and the conformality of $\hat m_\ve$ (recall that $\hat m = \hat m_\ve$
depends on $\ve$ on $\hat I_{z, \ve_0, \ve} \setminus Y$), we estimate,\footnote{We use the notation
$a = C^{\pm 1} b$ to mean $C^{-1} b \le a \le Cb$ for some constant $C\ge 1$.}
\begin{equation}
\label{eq:ratio}
\begin{split}
\sum_{i \in A_j} \frac{\hat\mu_\ve(Y_i)}{ \hat\mu_\ve(Z_j)} 
& = \sum_{i \in A_j} e^{\pm 2C_d {\scriptsize \diam}(Z_j)^\eta} \frac{ \int_{\hatf^k(Y_i)} e^{S_k\hat\vf \circ \hatf_j^{-k}} \, d\hat m_\ve }
{\int_{\hatf^k(Z_j)} e^{S_k\hat\vf \circ \hatf_j^{-k}} \, d\hat m_\ve } \\
& = e^{\pm 2C_d {\scriptsize \diam}(Z_j)^\eta} P_\ve(S_k\hat\vf(Z_j))^{\pm 2} \frac{\hat m_\ve\left(  \hat H_\ve(Z_j) \cap \hatf^{-p}(\hat H_\ve)\right)}{\hat m_\ve( \hat H_\ve(Z_j))}  \,
\end{split}
\end{equation}
where $\hat f_j^{-k}$ is the inverse branch of $\hat f^k|_{Z_j}$ and 
\[
P_\ve(S_k\hat \vf(Z_j)) = \sup_{x, y \in Z_j} e^{S_k \hat\vf(x) - S_k \hat\vf(y)} \, .
\]

Recall that since we cut at $f^{-1}(z)$ and $f^{-1}(z \pm \ve)$, during our construction of $\hat I_{z, \ve_0, \ve}$, 
$\hat H_\ve(Z_j)$ must satisfy either $\pi(\hat H_\ve(Z_j)) = (z, z+\ve)$ or $\pi(\hat H_\ve(Z_j)) = (z-\ve, z)$.
Let us denote these intervals above half the hole by $\hat H_\ve(Z_j)^+$ or $\hat H_\ve(Z_j)^-$,
accordingly.  Since $\hat f^p$ is continuous and injective on $\hat H_\ve(Z_j)$ by (iv), 
$\hatf^p(\hat H_\ve(Z_j) \cap \hatf^{-p}(\hat H_\ve))$ contains a full interval in the fibre above half the hole 
(possibly different from $\hat H_\ve(Z_j)$), which we can
also denote by $+$ or $-$ as appropriate.
Note that the conformal measure of all the lifts of the right half hole $(z, z+\ve)$ have the same measure, and so do all the lifts
of the left half hole.

We proceed to prove item (b) of the lemma first.
If $f^p$ is orientation preserving at $z$, then using conformality and bounded distortion, we have on 
either half of the hole,
\begin{equation}
\label{eq:conf hole}
\hat m_\ve (  \hat H_\ve(Z_j)^\pm \cap \hatf^{-p}(\hat H_\ve)) = P_\ve(S_p \vf(H_\ve(z)))^{\pm 1} e^{S_p\vf(z)} \hat m_\ve( \hat H_\ve(Z_j)^\pm) \, ,
\end{equation}
where we have used the fact that $S_p\hat\vf(z) = S_p \vf(z)$.  On the other hand, if $f^p$ is orientation reversing at $z$,
then we are left with, for example, the right half hole mapping onto the left half,
\[
\hat m_\ve (  \hat H_\ve(Z_j)^+ \cap \hatf^{-p}(\hat H_\ve)) = P_\ve(S_p \vf(H_\ve(z)))^{\pm 1} e^{S_p\vf(z)} \hat m_\ve( \hat H_\ve(Z_j)^-) \, ,
\]
and so to conclude the desired cancellation in \eqref{eq:ratio}, we use the assumption 
$\lim_{\ve \to 0} \frac{m_\vf(z, z+\ve)}{m_\vf(z-\ve,z)} = 1$.  

Thus under either alternative in item (b), we combine the estimates in \eqref{eq:ratio} to write,
\[
\begin{split}
&\sum_{\substack{Y_i \subset \hat H_\ve'(*, N_0) \\ \ell_i \ge 2}} \hat\mu_\ve(Y_i)
 =\sum_{Z_j} \hat\mu_\ve(Z_j) e^{\pm 2C_d \diam(Z_j)^\eta}    P_\ve(S_k\hat \vf(Z_j))^{\pm 2}  P_\ve(S_p \vf(H_\ve(z)))^{\pm 1}   e^{S_p\vf(z)} \\
& = e^{\pm 2C_d \max_j \diam(Z_j)^\eta} P_\ve(S_k\hat \vf(Z_j))^{\pm 2}  P_\ve(S_p \vf(H_\ve(z)))^{\pm 1}  e^{S_p\vf(z)} \hat\mu_\ve(H_\ve') + 
\mathcal{O}\!\left( e^{-\alpha N_0} \hat \mu_\ve(\hat H_\ve) \right) \, .
\end{split}
\]
Analogous estimates follow for each $\ell \ge 3$.  Then using that $e^{S_{\ell p}\vf(z)} = e^{\ell S_p(z)}$, we
estimate \eqref{eq:expand} 
\begin{equation}
\label{eq:ell sum}
\begin{split}
\hat\mu_\ve(\hat H_\ve) = & \sum_{\ell = 1}^{N_0} e^{\pm 2C_d \max_j \diam(Z_j)^\eta} P_\ve(S_k\hat \vf(Z_j))^{\pm 2}  
P_\ve(S_{\ell p} \vf(H_\ve(z)))^{\pm 1}  e^{\ell S_p\vf(z)} \hat\mu_\ve(H_\ve') \\ 
& +  
\mathcal{O}\!\left( N_0 e^{-\alpha N_0} \hat \mu_\ve(\hat H_\ve) \right) \, .
\end{split}
\end{equation}
Since $N_0$ is fixed, $Z_j \subset Y$ and the first entry of $Z_j$ to $\hat H_\ve$ occurs before time $N_0$, 
we have $\max_j \diam(Z_j) \to 0$ as $\ve \to 0$. 
In addition, both $P_\ve(S_k\hat \vf(Z_j))$ and $P_\ve(S_{\ell p} \vf(H_\ve(z)))$
approach 1 as $\ve \to 0$ since the lengths of the orbit segments are uniformly bounded by $pN_0$ and $\hat\vf$ is 
continuous along each orbit segment. 
Dividing by $\hat\mu_\ve(H_\ve)$ and taking the limit $\ve \to 0$ yields for each $N_0 > 0$, 
\begin{equation}
\label{eq:partial sum}
1 = \sum_{\ell = 1}^{N_0} e^{\ell S_p\vf(z)} \lim_{\ve \to 0} \frac{\hat \mu_\ve(\hat H_\ve')}{\hat \mu_\ve(\hat H_\ve)} + 
\mathcal{O}\!\left( N_0 e^{-\alpha N_0}  \right) \, .
\end{equation}
Finally, taking $N_0 \to \infty$ proves item (b) of the lemma.

The proof of item (a) proceeds similarly starting from \eqref{eq:ratio}.  Now, however, since
$z$ is disjoint from the post-critical orbit, we may choose $\ve>0$ sufficiently small that $f^n(c) \notin H_\ve$ for
all $n \le pN_0$ and $c \in \Crit$.  Then
the interval
$Z_j$ from (iii) can be chosen so that $\pi(\hat f^{k(Z_j)}(Z_j)) = (z-\ve, z+\ve)$, i.e. $\hat f^{k(Z_j)}(Z_j)$ covers
a level of the fibre above the full hole.  Thus we may combine the left and right halves of the hole to obtain
the analogue of \eqref{eq:conf hole} in this case,
\begin{equation}
\label{eq:whole hole}
\hat m_\ve (  \hat H_\ve(Z_j) \cap \hatf^{-p}(\hat H_\ve)) = P_\ve(S_p \vf(H_\ve(z)))^{\pm 1} e^{S_p\vf(z)} \hat m_\ve( \hat H_\ve(Z_j)) \, ,
\end{equation}
and the orientation preseving character of $f^p$ at $z$ is irrelevant.  The proof of item (a) of the lemma is then
complete, following \eqref{eq:ell sum} and \eqref{eq:partial sum} precisely as written.
\end{proof}

Now Lemmas~\ref{lem:holder aperiodic} and \ref{lem:holder per}, together with Theorem~\ref{thm:induced limit}
and Proposition~\ref{prop:ratio}, complete the proof of Theorem~\ref{thm:zerohole_Holder}, via \eqref{eq:string}.


\subsection{Final step of the proof of Theorem~\ref{thm:zerohole_geom}: the geometric case } 
\label{ssec:geo limit}

In this section, we prove the third limit in \eqref{eq:three amigos} in the case when $\vf = -t \log |Df| - p_t$,
$t \in (t^-, t_1)$, where $t_1$ is defined by \eqref{eq:t1}.  We assume the slow approach condition
\eqref{eq:slow approach} as well as the polynomial growth condition on the derivative along the post-critical 
orbit \eqref{eq:Dn}, formulated in Section~\ref{sssec:geo}.

We first prove an analogue of Lemma~\ref{lem: low dens Hold} in this case.

\begin{lemma}
If $f\in \F$, $\vf=-t\log|Df|-p_t$ and $z$ satisfies \eqref{eq:slow approach} with $t\in (t^-, t_1)$, then 
there exists $\zeta > 0$ such that for all $\ve >0$ sufficiently small, 
$\inf_{x \in H_\ve(z)} \frac{d\mu_\vf}{dm_\vf}(x) \ge \zeta$, where $\mu_\vf$ and $m_\vf$ are the relevant invariant and conformal measures.
\label{lem: low dens geom}
\end{lemma}

\begin{proof}
The proof is nearly identical to that of Lemma~\ref{lem: low dens Hold}:
While for the geometric potential with $t<0$ it may be that 
$\inf_{x \in I} S_n\vf(x)=-\infty$, the slow approach condition \eqref{eq:slow approach} ensures that 
for $x \in H_\ve$, there is a finite lower bound on $S_n\vf(x)$ that is uniform in $\ve$,
since $n$ is fixed and independent of $\ve$ in \eqref{eq:lower g}.
\end{proof}

In order to prove the required convergence for geometric potentials, we will use the setup and notation of \cite{Henk}.  
It follows from \eqref{eq:Dn} and our choice of $\gamma_n$, that
\[
\sum_n \gamma_n < \infty \quad \mbox{ and } \quad 
\sum_{n \ge 1} \big(\gamma_n^{\order-1} D_n(c)\big)^{-1/\order} < \infty \quad \mbox{for all } c \in \Crit \, .
\]
This is precisely the condition\footnote{Indeed, this condition is equivalent to the simpler condition,
$\sum_{n \ge 1} D_n(c)^{-1/(2\order-1)} < \infty$ \cite[Lemma~2.1]{Henk}, but we use the formulation above
in order to directly apply the results of \cite{Henk}.  Our condition \eqref{eq:Dn} is slightly stronger and generalises the 
exponent to values of $s_t < 1$.}  required of $f$ in \cite{Henk}.

We will not need the full strength of the results from \cite{Henk}; rather, we will use the estimates on the recovery times
for expansion for orbits that pass close to the set $\Crit$.  To this end, for a value of $\delta>0$ to be specified later, we define 
$B_\delta(\Crit) = \cup_{c \in \Crit} (c-\delta, c+\delta)$, for $\delta > 0$.  
A key estimate of \cite{Henk}
is the following.

\begin{lemma}
\cite[Lemma~2.4]{Henk}
\label{lem:Henk}
For $\delta>0$ sufficiently small, there exist constants $C_\delta, \beta_\delta >0$ such that every orbit segment $\{ f^i(x) \}_{i=0}^{k-1}$ such that
$\{ f^i(x) \}_{i=0}^{k-1} \cap B_\delta(\Crit) = \emptyset$, we have
\[
|Df^k(x)| \ge C_\delta e^{\beta_\delta k} \, .
\]
If, in addtion, $f^k(x) \in B_\delta(\Crit)$, then there exists $\kappa>0$ independent of $\delta$ such that,
\[
|Df^k(x)| \ge \max \{ \kappa, C_\delta e^{\beta_\delta k} \} \, .
\]
\end{lemma}

Next, we define the notion of binding period recalling the sequence $(\gamma_n)_{n \in \mathbb{N}}$
from \eqref{eq:Dn}.  If $x \in B_\delta(\Crit)$, then
\[
b(x) = \max \{ b \in \mathbb{N} : |f^k(x) - f^k(c)| \le \gamma_k|f^k(c) - \Crit| \quad \forall \, k \le b-1 \} \, ,
\]
and $b(x) = 0$ if $x \notin B_\delta(\Crit)$.  Let $I_b = \{ x \in I : b(x) = b \}$ denote the level sets of $b$.
The binding period will be useful in estimating the following important quantity, defined for each $c \in \Crit$,
\[
Df^b_{\min}(c) := \min \{ |Df^b(x)| : x \in I_b \cap B_\delta(c) \},
\]
which governs the minimum rate of growth in expansion along orbit segments.

Note that $b_\delta = \min \{ b(x) : x \in B_\delta(\Crit)  \}$ tends to $\infty$ as $\delta \to 0$.
This fact is used in \cite{Henk} to make $Df^{b_\delta}_{\min}(c)$ arbitrarily large by choosing $\delta>0$ sufficiently small.

Each orbit of length $n$ is assigned an itinerary $(\nu_1, b_1), (\nu_2, b_2), \ldots, (\nu_k, b_k)$,
where each $\nu_i = \nu_i(x)$ represents the first time larger than $\nu_{i-1} + b_{i-1}$ such that the orbit of $x \in I$ makes a return to $B_\delta(\Crit)$.
Each return $i$ is called a deep return and placed in a set $S_d = S_d(x)$ if the orbit enters $B_\delta(\Crit)$ at time $\nu_i$; it
is called a shallow return and placed in a set $S_s(x)$ if the orbit does not enter $B_\delta(\Crit)$ at 
time $\nu_i$, but is part of a
dynamically defined interval that intersects $B_\delta(\Crit)$.

The key estimates from \cite{Henk} using the information from binding periods are as follows.

\begin{lemma} \cite[Lemmas~2.5 and 3.2]{Henk}
\label{lem:bound growth}
\begin{itemize}
  \item[a)]  There exists $C_0 >0$ independent of $\delta >0$ such that for all $c \in \Crit$ and 
  $b \ge b_\delta$ with $I_b \neq \emptyset$,
\[
Df^b_{\min}(c) \ge C_0 (\gamma_b^{\order -1}D_b(c))^{1/\order} \, .
\]
  \item[b)]There exist $K_0 > 0$ and $\rho \in (0,1)$, independent of $\delta$, such that for an orbit $\{ f^i(x) \}_{i=0}^{n-1}$
  with a given sequence $(\nu_1,b_1), \ldots, (\nu_k,b_k)$ at time $n \ge \nu_k + b_k$, we have
  \[
  |Df^n(x)| \ge  \max \left\{ C_\delta^{\#S_d} e^{\beta_\delta (n - \sum_{i=0}^k b_i)},
   \Big(\frac{\kappa}{K_0}\Big)^{\#S_d} \rho^{-\# S_s} \right\} \prod_{i\in S_d} Df^{b_i}_{\min}(c_i) \, ,
  \]
where $c_i$ is the critical point associated to the return at time $\nu_i$.
\end{itemize}
\end{lemma}

With these key estimates recalled, we are ready to begin our proofs of the relevant limits.  As in 
Section~\ref{sec:holder}, we begin with the aperiodic case.

\begin{lemma}
\label{lem:t aperiodic}
Suppose $f \in \mathcal{F}_\order$ and $\phi = -t \log |Df|$ for $t \in (t^-, t_1)$ satisfies \eqref{eq:Dn}.
Let $z$ be an aperiodic point for $f$ satisfying \eqref{eq:slow approach}. 
Then
\[
\lim_{\ve \to 0} \frac{\hat \mu_\ve(\hat H_\ve')}{\hat \mu_\ve(\hat H_\ve)} = 1 \, .
\]
\end{lemma}

\begin{proof}
We will follow the strategy of the proof of Lemma~\ref{lem:holder aperiodic}, using the same
notation defined there.  Following \eqref{eq:unwrap}, we must show as before that
\begin{equation}
\label{eq:o}
\sum_{k \ge 2} \sum_{Y_i \subset \hat H'_\ve(k)} (k-1) \hat \mu_\ve(Y_i) = o(\hat \mu_\ve(\hat H_\ve) ) \, .
\end{equation}
However, the estimate in this case is not so simple since the analogous expression to
\eqref{eq:entry} does not enjoy uniform exponential contraction in $T_i$.  Rather, we split 
$Y_i \subset \hat H'_\ve(k)$ into those cylinders which are `bound' (i.e., in the midst of a binding period)
at the time $T_i$ of their $k^{th}$ 
entry to $\hat H_\ve$, and those cylinders which are not bound, which we call `free.'

As before, we fix $N \in \mathbb{N}$ and choose $\ve$ sufficiently small that
$\hat f^j(\hat H_\ve) \cap \hat H_\ve = \emptyset$, for all $j < N$.

\medskip
{\em Estimate on free pieces.}
To estimate the contribution to \eqref{eq:o} from cylinders that are free at time $T_i$, we begin
as in \eqref{eq:entry},
\begin{equation}
\label{eq:conformal entry}
\hat \mu_\ve(Y_i) = C^{\pm 1} \hat m_\ve(Y_i) = C^{\pm 1} e^{S_{T_i}\hat \vf(y_i)} \hat m_\ve(\hat f^{T_i}Y_i)
= C^{\pm 1} |Df^{T_i}(y_i)|^{-t} e^{- p_t T_i} \hat m_\ve(\hat f^{T_i}Y_i) \, .
\end{equation}

We estimate the above expression differently depending whether $t<1$ or $t \ge 1$.
In all cases, we fix $\delta>0$ sufficiently small that $Df^{b_\delta}_{\min}(c_i) \ge 2K_0/\kappa$.

If $t<1$, then we consider the following two cases, depending on the itinerary
$(\nu_1, b_1), \ldots, (\nu_{k_i}, b_{k_i})$ associated to $y_i$ from time 0 until time $T_i$.  
Since $\hat f^{T_i}(Y_i)$ is free, we have $T_i \ge \nu_{k_i} + b_{k_i}$ so we may apply
Lemma~\ref{lem:bound growth}.
Choose $\epsilon > 0$ sufficiently small that $\epsilon \le \min \{ \frac 14, \frac{-\beta_\delta b_\delta}{4 \log C_\delta} \}$.

{\em Case 1. $\sum_{j=0}^{k_i} b_j > \epsilon T_i$.}
Using the second estimate in Lemma~\ref{lem:bound growth}(b), we have by choice of $\delta$,
\[
|D\hat f^{T_i}(y_i)| \ge 2^{\# S_d} = 2^{k_i} \, .
\]

{\em Case 2. $\sum_{j = 0}^{k_i} b_j < \epsilon T_i$.} 
It follows that $\# S_d(y_i) = k_i \le \epsilon T_i/b_\delta$.  Thus using the first estimate in 
Lemma~\ref{lem:bound growth}(b), we have by our choice of $\epsilon$,
\[
|D\hat f^{T_i}(y_i)| \ge C_\delta^{\epsilon T_i/b_\delta} e^{\beta_\delta T_i(1-\epsilon)} \ge e^{\beta_\delta T_i/2} \, .
\]

In either case, our estimate in \eqref{eq:conformal} for $t<1$ becomes,
\[
\hat m(Y_i) \le C e^{-b_t T_i} \hat m(\hat f^{T_i}Y_i) \, .
\]

In the statement of Theorem~\ref{thm:zerohole_geom} we only consider the case $t \ge 1$ under a (CE) condition along the critical orbits:
\[
\exists C, \gamma> 0 \mbox{ s.t. } D_n(c) \ge C e^{\gamma n} \, .
\]
In this case, it suffices that $\gamma_n \ge e^{-\gamma n/(2(\order-1))}$, so that
$Df^b_{\min}(c) \ge C_0 e^{\gamma b/(2\order)}$ by Lemma~\ref{lem:bound growth}(a).
For this range of $t \ge 1$, we consider two slightly different cases.
Using the same choice of $\delta$ as above, we choose
$\epsilon = - \bar\beta/ \log (C_\delta C_0)$, where 
$\bar\beta = \frac 12 \min \{ \beta_\delta, \gamma/(2 \order) \}$.

{\em Case 1. $k_i = \# S_d(y_i) \ge \epsilon T_i$.}  Using the second estimate in
Lemma~\ref{lem:bound growth}(b) and our choice of $\delta$,
\[
|D\hat f^{T_i}(y_i)| \ge \Big( \frac{\kappa}{K_0} \Big)^{k_i} \prod_{j \in S_d(y_i)} Df^{b_j}_{\min}(c_j)
\ge 2^{\epsilon T_i} \, .
\]

{\em Case 2. $k_i \le \epsilon T_i$.}  Using the first estimate in Lemma~\ref{lem:bound growth}(b),
we have using our choice of $\epsilon$,
\[
|D\hat f^{T_i}(y_i)| \ge C_\delta^{\epsilon T_i} e^{\beta_\delta(T_i - \sum_{j=0}^{k_i} b_j)} C_0^{\epsilon T_i} 
e^{\gamma \sum_{j=0}^{k_i} b_j/(2\order)} \ge e^{\bar \beta T_i} \, .
\] 

In either case, our estimate in \eqref{eq:conformal} for $t \ge 1$ becomes,
\[
\hat m(Y_i) \le C e^{-(t\beta_1 + p_t) T_i} \hat m(\hat f^{T_i}Y_i) \, ,
\]
where $e^{\beta_1} = \min \{ e^{\bar\beta}, 2^\epsilon \}$.

To unify notation, set $\hat\alpha = p_t$ when $t<1$ in all cases,  and $\hat\alpha = t\beta_1 + p_t$ in the (CE) case when $t \ge 1$.  Recall that we defined $t_1=1$ in the non-(CE) case; in the (CE) case set
\begin{equation}
\label{eq:t1}
t_1 := \sup \{t  \in (1, t^+) : t \beta_1 + p_t > 0 \}  \, ,
\end{equation}
noting that since $p_1 = 0$, such a $t_1 > 1$ exists by continuity of $p_t$.

Now the above estimates in conjunction with the complexity estimate \eqref{eq:card} yield by
\eqref{eq:error bound},
\begin{equation}
\label{eq:error bound free}
\sum_{k \ge 2} \sum_{\substack{Y_i \subset \hat H_\ve'(k) \\ Y_i \, \mbox{\scriptsize free}}} (k-1) \hat \mu_\ve(Y_i) 
\le \sum_{k \ge 2} \sum_{j \ge N(k-1)} C e^{-(\hat\alpha - \xi) j} (k-1) (j+L)^3 m_\vf(H_\ve) \\
\le C' \hat \mu_\ve(\hat H_\ve) e^{-(\hat\alpha - \xi)N} \, ,
\end{equation}
where we may choose $L$ sufficiently large that $\xi < \hat\alpha$, and in the second inequality
we have used Lemma~\ref{lem: low dens geom} and the fact that $\hat\mu_\ve(\hat H_\ve) = \mu_\vf(H_\ve)$.

\medskip
{\em Estimate on bound pieces.}  Next we estimate the contribution to \eqref{eq:o} from cylinders
$Y_i$ which are undergoing a bound period at time $T_i$.  Let $\nu_i$ denote the time that
$Y_i$ enters this bound period.  By assumption, $\nu_i \le T_i < \nu_i + b_i$.  
Let $x \in \hat f^{\nu_i}(Y_i) \subset B_\delta(c)$.  Then using the slow approach condition
\eqref{eq:slow approach} and the definition of $b$,
\[
\delta_z \gamma_{T_i - \nu_i}^{1-\theta} \le |\hat f^{T_i-\nu_i}(c) - z|
\le |\hat f^{T_i-\nu_i}(c) - \hat f^{T_i-\nu_i}(x)| + |\hat f^{T_i-\nu_i}(x) - z| 
\le \gamma_{T_i-\nu_i} + \tfrac 12 |H_\ve(z)| \, .
\]
This implies that 
\[
\delta_z \le 2 \max \{ \gamma_{T_i - \nu_i}^\theta, \tfrac 12 |H_\ve(z)| \gamma_{T_i-\nu_i}^{\theta - 1} \} \, .
\]
We consider the ways in which this can be satisfied.  First,
\[
\delta_z \le 2 \gamma_{T_i-\nu_i}^{\theta} \implies \gamma_{T_i - \nu_i} \ge (\delta_z/2)^{1/\theta} \, .
\]
Since $\gamma_n$ is summable, this condition can be satisfied by only finitely many values of
$T_i - \nu_i$, that depend only on $\gamma_n$, $\delta_z$ and $\theta$.  Indeed, we can render this
set empty since \eqref{eq:slow approach} implies $\{ f^n(c) \}_{n \ge 0} \cap \{ z \} = \emptyset$.
So by choosing $\ve, \delta >0$ sufficiently small, we can make $f^k(B_\delta(\Crit))$ disjoint from $H_\ve$
for these finitely many iterates.

Next, the second possibility is that 
\begin{equation}
\label{eq:gamma cond}
\delta_z \le |H_\ve(z)| \gamma_{T_i-\nu_i}^{\theta -1}
\implies \gamma_{T_i - \nu_i} \le \Big( \frac{|H_\ve(z)|}{\delta_z} \Big)^{1/(1-\theta)} \, .
\end{equation}
Recall from Section~\ref{sssec:geo} that we defined
$\gamma_n = n^{-r}$ for some $r > \frac{1}{s_t(1-\theta)}$. 
Then \eqref{eq:gamma cond} implies
\[
T_i - \nu_i \ge \Big( \frac{\delta_z}{|H_\ve(z)|} \Big)^{1/r(1-\theta)} \, .
\]
This implies that the return time to $Y$ for $Y_i$ satisfies 
\[
R_i \ge \max \{ (k-1)N, T_i \} \ge \tfrac 12 \left( (k-1)N + \Big( \frac{\delta_z}{|H_\ve(z)|} \Big)^{1/r(1-\theta)}
\right) =: \tau_\ve ,
\]
where the first condition comes from the fact that $Y_i \subset \hat H_\ve'(k)$ and $N$ comes from
the aperiodicity condition on $z$.  Thus using Theorem~\ref{thm:tails},
\begin{equation}
\label{eq:error bound bound}
\begin{split}
\sum_{k \ge 2} & \sum_{\substack{Y_i \subset \hat H_\ve'(k) \\ Y_i \, \mbox{\scriptsize bound}}} (k-1) \hat \mu_\ve(Y_i) 
\le \sum_{k \ge 2} \sum_{j \ge \tau_\ve} 
\sum_{R_i = j} (k-1) \hat \mu_\ve(Y_i) 
\le \sum_{k\ge 2} \sum_{j \ge \tau_\ve} C (k-1) e^{- \alpha j}
\\
& \le \sum_{k \ge 2} (k-1) C' e^{- \alpha (k-1)N/2} e^{- \frac \alpha 2 (\frac{\delta_z}{|H_\ve(z)|} )^{1/r(1-\theta)}} = o(e^{-\alpha N/2} |H_\ve(z)|^s) = o(\hat \mu_\ve(\hat H_\ve)) \, ,
\end{split}
\end{equation}
where $s>0$ represents any positive power, and the switch to $\hat \mu_\ve(H_\ve)$ is possible due to the
scaling exponent $s_t$ for the conformal measure $m_\vf$ as well as Lemma~\ref{lem: low dens geom}. 

Combining \eqref{eq:error bound free} and \eqref{eq:error bound bound} proves
\eqref{eq:o}, which by \eqref{eq:unwrap} completes the proof of the lemma.
\end{proof}

Next, we address the case when $z$ is periodic with prime period $p$.  
We continue to assume the slow approach condition \eqref{eq:slow approach}.

\begin{lemma}
\label{lem:t per}
Suppose $f \in \mathcal{F}_\order$ and $\phi = -t \log |Df|$ for $t \in (t^-, t_1)$ satisfies \eqref{eq:Dn}.
Let $z$ be a periodic point for $f$ of prime period $p$ satisfying \eqref{eq:slow approach}.
Then
\[
\lim_{\ve \to 0} \frac{\hat\mu_\ve(\hat H_\ve')}{\hat\mu_\ve(\hat H_\ve)} = 1 - e^{S_p\vf(z)} \, .
\]
\end{lemma}

\begin{proof}
We follow the proof of Lemma~\ref{lem:holder per}, which needs few modifications
now that we have recorded the relevant estimates over free and bound pieces.

Fix $N_0 > 0$ and choose
$\ve>0$ sufficiently small that properties (i)-(iv) enumerated at the start of the proof of 
Lemma~\ref{lem:holder per} hold.  We expand
$\hat \mu_\ve(\hat H_\ve)$ precisely as in \eqref{eq:period expand}. 
First, we must show that
the second and third sums in
that expression are the error terms in the expansion.

As in the proof of Lemma~\ref{lem:t aperiodic}, we call each $Y_i$ bound or free depending on
whether $\hat f^{T_i}(Y_i)$ is undergoing a bound period at time $T_i$ or not.  When
summing over the free pieces, \eqref{eq:error bound free} implies that both sums are of order
$\mathcal{O}(\hat \mu_\ve(\hat H_\ve) e^{-(\hat \alpha - \xi)N_0})$ since the entry time for each such $Y_i$
to $\hat H_\ve$ is greater than $N_0$.  Similarly, we estimate the second and third sums in
\eqref{eq:period expand} over bound pieces $Y_i$ using the slow approach condition
\eqref{eq:slow approach} so that by \eqref{eq:error bound bound}, these sums are 
$\mathcal{O}(\hat \mu_\ve(\hat H_\ve) e^{-\alpha N_0/2})$.  We thus arrive at equation
\eqref{eq:expand} as before.

Next, we derive \eqref{eq:ratio} as before
since that uses only property (iii) and the uniform $\log$-H\"older property of the invariant density
 $g_\ve$ (Lemma~\ref{lem:reg}); so we obtain the same expressions with the same definition of 
 $P_\ve(S_k\hat \vf(Z_j))$.
 
Since the slow approach condition \eqref{eq:slow approach} implies that $z$ is disjoint from the
post-critical orbit, we may choose
$\ve$ sufficiently small such that $f^k(c) \notin H_\ve(z)$ for $k \le pN_0$ 
and all $c \in \Crit$.
Thus we may follow the proof of the simpler item (a) of Lemma~\ref{lem:holder per},
without having to consider the left and right halves of the hole separately.
We use \eqref{eq:whole hole} to estimate the ratio in \eqref{eq:ratio} and so arrive at
\eqref{eq:ell sum} precisely as before.
  
Now, $\hat\vf = - t \log |Df| \circ \hat \pi - P(- t \log |Df|)$.  
Although $\hat\vf$ is not continuous on $\hat I_{z, \ve_0, \ve}$, it is still true on each $Y_i$ and
for each orbit segment of length at most $pN_0$, that $S_k\hat\vf$ is continuous with bounded 
ratio on $Y_i$
and each component of $\hat H_\ve'$ on level at most $pN_0$.  This follows since we have trimmed
$L$-cylinders in our construction of $Y = \hat I'(L)$.  This extends to $Z_j$ 
since 
$\hat f^{-pN_0}(\hat c) \cap Z_j = \emptyset$ for each $Z_j$ by choice of $\ve$, and so 
$P_\ve(S_k\hat \vf(Z_j)) \to 1$ as $\ve \to 0$.

We thus arrive at \eqref{eq:partial sum} with error term $\mathcal{O}(N_0 e^{-\tilde \alpha N_0})$ and 
$\tilde \alpha = \min \{ \hat \alpha - \xi, \alpha/2 \}$,
and taking $N_0 \to \infty$ completes the proof of the lemma. 
\end{proof}

Finally, Lemmas~\ref{lem:t aperiodic} and \ref{lem:t per} together with Theorem~\ref{thm:induced limit} and
Proposition~\ref{prop:ratio} complete the proof of Theorem~\ref{thm:zerohole_geom}, using \eqref{eq:string}.


\end{document}